\documentclass[a4paper,10pt]{article}

\usepackage{geometry}
\usepackage[english]{babel}

\usepackage{amsmath}
\usepackage{amsfonts}
\usepackage{amsthm}
\usepackage{amssymb}
\usepackage{ifthen}
\usepackage{color}

\usepackage{enumitem}

\numberwithin{equation}{section}

\newtheorem{theorem}{Theorem}[section]
\newtheorem*{thmACT}{Theorem \ref{thm:ACT2}}
\newtheorem{lemma}[theorem]{Lemma}
\newtheorem{proposition}[theorem]{Proposition}
\newtheorem{corollary}[theorem]{Corollary}
\newtheorem{definition}[theorem]{Definition}

\newtheoremstyle{rem}{}{}{}{0pt}{\bfseries}{.}{5pt}{}
\theoremstyle{rem}
\newtheorem*{remark}{Remark}

\newtheoremstyle{ack}{}{}{}{\parindent}{\itshape\bfseries}{.}{5pt}{}
\theoremstyle{ack}

\newtheoremstyle{ass}{}{}{\itshape}{0pt}{\bfseries}{:}{5pt}{}
\theoremstyle{ass}
\newtheorem{ass}{Assumption}

\newcommand\eps\varepsilon
\newcommand\flow\varphi

\newcommand\abs[1]{\left\lvert#1\right\rvert}
\newcommand\norm[1]{\left\lVert#1\right\rVert}

\newcommand\const[2]{#1^{\scriptscriptstyle (#2)}}

\newcommand\R{\mathbf{R}}

\newcommand\N{\mathbf{N}}
\newcommand\dx{\mathrm{d}}
\newcommand\sprod[1]{\left\langle#1\right\rangle}
\DeclareMathOperator{\Lipschitz}{Lip}
\newcommand\Lip{\Lipschitz}
\DeclareMathOperator{\diam}{diam}
\DeclareMathOperator{\id}{id}
\DeclareMathOperator{\graph}{graph}
\DeclareMathOperator{\Int}{Int}
\DeclareMathOperator{\vol}{vol}
\DeclareMathOperator{\dimension}{dim}
\DeclareMathOperator{\Emb}{Emb}
\DeclareMathOperator{\Diff}{Diff}
\DeclareMathOperator{\Span}{span}

\DeclareMathOperator{\determinante}{det}

\newcommand\uTS[1]{H_{#1}(\omega,x)}
\newcommand\sTS[1]{E_{#1}(\omega,x)}
\newcommand\Lnorm[2][\empty]{\ifthenelse{\equal{#1}{\empty}} {\left\lVert#2\right\rVert_{(\omega,x),n}} {\left\lVert#2\right\rVert_{(\omega,x),#1}} }
\newcommand\LnormAt[3][\empty]{\ifthenelse{\equal{#1}{\empty}} {\left\lVert#3\right\rVert_{(\omega,#2),n}} {\left\lVert#3\right\rVert_{(\omega,#2),#1}} }
\newcommand\Linner[3][\empty]{\ifthenelse{\equal{#1}{\empty}} {\left\langle #2, #3 \right\rangle_{(\omega,x),n}} {\left\langle #2, #3 \right\rangle_{(\omega,x),#1}}}

\newcommand\sball[3][\empty]{\ifthenelse{\equal{#1}{\empty}} {B_{#3}^s\left(#2\right)} {B_{#3}^s\left(#1,#2\right)} }
\newcommand\uball[3][\empty]{\ifthenelse{\equal{#1}{\empty}} {B_{#3}^u\left(#2\right)} {B_{#3}^u\left(#1,#2\right)} }
\newcommand\uballat[4][\empty]{\ifthenelse{\equal{#1}{\empty}} {B_{#4,#3}^u\left(#2\right)} {B_{#4,#3}^u\left(#1,#2\right)} }
\newcommand\sLball[3][\empty]{\ifthenelse{\equal{#1}{\empty}} {\tilde B_{#3}^s\left(#2\right)} {\tilde B_{#3}^s\left(#1,#2\right)} }
\newcommand\uLball[3][\empty]{\ifthenelse{\equal{#1}{\empty}} {\tilde B_{#3}^u\left(#2\right)} {\tilde B_{#3}^u\left(#1,#2\right)} }
\newcommand\ball[3][\empty]{\ifthenelse{\equal{#1}{\empty}} {B_{#3}\left(#2\right)} {B_{#3}(#1,#2)} }
\newcommand\Lball[3][\empty]{\ifthenelse{\equal{#1}{\empty}} {\tilde B_{#3}\left(#2\right)} {\tilde B_{#3}\left(#1,#2\right)} }
\newcommand\pLball[3][\empty]{\ifthenelse{\equal{#1}{\empty}} {\tilde U \left(#2,#3\right)} {\tilde U_{#1} \left(#2,#3\right)} }

\newcommand\prodm[1]{\nu^{\N} \times \mu \left(#1\right)}
\newcommand\collLSM[2]{\mathcal{F}_{\Delta^l_{\omega}}(#1,#2)}
\newcommand\localLSM[2]{\tilde\Delta_{\omega}^{l}(#1,#2)}

\newcommand\poincare[1]{P_{W^1,W^2}\left(#1\right)}

\newcommand\LSM

\usepackage[square,numbers]{natbib}

\newcommand\rref[1]{(\ref{#1})}

\title{Absolute Continuity Theorem \\ for Random Dynamical Systems on $\R^d$}
\author{Moritz Biskamp\footnote{Institut f\"ur Mathematik, MA 7-4, Technische Universit\"at Berlin, Stra{\ss}e des 17.\ Juni 136, D-10623 Berlin, biskamp@math.tu-berlin.de}}

\begin{document}

\maketitle

\begin{abstract}
In this article we provide a proof of the so called absolute continuity theorem for random dynamical systems on $\R^d$ which have an invariant probability measure. First we present the construction of local stable manifolds in this case. Then the absolute continuity theorem basically states that for any two transversal manifolds to the family of local stable manifolds the induced Lebesgue measures on these transversal manifolds are absolutely continuous under the map that transports every point on the first manifold along the local stable manifold to the second manifold, the so-called Poincar{\'e} map or holonomy map. In contrast to known results, we have to deal with the non-compactness of the state space and the randomness of the random dynamical system.
\end{abstract}

\vspace{1ex}

\vspace{2ex}

\section{Introduction}

The absolute continuity theorem is one of the main ingredients to prove Pesin's formula, which relates the entropy of a smooth dynamical system with its positive Lyapunov exponents. This remarkable formula was first established for deterministic dynamical systems on a compact Riemannian manifold preserving a smooth measure (see \citep{Pesin76}, \citep{Pesin77} and \citep{Pesin77b}). For random dynamical systems on $\R^d$ we will give a proof in \citep{Biskamp11}.

For deterministic dynamical systems preserving a smooth probability measure, which is roughly speaking the process generated by successiv applications of a diffeomorphism on some space or manifold, Pesin first proved general results concerning the existence of families of stable manifolds (see \citep{Pesin76}). Later, results were generalized to dynamical systems preserving only a Borel measure (see \citep{Ruelle79}, \citep{Fathi83}) and for dynamical systems with singularities (see \citep{Katok86}). In \citep{Barreira07} one finds a comprehensive and self-contained account on the theory dynamical systems with nonvanishing Lyapnov exponents, i.e. non-uniform hyperbolicity theory.

In this article we are interested in random dynamical systems, i.e. the evolution of the process generated by the successive application of {\it random} diffeomorphisms, which will be assumed to be chosen independently according to some probability measure on the set of diffeomorphisms. Since it is much too restrictive to assume invariance of some probability measure for each diffeomorphism, the notion of invariance was extended to random dynamical system in \citep{Kifer86}: a probability measure is said to be invariant for a random dynamical system if the average over all possible diffeomorphisms preserves the measure (see definition below). Then in \citep{Kifer86}, \citep{Ledrappier88} and \citep{Liu95}, Pesin's results were generalized to random dynamical system on compact Riemannian manifolds.

In this article and in \citep{Biskamp11} we will extend the results to random dynamical systems on the non-compact space $\R^d$. The main application we have in mind when we consider random dynamical systems on $\R^d$ are stochastic flows on $\R^d$ with stationary and independent increments. In \citep{Arnold95} it was proven that under some regularity assumptions there is a one to one relation between random dynamical systems and stochastic flows of a Kunita type (see \citep{Kunita90}). 

In the first part of the paper we will present the construction and the existence of local stable manifolds for random dynamical system on $\R^d$, which provide an invariant probability measure. This chapter follows very closely the general plan of \citep{Liu95}. Roughly speaking, the stable manifold at any point $x$ in space consists of those points which converge by application of the iterated functions with exponential speed to the iterated of $x$. One important construction within the proof is to define sets, nowadays called Pesin sets, which are chosen in such a way that one has uniform hyperbolicity on these sets (see Section \ref{sec:lyapunovmetric}), i.e. uniform bounds (in space {\it and} randomness) on the behaviour of the differential of the iterated maps (see Lemma \ref{lem:ExistenceOfl}).

If we consider a small region around some point $x$ in space and two manifolds, which are transversal (see Definition \ref{def:transversalMfld}) to each local manifold in this region, the absolute continuity theorem states that the induced Lebesgue measures on these manifolds are absolutely continuous under the map that transports every point on the first manifold along the local stable manifold to the second manifold. This is usually called Poincar{\'e} map or holonomy map. Even more we can show that the Jacobian of the Poincar{\'e} map is bounded away from $0$. The main conclusion that follows from the absolute continuity theorem is that the conditional measure with respect to the family of local stable manifolds of the volume on the state space is absolutely continuous (in fact, even equivalent) to the induced volume on the local stable manifolds (see \citep[Section 7]{Biskamp11}). This absolute continuity property was also first established by Pesin in his famous paper \citep{Pesin76} for deterministic dynamical system and extended to the random setting in \citep{Ledrappier88}. In this article we state the proof for the absolute continuity theorem, which will stick very closely to the proof of \citep{Katok86} which itself follows Pesin's original proof. The proof presented here is thus a detailed and complete proof in the case of random dynamical systems on $\R^d$ which posseses an invariant probability measure.

Finally let us emphasize that we obviously can not equip the space of two-times continuously differentiable diffeomorphisms on $\R^d$ with the uniform topology, as done in the case of a compact state space. Here we will use the topology induced by uniform convergence on compact sets (see \citep{Kunita90}). Clearly by this we lose the uniform bounds used in \citep{Liu95} to establish local stable manifolds (in particular Lemma \ref{lem:ExistenceOfr}). To replace these uniform bounds we need to assume certain integrability assumptions (see Section \ref{sec:rds}). In case of stochastic flows on $\R^d$ we will show in \citep{Biskamp11} that all these assumptions are satisfied for a broad class of stochastic flows.

The article is organized as follows. In Section \ref{sec:Preliminaries} we give an introduction to random dynamical systems and present the construction of local and global stable manidfolds. In Section \ref{sec:TheoremAbsoluteContinuity} we will state the absolute continuity theorem and prove it in Section \ref{sec:FinalProofACT}, whereas Section \ref{sec:ACT} is devoted to the preparation of the final proof.

\section{Preliminaries} \label{sec:Preliminaries} 

\subsection{Random Dynamical Systems} \label{sec:rds}

Let us abbreviate the set of twice continuously differentiable diffeomorphisms on $\R^d$ by $\Omega$. The topology on $\Omega$ is the one induced by uniform convergence on compact sets for all derivatives up to order 2 as described in \citep[Section 3.1]{Kunita90}. With this topology $\Omega$ becomes a separable Banach space. Let us fix a Borel probability measure $\nu$ on $(\Omega, \mathcal{B}(\Omega))$, where $\mathcal{B}(\Omega)$ denotes the Borel $\sigma$-algebra of $\Omega$.

We are interested in ergodic theory of the evolution process generated by successive applications of randomly chosen maps from $\Omega$. These maps will be assumed to be independent and identically distributed with law $\nu$. Thus let
\begin{align*}
\left(\Omega^\N, \mathcal{B}(\Omega)^\N, \nu^\N\right) = \prod_{i =0}^{+\infty} (\Omega, \mathcal{B}(\Omega),\nu)
\end{align*}
be the infinite product of copies of the measure space $(\Omega, \mathcal{B}(\Omega),\nu)$. Let us define for every $\omega = (f_0(\omega), f_1(\omega), \dots) \in \Omega^\N$ and $n \geq 0$
\begin{align*}
f^0_\omega = \id, \qquad f^n_\omega = f_{n-1}(\omega) \circ f_{n-2}(\omega) \circ \dots \circ f_0(\omega).
\end{align*}
The random dynamical system generated by these composed maps, i.e. $\{f^n_\omega : n \geq 0, \omega \in (\Omega, \mathcal{B}(\Omega),\nu)\}$ will be referred to as $\mathcal{X}^+(\R^d, \nu)$.

Let us further define the important space $\Omega^\N \times \R^d$ equipped with the  product $\sigma$-algebras  $\mathcal{B}(\Omega)^\N \times \mathcal{B}(\R^d)$. As already mentioned above $\Omega$ is a separable Banach space with the uniform topology on compact sets. Hence we have
\begin{align*}
\mathcal{B}(\Omega)^\N \times \mathcal{B}(\R^d) &= \mathcal{B}(\Omega^\N \times \R^d),
\end{align*}
Further let us denote by $\tau$ the left shift operator on $\Omega^\N$, namely
\begin{align*}
 f_n(\tau \omega) = f_{n+1}(\omega)
\end{align*}
for all $\omega = (f_0(\omega), f_1(\omega), \dots) \in \Omega^\N$ and $n \geq 0$. Finally let
\begin{align*}
F&: \Omega^\N \times \R^d \to \Omega^\N \times \R^d, &(\omega,x) \mapsto (\tau \omega, f_0(\omega) x).
\end{align*}
The system $(\Omega^\N \times \R^d, F)$ will be a link between the analysis of random dynamical systems and that of deterministic dynamical systems.

Now we will come to the notion of invariant measures of $\mathcal{X}^+(\R^d, \nu)$.

\begin{definition} \label{def:invariantMeasure}
A Borel probability measure $\mu$ on $\R^d$ is called an invariant measure of $\mathcal{X}^+(\R^d, \nu)$ if
\begin{align*}
\int_\Omega \mu(f^{-1}(\cdot)) \dx \nu(f) = \mu.
\end{align*}
\end{definition}

From now let us assume that there exists an invariant measure $\mu$ of $\mathcal{X}^+(\R^d, \nu)$ and let us denote the random dynamical system associated with $\mu$ by $\mathcal{X}^+(\R^d,\nu,\mu)$. From \citep[Lemma I.2.3]{Kifer86} we have the following Lemma, which relates the notion of invariance defined above with the invariance with respect to the skew product, i.e. the function $F$ on $\Omega^\N \times \R^d$.

\begin{lemma}
Let $\mu$ be a probability measure on $\R^d$. Then $\mu$ is an invariant measure of $\mathcal{X}^+(\R^d, \nu)$ (in the sense of Definition \ref{def:invariantMeasure}) if and only if $\nu^\N \times \mu$ is $F$-invariant, i.e. $(\nu^\N\times\mu) \circ F^{-1} = \nu^\N\times\mu$.
\end{lemma}

\begin{proof}
See \citep[Lemma I.2.3]{Kifer86}.
\end{proof}

Let us denote the tangent space at some point $y \in \R^d$ by $T_y\R^d$. Although this is quite unusual for systems on $\R^d$ we will stick to the notation from \citep{Liu95}. So let us define the following map, in differential geometry known as the exponential function, for $y \in \R^d$
\begin{align*}
\exp_y : \R^d \cong T_y \R^d \to \R^d, \quad x \mapsto \exp_y(x) := x + y,
\end{align*}
where $\cong$ means that the two spaces are isometrically isomorphic and thus can be identified. In the following we will use this often implicitely. Then we can define for $(\omega,x) \in \Omega^\N \times \R^d$ and $n \geq 0$ the map
\begin{align*}
F_{(\omega,x),n} : T_{f^{n}_\omega x} \R^d \to T_{f^{n+1}_\omega x}\R^d; \qquad
F_{(\omega,x),n} := \exp_{f^{n+1}_\omega x}^{-1} \circ f_n(\omega) \circ \exp_{f^{n}_\omega x},
\end{align*}
which is the evolution centered around the trajectory of $x$, i.e. $F_{(\omega,x),n}(0) = 0$ for all $n \geq 0$. Throughout this article we will assume that the random dynamical system $\mathcal{X}^+(\R^d, \nu, \mu)$ satisfies the following three integrability assumptions on $\nu$ and $\mu$:

\begin{ass} \label{ass1}
Let $\nu$ and $\mu$ satisfy
\begin{align*}
\log^+\abs{D_x f_0(\omega)} &\in \mathcal{L}^1(\nu^\N\times\mu),
\end{align*}
where $\abs{D_x f_0(\omega)}$ denotes the operator norm of the differential as a linear operator from $T_x \R^d$ to $T_{f_0(\omega) x} \R^d$ and $\log^+(a) = \max\{\log(a); 0\}$.
\end{ass}

\begin{ass}\label{ass2}
Let $\nu$ and $\mu$ satisfy
\begin{align*}
\log\left(\sup_{\xi \in B_x(0,1)}\abs{D^2_\xi F_{(\omega,x),0}}\right) &\in \mathcal{L}^1(\nu^\N\times\mu),\\
\log\left(\sup_{\xi \in B_x(0,1)}\abs{D^2_{F_{(\omega,x),0}(\xi)} F^{-1}_{(\omega,x),0}}\right) &\in \mathcal{L}^1(\nu^\N\times\mu),
\end{align*}
where $B_x(0,r)$ denotes the open ball in $T_x \R^d$ around the origin with radius $r >0$ and $D^2$ is the second derivative operator.
\end{ass}

\begin{ass} \label{ass1b}
Let $\nu$ and $\mu$ satisfy
\begin{align*}
\log \abs{D_0F_{(\omega,x),0}^{-1}} = \log\abs{D_{f_0(\omega)x} f_0(\omega)^{-1}} &\in \mathcal{L}^1(\nu^\N\times\mu).
\end{align*}
\end{ass}

Assumption \ref{ass1} is necessary for the application of the multiplicative ergodic theorem (see next section), whereas Assumption \ref{ass1b} is used in Lemma \ref{lem:DerivativeEstimate} to achieve an estimate on the derivative of the inverse. Assumption \ref{ass2} is used in Lemma \ref{lem:ExistenceOfr} to get a uniform bound on the Lipschitz-constant of the derivative and its inverse on some specific set $\Gamma_0 \subset \Omega^\N \times \R^d$. Let us remark that Assumption \ref{ass2} can be relaxed by taking not the unit ball in $T_x\R^d$ into consideration but some ball with positive radius.

\subsection{Multiplicative Ergodic Theorem and Lyapunov Exponents}

By Assumption \ref{ass1} in the previous section the multiplicative ergodic theorem yields the existence of linear subspaces with corresponding Lyapunov exponents, which play an extraordinary important role in the analysis of dynamical systems. The following theorem is \citep[Theorem I.3.2]{Liu95}.

\begin{theorem} \label{thm:met}
For the given system $\mathcal{X}^+(\R^d,\nu,\mu)$ there exists a Borel set $\Lambda_0 \subset \Omega^\N \times \R^d$ with $\nu^\N \times \mu (\Lambda_0) = 1$, $F \Lambda_0 \subset \Lambda_0$ such that:
\begin{enumerate}
\item For every $(\omega,x) \in \Lambda_0$ there exists a sequence of linear subspaces of $T_x\R^d$
\begin{align*}
\{0\} = V^{(0)}_{(\omega,x)} \subset V^{(1)}_{(\omega,x)} \subset \ldots \subset V^{(r(x))}_{(\omega,x)}
\end{align*}
and numbers (called Lyapunov exponents)
\begin{align*}
\lambda^{(1)}(x) < \lambda^{(2)}(x) < \ldots < \lambda^{(r(x))}(x)
\end{align*}
($\lambda^{(1)}(x)$ may be $-\infty$), which depend only on $x$, such that
\begin{align*}
\lim_{n\to +\infty} \frac{1}{n} \log \abs{D_x f^n_\omega \xi} = \lambda^{(i)}(x)
\end{align*}
for all $\xi \in V^{(i)}_{(\omega,x)} \setminus V^{(i-1)}_{(\omega,x)}$, $1 \leq i \leq r(x)$, and in addition
\begin{align*}
\lim_{n \to +\infty} \frac{1}{n} \log \abs{D_x f^n_\omega} &= \lambda^{(r(x))}(x)\\
\lim_{n \to +\infty} \frac{1}{n} \log \abs{\determinante (D_x f^n_\omega)} &= \sum_i \lambda^{(i)}(x) m_i(x)
\end{align*}
where $m_i(x) = \dimension(V^{(i)}_{(\omega,x)}) - \dimension(V^{(i-1)}_{(\omega,x)})$, which depends only on $x$ as well. Moreover, $r(x), \lambda^{(i)}(x)$ and $V^{(i)}_{(\omega,x)}$ depend measurably on $(\omega,x) \in \Lambda_0$ and
\begin{align*}
r(f_0(\omega)x) = r(x),\quad \lambda^{(i)}(f_0(\omega)x) = \lambda^{(i)}(x), \quad 
D_x f_0(\omega) V^{(i)}_{(\omega,x)} = V^{(i)}_{F(\omega,x)},
\end{align*}
for each $(\omega,x) \in \Lambda_0$, $1 \leq i \leq r(x)$.

\item For each $(\omega,x) \in \Lambda_0$, we introduce
\begin{align}
 \rho^{(1)}(x) \leq  \rho^{(2)}(x) \leq \ldots \leq \rho^{(d)}(x)
\end{align}
to denote $\lambda^{(1)}(x), \dots, \lambda^{(1)}(x), \dots, \lambda^{(i)}(x), \dots, \lambda^{(i)}(x), \dots \lambda^{(r(x))}(x), \dots, \lambda^{(r(x))}(x)$ with $\lambda^{(i)}(x)$ being repeated $m_i(x)$ times. Now, for $(\omega,x) \in \Lambda_0$, if $\{\xi_1, \dots, \xi_d\}$ is a basis of $T_x \R^d$ which satisfies
\begin{align*}
\lim_{n \to +\infty} \frac{1}{n} \log \abs{D_x f^n_\omega \xi_i} = \rho^{(i)}(x)
\end{align*}
for every $1 \leq i \leq d$, then for every two non-empty disjoint subsets $P,Q \subset \{1, \dots, d\}$ we have
\begin{align*}
\lim_{n \to +\infty} \frac{1}{n} \log \gamma (D_x f^n_\omega E_P, D_x f^n_\omega E_Q) = 0,
\end{align*}
where $E_P$ and $E_Q$ denote the subspaces of $T_x \R^d$ spanned by the vectors $\{\xi_i\}_{i \in P}$ and $\{\xi_j\}_{j \in Q}$ respectively and $\gamma(\cdot, \cdot)$ denotes the angle between the two associated subspaces.
\end{enumerate}
\end{theorem}

For more details on the multiplicative ergodic theorem for random dynamical systems and Lyapunov exponents see for example \citep{Arnold98} or \citep[Section I.3]{Liu95}. In the theorem the angle between to linear subspaces $E$ and $E'$ of a tangent space $T_x\R^d$ for some $x \in \R^d$ is defined by
\begin{align*}
\gamma(E,E') := \inf\left\{\cos^{-1}\left(\sprod{\xi,\xi'}\right) : \xi \in E, \xi' \in E', \abs{\xi} = \abs{\xi'} = 1\right\},
\end{align*}
where $\sprod{\cdot,\cdot}$ denotes the Euclidean scalar product on $T_x\R^d$.

\subsection{Lyapunov Metric and Pesin Sets} \label{sec:lyapunovmetric}

In this section we will mainly follow the book of Liu and Qian \citep[Chapter III]{Liu95}. In general proofs are only given, if there is a need to change arguments due to the non-compactness of $\R^d$ as the state space of the random dynamical system. Otherwise we will state the reference for the proof.

Let us define for some interval $[a,b], a < b \leq 0$, of the real line the set
\begin{align*}
\Lambda_{a,b} := \left\{ (\omega,x) \in \Lambda_0 : \lambda^i(x) \notin [a,b] \text{ for all } i\in 1,\dots,r(x) \right\}.
\end{align*}
Because of $F\Lambda_0 \subset \Lambda_0$ and the invariance of the Lyapunov exponents we have $F\Lambda_{a,b} \subset \Lambda_{a,b}$. For $(\omega,x) \in \Lambda_{a,b}$ and $n\geq 1$ define the following linear subspaces
\begin{align*}
\sTS{0} &:= \bigcup_{\lambda^{(i)}(x) < a} V^{(i)}_{(\omega,x)},\qquad & \uTS{0} &:= \sTS{0}^\bot,\\
\sTS{n} &:= D_xf^n_\omega \sTS{0}, \qquad & \uTS{n} &:= D_xf^n_\omega \uTS{0}.
\end{align*}
For $n,l \geq 1$ let us denote the iterated functions by
\begin{align*}
f^0_n(\omega) := \id, \qquad f^l_n(\omega) = f_{n+l-1}(\omega) \circ \dots \circ f_n(\omega).
\end{align*}
And for $n,l \geq 1$ we define the derivative of $f^l_n(\omega)$ at $f^n_\omega x$ by $T^l_n(\omega,x) := D_{f^n_\omega x}f^l_n(\omega)$ and its restriction to $\sTS{n}$ and $\uTS{n}$ respectively by
\begin{align*}
S^l_n(\omega,x) := T^l_n(\omega,x)|_{\sTS{n}}, \qquad U^l_n(\omega,x) := T^l_n(\omega,x)|_{\uTS{n}}.
\end{align*}

Let us now fix some $k\geq 1$ and $0 < \eps \leq \min\{1,(b-a)/(200d)\}$ and let us assume that the set
\begin{align*}
\Lambda_{a,b,k} := \{(\omega,x) \in \Lambda_{a,b} : \dimension \sTS{0} = k\}
\end{align*}
is non-empty. Then we have the following lemma from \citep[Lemma III.1.1]{Liu95}.

\begin{lemma} \label{lem:ExistenceOfl}
There exists a measurable function $l: \Lambda_{a,b,k} \times \N \to (0,+\infty)$ such that for each $(\omega,x) \in \Lambda_{a,b,k}$ and $n,l \geq 1$ we have
\begin{enumerate}
\item $\abs{S^l_n(\omega,x)\xi} \leq l(\omega,x,n) e^{(a+\eps)l}\abs{\xi}$, for all $\xi \in \sTS{n}$;
\item $\abs{U^l_n(\omega,x)\eta} \geq l(\omega,x,n)^{-1} e^{(b-\eps)l}\abs{\eta}$, for all $\eta \in \uTS{n}$;
\item $\gamma(\sTS{n+l},\uTS{n+l}) \geq l(\omega,x,n)^{-1}e^{-\eps l}$;
\item $l(\omega,x,n+l) \leq l(\omega,x,n) e^{\eps l}$,
\end{enumerate}
where $\gamma(\cdot,\cdot)$ again denotes the angle between two linear subspaces.
\end{lemma}

\begin{proof}
See \citep[Proof of Lemma III.1.1]{Liu95}.
\end{proof}

Let us fix a number $l' \geq 1$ such that the set
\begin{align*}
\Lambda^{l'}_{a,b,k,\eps}:= \left\{(\omega,x) \in \Lambda_{a,b,k} : l(\omega,x,0) \leq l'\right\}
\end{align*}
is non-empty. This family of sets, on which we have uniform bounds on the derivative by Lemma \ref{lem:ExistenceOfl}, is often called {\it Pesin sets}. We even can show some continuity of the subspaces $\sTS{0}$ and $\uTS{0}$ on these sets, which is \citep[Lemma III.1.2]{Liu95}.

\begin{lemma} \label{lem:ContinuousDependeceOfStableSpaces}
The linear subspaces $\sTS{0}$ and $\uTS{0}$ depend continuously on $(\omega,x) \in \Lambda^{l'}_{a,b,k,\eps}$.
\end{lemma}

\begin{proof}
Although this is \citep[Lemma III.1.2]{Liu95}, we will say a few words concerning the topology on $\Omega^\N$. As mentioned in the beginning of this section the topology on $\Omega = \Diff^2(\R^d)$ will be the one induced by uniform convergence on compact sets for all derivatives up to order 2 (see \citep[Chapter 4]{Kunita90}). Thus on $\Omega^\N$ we can use the usual topology of uniform convergence on finitely many elements. The space of all $k$-dimensional subspaces of $T_x\R^d \cong \R^d$ will be equipped with the Grasmannian metric, by which this space is compact.

Let $(\omega_n,x_n) \in \Lambda^{l'}_{a,b,k,\eps}$ be a sequence converging to $(\omega,x) \in \Lambda^{l'}_{a,b,k,\eps}$. By compactness of the Grassmanian there exists a subsequence of $\{(\omega_n,x_n)\}_n$ (denoted by the same symbols) such that $E_0(\omega_n,x_n)$ converges to some linear subspace $E$. Clearly $E$ is a subspace of $T_x\R^d$. For each $\zeta \in E$ there is a sequence $\xi_n \in E_0(\omega_n,x_n)$ such that $\abs{\zeta -\xi_n} \to 0$. Because for $n \in \N$ we have by Lemma \ref{lem:ExistenceOfl} that
\begin{align*}
\abs{T^l_0(\omega_n,x_n)\xi_n} = \abs{S^l_0(\omega_n,x_n)\xi_n} \leq l' e^{(a+\eps)l}\abs{\xi_n} \to l' e^{(a+\eps)l}\abs{\zeta}
\end{align*}
we only need to show that the left hand side converges to $\abs{T^l_0(\omega,x)\zeta}$. Since $\{\xi_n\}_{n\in\N} \cup \{\zeta\}$ is a compact set in $\R^d$ and the derivatives of each component of $\omega_n$ converge uniformly on compact sets we get for all $\zeta \in E$
\begin{align*}
\abs{T^l_0(\omega,x)\zeta} \leq l' e^{(a+\eps)l}\abs{\zeta}.
\end{align*}
Then Lemma \ref{lem:ExistenceOfl} implies that actually $\zeta \in E(\omega,x)$, which completes the proof.
\end{proof}

For $(\omega,x) \in \Lambda^{l'}_{a,b,k,\eps}$ and $n \in \N$ Lemma \ref{lem:ExistenceOfl} also allows us to define an inner product $\Linner{~}{~}$ on $T_{f^n_\omega x}\R^d$ such that
\begin{align*}
\Linner{\xi}{\xi'} &= \sum_{l=0}^{+\infty} e^{-2(a + 2\eps)l} \Big\langle S^l_n(\omega,x)\xi, S^l_n(\omega,x)\xi'\Big\rangle, && \text{for } \xi,\xi' \in \sTS{n}\\
\Linner{\eta}{\eta'} &= \sum_{l=0}^{n} e^{2(b - 2\eps)l} \left\langle \left[U^l_{n-l}(\omega,x)\right]^{-1}\eta, \left[U^l_{n-l}(\omega,x)\right]^{-1}\eta'\right\rangle, && \text{for } \eta,\eta' \in \uTS{n}.\\
\end{align*}
and $\sTS{n}$ and $\uTS{n}$ are orthogonal with respect to $\Linner{~}{~}$. Thus we can define the norms
\begin{align*}
\Lnorm{\xi} &:= \left[\Linner{\xi}{\xi}\right]^{\frac{1}{2}}& &\text{for } \xi \in \sTS{n};\\
\Lnorm{\eta} &:= \left[\Linner{\eta}{\eta}\right]^{\frac{1}{2}}& &\text{for } \eta \in \uTS{n};\\
\Lnorm{\zeta} &:= \max\left\{\Lnorm{\xi} ,\Lnorm{\eta}\right\}& &\text{for } \zeta = \xi + \eta \in \sTS{n} \oplus \uTS{n}.
\end{align*}

The sequence of norms $\{\Lnorm{\cdot}\}_{n\in \N}$ is usually called {\it Lyapunov metric} at the point $(\omega,x)$. By the definition of the inner product and by Lemma \ref{lem:ContinuousDependeceOfStableSpaces} the inner product $\Linner{~}{~}$ depends continuously on $(\omega,x) \in \Lambda^{l'}_{a,b,k,\eps}$. Now we can state \citep[Lemma III.1.3]{Liu95}.

\begin{lemma} \label{lem:EstimatesOnLnorm}
Let $(\omega,x) \in \Lambda^{l'}_{a,b,k,\eps}$. Then the Lyapunov metric at $(\omega,x)$ satisfies for each $n\in \N$
\begin{enumerate}
\item $\Lnorm[n+1]{S^1_{n}(\omega,x)\xi} \leq e^{a+2\eps} \Lnorm{\xi}\quad$ for $\xi \in \sTS{n}$;
\item $\Lnorm[n+1]{U^1_{n}(\omega,x)\eta} \geq e^{b-2\eps} \Lnorm{\eta}\quad$ for $\eta \in \uTS{n}$;
\item $\frac{1}{2} \abs{\zeta} \leq \Lnorm{\zeta} \leq Ae^{2\eps n}\abs{\zeta}\quad$ for $\zeta \in T_{f^n_\omega x}\R^d$, where $A = 4(l')^2(1-e^{-2\eps})^{-\frac{1}{2}}$.
\end{enumerate}
\end{lemma}

To the end of this section we will prove the following important lemma, which is basically \citep[Lemma III.1.4]{Liu95}. The proof is similar to the one of \citep[Lemma III.1.4]{Liu95} but has to be adapted to the situation of a non-compact state space, here Assumption \ref{ass2} plays an important role. We will use $\Lip(\cdot)$ to denote the Lipschitz constant of a function with respect to the norm $\abs{\cdot}$ if not mentioned otherwise.

\begin{lemma} \label{lem:ExistenceOfr}
 There exists a Borel set $\Gamma_0 \subset \Omega^\N \times \R^d$ and a measurable function $r: \Gamma_0 \to (0, \infty)$ such that $\nu^\N \times \mu (\Gamma_0) = 1$, $F\Gamma_0 \subset \Gamma_0$ and for all $(\omega, x) \in \Gamma_0$ 
\begin{enumerate}
 \item the map 
\begin{align*}
 F_{(\omega,x),0} := \exp_{f_0(\omega)x}^{-1} \circ f_0(\omega) \circ \exp_x : T_x\R^d \ni B_x(0,1) \to T_{f_0(\omega)x}\R^d,
\end{align*}
where $B_x(0,1)$ denotes the unit ball in $T_x\R^d$ around $0$, satisfies
\begin{align*}
 \Lip(D_{\cdot}F_{(\omega,x),0}) &\leq r(\omega,x),\\
 \Lip(D_{F_{(\omega,x),0}(\cdot)}F_{(\omega,x),0}^{-1}) &\leq r(\omega,x);
\end{align*}
\item $r(F^n(\omega,x)) = r(\tau^n\omega,f^n_\omega x) \leq r(\omega, x) e^{\eps n}$.
\end{enumerate}
\end{lemma}

\begin{proof}
Let us define the function $r': \Omega^\N \times \R^d$ by
\begin{align*}
r'(\omega,x) := &\max\left\{ \sup_{\xi \in B_x(0,1)} \abs{D^2_\xi F_{(\omega,x),0}}; \sup_{\xi \in B_{x}(0,1)}\abs{D^2_{F_{(\omega,x),0}(\xi)} F^{-1}_{(\omega,x),0}} \right\},
\end{align*}
where $D^2$ is the second derivative operator. Then by Assumption \ref{ass2} we have $\log(r') \in \mathcal{L}^1(\nu^{\N}\times\mu)$. According to Birkhoff's ergodic theorem there exists a measurable set $\Gamma_0 \subseteq \Omega^\N \times \R^d$ with $\nu^{\N}\times\mu(\Gamma_0) = 1$ and $F\Gamma_0 \subseteq \Gamma_0$ such that for all $(\omega,x) \in \Gamma_0$ we have
\begin{align*}
\lim_{n\to \infty} \frac{1}{n} \log\left(r'(F^n(\omega,x))\right) = 0.
\end{align*}
Thus it follows that
\begin{align*}
r(\omega,x) := \sup_{n \geq 0}\left\{r'(F^n(\omega,x)) e^{-\eps n} \right\}
\end{align*}
is finite at each point $(\omega,x) \in \Gamma_0$ and $r$ satisfies the requirements of the lemma by the mean value theorem.
\end{proof}

\subsection{Local Stable Manifolds}

Now chose a number $r' \geq 1$ such that the Borel set
\begin{align*}
\Lambda^{l',r'}_{a,b,k,\eps} := \left\{(\omega,x) \in \Lambda^{l'}_{a,b,k,\eps} \cap \Gamma_0 : r(\omega,x)\leq r'\right\}
\end{align*}
is non-empty. For ease of notation we will abbreviate $\Lambda' :=  \Lambda^{l',r'}_{a,b,k,\eps}$ for fixed parameters. Let us introduce the notion of {\it local stable manifolds} as in \citep[Section III.3]{Liu95}. By $\Emb^1(B^k,\R^d)$ we will denote the set of continuously differentiable embeddings from the open unit ball in $\R^k$, i.e. $B^k := := \{\xi \in \R^k: \abs{\xi} < 1\}$ into $\R^d$. The set of embeddings is equipped with the unifrom convergence on compact sets for all derivatives up to order one.

\begin{definition}
Let $X$ be a metric space and let $\{D_x\}_{x\in X}$ be a collection of subsets of $\R^d$. We call $\{D_x\}_{x \in X}$ a continuous family of $C^1$ embedded $k$-dimensional discs in $\R^d$ if there is a finite open cover $\{U_i\}_{i=1,\dots,l}$ of $X$ such that for each $U_i$ there exists a continuous map $\theta_i: U_i \to \Emb^1(B^k,\R^d)$ such that $\theta_i(x)B^k = D_x$, $x \in U_i$.
\end{definition}

Let us now state the main theorem on the existence of local stable manifolds \citep[Theorem III.3.1]{Liu95}.

\begin{theorem} \label{thm:localStableManifold}
For each $n \in \N$ there exists a continuous familiy of $C^1$ embedded $k$-di\-men\-sio\-nal discs $\{W_n(\omega,x)\}_{(\omega,x)\in \Lambda'}$ in $\R^d$  and there exist numbers $\alpha_n, \beta_n$ and $\gamma_n$ which depend only on $a, b, k, \eps, l'$ and $r'$ such that the following hold true for every $(\omega,x) \in \Lambda'$:
\begin{enumerate}
\item There exists a $C^{1,1}$ map
  \begin{align*}
  h_{(\omega,x),n} : O_n(\omega,x) \to H_n(\omega,x),
  \end{align*}
  where $O_n(\omega,x)$  is an open subset of $E_n(\omega,x)$ which contains $\{\xi \in E_n(\omega,x) : \abs{\xi} \leq \alpha_n\}$, such that
\begin{enumerate}
  \item $h_{(\omega,x),n} (0) = 0$;
  \item $\Lip(h_{(\omega,x),n}) \leq \beta_n, \Lip(D_\cdot h_{(\omega,x),n}) \leq \beta_n$;
  \item $W_n(\omega,x) = \exp_{f^n_\omega x} \graph(h_{(\omega,x),n})$ and $W_n(\omega,x)$ is tangent to $E_n(\omega,x)$ at the point $f^n_\omega x$;
\end{enumerate}
\item $f_n(\omega) W_n(\omega,x) \subseteq W_{n+1}(\omega,x)$
\item $d^s(f^l_n(\omega)y,f^l_n(\omega)z) \leq \gamma_n e^{(a+4\eps)l} d^s(y,z)$ for $y,z \in W_n(\omega,x)$, $l\in \N$, where $d^s(\cdot, \cdot)$ is the distance along $W_m(\omega,x)$ for $m \in \N$;
\item $\alpha_{n+1} = \alpha_n e^{-5\eps}, \beta_{n+1} = \beta_n e^{7\eps}$ and $\gamma_{n+1} = \gamma_n e^{2\eps}$.
\end{enumerate}
\end{theorem}

\begin{proof}
For the proof see \citep[Theorem III.3.1]{Liu95}. But let us emphasize that the following estimates are essential for the proof and that they are satisfied in our situation. Put
\begin{align} \label{eq:DefConstants}
 \eps_0 := e^{a + 4 \eps} - e^{a + 2\eps}, \quad c_0 :=4 A r' e^{2\eps}, \quad r_0 := c_0^{-1}\eps_0.
\end{align}
Then one can easily check by using the results from Section \ref{sec:lyapunovmetric} that for $l \geq 0$ the map
\begin{align*}
F_{(\omega,x),l} = \exp^{-1}_{f^{l+1}_\omega x} \circ f_l(\omega) \circ \exp_{f^l_\omega x} : \left\{\xi \in T_{f^l_\omega x}\R^d : \Lnorm[l]{\xi} \leq r_0 e^{-3\eps l}\right\} \to T_{f^{l+1}_\omega x}\R^d
\end{align*}
satisfies
\begin{align} \label{eq:EstimatesOfStabManfThm}
\Lip_{\norm{\cdot}}(D_\cdot F_{(\omega,x),l}) \leq c_0 e^{3\eps l} \qquad \text{and} \qquad
\Lip_{\norm{\cdot}}(F_{(\omega,x),l} - D_0 F_{(\omega,x),l}) \leq \eps_0,
\end{align}
where $\Lip_{\norm{\cdot}}$ denotes the Lipschitz-constant with respect to $\Lnorm[l]{\cdot}$ and $\Lnorm[l+1]{\cdot}$. Furthermore if we define for $n, l \geq 0$
\begin{align*}
F^0_n(\omega,x) = \id,\qquad F^l_n(\omega,x) := F_{(\omega,x),n+l-1} \circ \dots \circ F_{(\omega,x),n}
\end{align*}
then for $(\xi_0,\eta_0) \in \exp^{-1}_x(W_0(\omega,x))$ with $\Lnorm[0]{(\xi_0,\eta_0)} \leq r_0$ we get for every $n \geq 0$ the estimate 
\begin{align} \label{eq:EstiamteOnStableManifold}
\Lnorm[n]{F^n_0(\omega,x)(\xi_0,\eta_0)} \leq \Lnorm[0]{(\xi_0,\eta_0)} e^{(a + 6\eps)n}.
\end{align}
\end{proof}

\subsection{Global Stable Manifolds}

Let us now show the existence of {\it global stable manifolds}. Denote
\begin{align*}
\hat\Lambda_0 := \Lambda_0 \cap \Gamma_0, \qquad \hat \Lambda_{a,b,k} := \Lambda_{a,b,k} \cap \hat\Lambda_0,
\end{align*}
where $\Lambda_0$ comes from Theorem \ref{thm:met} and $\Gamma_0$ from Lemma \ref{lem:ExistenceOfr}. Let $\{l'_m\}_{m \in \N}$ and $\{r'_m\}_{m \in\N}$ be a monotone sequence of positive numbers such that $l'_m \nearrow +\infty$ and $r'_m \nearrow +\infty$ as $m \to +\infty$. Then we have for all $m \in \N$
\begin{align*}
\Lambda_{a,b,k,\eps}^{l'_m,r'_m} \subset \Lambda_{a,b,k,\eps}^{l'_{m+1},r'_{m+1}}
\end{align*}
and
\begin{align*}
 \hat\Lambda_{a,b,k} = \bigcup_{m=1}^{+\infty}\Lambda_{a,b,k,\eps}^{l'_m,r'_m}.
\end{align*}
If we denote
\begin{align*}
\{[a_n,b_n]\}_{n\in\N} := \left\{[a,b]: a < b \leq 0,\text{ $a$ and $b$ are rational}\right\}
\end{align*}
let us define
\begin{align*}
\eps_n := \frac{1}{2}\min\left\{1,\frac{1}{(200d)}(b_n - a_n)\right\},
\end{align*}
then we have
\begin{align*}
 \hat\Lambda_0 = \left\{\bigcup_{n=1}^{+\infty}\bigcup_{k=1}^{d}\bigcup_{m=1}^{+\infty}\Lambda_{a_n,b_n,k,\eps_n}^{l'_m,r'_m}\right\} \cup \left\{(\omega,x) \in \hat\Lambda_0: \lambda^{(i)}(x) \geq 0, 1 \leq i \leq r(x)\right\}.
\end{align*}
Now we can state the following theorem, which is \citep[Theorem III.3.2]{Liu95} on the existence of global stable manifolds.

\begin{theorem} \label{thm:globalStableManifold}
Let $(\omega,x) \in \hat \Lambda_0\!\setminus\!\left\{(\omega,x) \in \hat\Lambda_0: \lambda^{(i)}(x) \geq 0, 1 \leq i \leq r(x)\right\}$ and let $\lambda^{(1)}(x) < \cdots < \lambda^{(p)}(x)$ be the strictly negative Lyapunov exponents at $(\omega,x)$. Define $W^{s,1}(\omega,x) \subset \cdots \subset W^{s,p}(\omega,x)$ by
\begin{align*}
W^{s,i}(\omega,x) := \left\{y \in \R^d : \limsup_{n\to\infty} \frac{1}{n} \log (\abs{f^n_\omega x - f^n_\omega y}) \leq \lambda^{(i)}(x)\right\}
\end{align*}
for $1 \leq i \leq p$. Then $W^{s,i}(\omega,x)$ is the image of $V^{(i)}_{(\omega,x)}$ under an injective immersion of class $C^{1,1}$ and is tangent to $V^{(i)}_{(\omega,x)}$ at $x$. In addition, if $y \in W^{s,i}(\omega,x)$ then
\begin{align*}
\limsup_{n\to\infty} \frac{1}{n} \log d^s(f^n_\omega x, f^n_\omega y) \leq \lambda^{(i)}(x)
\end{align*}
where $d^s(~,~)$ denotes the distance along the submanifold $f^n_\omega W^{s,i}(\omega,x)$.
\end{theorem}

\begin{proof}
See \citep[Theorem III.3.2]{Liu95}.
\end{proof}

\begin{definition}
For $(\omega,x)\in \Omega^\N\times\R^d$ the global stable manifold $W^s(\omega,x)$ is defined by
\begin{align*}
W^s(\omega,x) := \left\{y\in\R^d : \limsup_{n\to\infty}\frac{1}{n} \log (\abs{f^n_\omega x - f^n_\omega y}) < 0 \right\}.
\end{align*}
\end{definition}

Let $\Lambda' = \Lambda_{a,b,k,\eps}^{l',r'}$ be as considered for Theorem \ref{thm:localStableManifold}. For $(\omega,x) \in \Lambda'$ let $\lambda^{(1)}(x) < \dots < \lambda^{(i)}(x)$ be the Lyapunov exponents smaller than $a$. Then one can see that
\begin{align*}
 W^{s,i}(\omega,x) = \left\{y\in\R^d : \limsup_{n\to\infty}\frac{1}{n} \log (\abs{f^n_\omega x - f^n_\omega y}) \leq a \right\}.
\end{align*}
Thus if $(\omega,x) \in \hat \Lambda_0\!\setminus\!\left\{(\omega,x) \in \hat\Lambda_0: \lambda^{(i)}(x) \geq 0, 1 \leq i \leq r(x)\right\}$ and $\lambda^{(1)}(x) < \cdots < \lambda^{(p)}(x)$ are the strictly negative Lyapunov exponents at $(\omega,x)$ then we get
\begin{align*}
 W^s(\omega,x) = W^{s,p}(\omega,x)
\end{align*}
and hence $W^s(\omega,x)$ is the image of $V^{(p)}_{(\omega,x)}$ under an injective immersion of class $C^{1,1}$ and is tangent to $V^{(p)}_{(\omega,x)}$ at $x$.

\subsection{Another Estimate on the Derivative}
Before coming to the main theorem of this article we finally need to bound the derivative of the inverse of the function $F_{(\omega,x),n}$ at $0$.

\begin{lemma} \label{lem:DerivativeEstimate}
 There exists a set $\Gamma_1 \subset \Omega^\N \times \R^d$, with $F\Gamma_1 \subset \Gamma_1$ and $\nu^\N \times \mu(\Gamma_1) = 1$ such that for every $\delta \in (0,1)$, there exists a positive measurable function $C_\delta$ defined on $\Gamma_1$ such that for every $(\omega,x) \in \Gamma_1$ and $n \geq 0$ one has
\begin{align*}
\abs{D_0F^{-1}_{(\omega,x),n}} \leq C_\delta(\omega,x) e^{\delta n}.
\end{align*}
\end{lemma}

\begin{proof}
By Assumption \ref{ass1b} we have $\log\abs{D_0F^{-1}_{(\omega,x),0}} \in \mathcal{L}^1(\nu^\N\times\mu)$ and hence we get by Birkhoff's ergodic theorem the existence of a measurable set $\Gamma_1 \subset \Omega^\N \times \R^d$, which satisfies $F\Gamma_1 \subset \Gamma_1$ and $\nu^\N \times \mu(\Gamma_1) = 1$ such that for all $(\omega,x) \in \Gamma_1$
\begin{align*}
\frac{1}{n} \log \abs{D_0F^{-1}_{(\omega,x),n}} = \frac{1}{n} \log \abs{D_0F^{-1}_{F^n(\omega,x),0}} \to 0.
\end{align*}
Thus for all $\delta \in (0,1)$ we find a measurable function $C_\delta$ such that for all $n \geq 0$ and $(\omega,x) \in \Omega^{\N}\times\R^{d}$
\begin{align*}
\abs{D_0F^{-1}_{(\omega,x),n}} \leq C_\delta(\omega,x) e^{\delta n}.
\end{align*}
\end{proof}

Let us fix some $C' \geq 1$ such that the set
\begin{align*}
\Lambda_{a,b,k,\eps}^{l',r',C'} := \left\{(\omega,x) \in \Lambda_{a,b,k,\eps}^{l',r'} \cap \Gamma_1: C_\eps(\omega,x) \leq C' \right\}
\end{align*}
is non-empty.

\section{The Absolute Continuity Theorem} \label{sec:TheoremAbsoluteContinuity}

Let us abbreviate in the following
\begin{align*}
\Delta := \Lambda_{a,b,k,\eps}^{r',l',C'},
\end{align*}
where all parameters are chosen in such a way that $\Delta$ is non-empty. These parameters will be fixed from now on. Let us choose a sequence of approximating compact sets $\{\Delta^l\}_{l}$ with $\Delta^l \subset \Delta$ and $\Delta^l \subset \Delta^{l+1}$ such that $\prodm{\Delta\!\setminus\!\Delta^l} \to 0$ for $l \to \infty$ and let us fix arbitrarily such a set $\Delta^l$. For $(\omega,x) \in \Delta$ and $r >0$ define
\begin{align*}
\pLball[\Delta,\omega]{x}{r} := \exp_x \left(\left\{\zeta \in T_x\R^d : \Lnorm[0]{\zeta} < r\right\}\right)
\end{align*}
and if $(\omega,x) \in \Delta^l$ let
\begin{align*}
V_{\Delta^l}((\omega,x),r) := \left\{(\omega',x') \in \Delta^l : d(\omega,\omega') < r, x' \in \pLball[\Delta,\omega]{x}{r}\right\},
\end{align*}
where the distance $d$ in $\Omega^\N$ is as before the one induced by uniform convergence on compact sets for all derivatives up to order $2$. Let us denote the collection of local stable manifolds $\{W_0(\omega,x)\}_{(\omega,x)\in \Delta^l}$ which was constructed in Theorem \ref{thm:localStableManifold} in the following by $\{W_{loc}(\omega,x)\}_{(\omega,x)\in \Delta^l}$. Since by Theorem \ref{thm:localStableManifold} this is a continuous family of $C^1$ embedded $k$-dimensional discs and $\Delta^l$ is compact there exists uniformly on $\Delta^l$ a number $\delta_{\Delta^l} > 0$ such that for any $0 < q \leq \delta_{\Delta^l}$ and $(\omega',x')\in V_{\Delta^l}((\omega,x),q/2)$ the local stable manifold $W_{loc}(\omega',x')$ can be represented in local coordinates with respect to $(\omega,x)$, i.e. there exists a $C^1$ map
\begin{align*}
\phi : \left\{ \xi \in \sTS{0} : \Lnorm[0]{\xi} < q \right\} \to \uTS{0}
\end{align*}
with
\begin{align*}
\exp_x^{-1}\left(W_{loc}(\omega',x') \cap \pLball[\Delta,\omega]{x}{q}\right) = \graph(\phi).
\end{align*}
By choosing $\delta_{\Delta^l}$ even smaller we can ensure, that for $0 < q \leq \delta_{\Delta^l}$
\begin{align*}
\sup \left\{\Lnorm[0]{D_\xi \phi} : \xi \in \sTS{0}, \Lnorm[0]{\xi} < q  \right\} \leq\frac{1}{3}.
\end{align*}

For $(\omega,x) \in \Delta^l$ and $0 < q \leq \delta_{\Delta^l}$ we denote by $\Delta^l_\omega := \left\{x \in \R^d: (\omega,x) \in \Delta^l\right\}$ and by $\collLSM{x}{q}$ the collection of local stable submanifolds $W_{loc}(\omega,y)$ passing through $y \in \Delta^l_{\omega} \cap \pLball[\Delta,\omega]{x}{q/2}$. Set
\begin{align*}
&\localLSM{x}{q} := \bigcup_{y \in \Delta^l_{\omega} \cap \pLball[\Delta,\omega]{x}{q/2}} W_{loc}(\omega,y) \cap \pLball[\Delta,\omega]{x}{q}.
\end{align*}

Now let us introduce the notion of transversal manifolds to the collection of local stable manifolds.

\begin{definition} \label{def:transversalMfld}
A submanifold $W$ of $\R^d$ is called transversal to the family $\collLSM{x}{q}$ if the following hold true
\begin{enumerate}
\item $W \subset \pLball[\Delta,\omega]{x}{q}$ and $\exp_{x}^{-1}W$ is the graph of a $C^1$ map 
\begin{align*}
\psi: \left\{\eta \in \uTS{0} : \Lnorm[0]{\eta} < q\right\} \to \sTS{0};
\end{align*}
\item W intersects any $W_{loc}(\omega,y)$, $y \in \Delta^l_{\omega} \cap \pLball[\Delta,\omega]{x}{q/2}$, at exactly one point and this intersection is transversal, i.e. $T_{z}W \oplus T_{z}W_{loc}(\omega,y) = \R^d$ where $z = W \cap W_{loc}(\omega, y)$.
\end{enumerate}
\end{definition}

For a submanifold $W$ of $\R^d$ transversal to $\collLSM{x}{q}$ let
\begin{align*}
\norm{W} := \sup_\eta\Lnorm[0]{\psi(\eta)} + \sup_{\eta}\Lnorm[0]{D_\eta \psi}
\end{align*}
where the supremum is taken over $\{\eta \in \uTS{0}: \Lnorm[0]{\eta} < q\}$ and $\psi$ is the map representing $W$ as in the previous definition.

Now fix some $0 < q \leq \delta_{\Delta^l}$ and consider two submanifolds $W^1$ and $W^2$ transversal to $\collLSM{x}{q}$. By the choice of $\delta_{\Delta^l}$ each local stable manifold passing through $y \in \Delta^l_\omega \cap \pLball[\Delta,\omega]{x}{q/2}$ can be represented via some function $\phi$, whose norm of the derivative with respect to the Lyapunov metric is bounded by $1/3$. Thus the following map, which is usually called {\it Poincar\'e map} or {\it holonomy map}, is well defined. Let
\begin{align*}
P_{W^1,W^2}: W^1 \cap \localLSM{x}{q} \to W^2 \cap \localLSM{x}{q}
\end{align*}
be defined by
\begin{align*}
P_{W^1,W^2} : z = W^1\cap W_{loc}(\omega,y) \mapsto W^2\cap W_{loc}(\omega,y),
\end{align*}
for each $y \in \Delta^l_{\omega} \cap \pLball[\Delta,\omega]{x}{q/2}$. Since the collection of local stable manifolds is by Theorem \ref{thm:localStableManifold} a continuous family of $C^1$ embedded $k$-dimensional discs $P_{W^1,W^2}$ is a homeomorphism. To define, what is meant by absolute continuity of $\collLSM{x}{q}$, we will denote the Lebesgue measures on $W^i$ by $\lambda_{W^i}$ for $i = 1,2$.

\begin{definition}
The family $\collLSM{x}{q}$ is said to be absolutely continuous if  there exists a number $\eps_{\Delta^l_{\omega}}(x,q) > 0$ such that for any two submanifolds $W^1$ and $W^2$ transversal to $\collLSM{x}{q}$ and satisfying $\norm{W^{i}} \leq \eps_{\Delta^l_{\omega}}(x,q)$, $i = 1,2$, the Poincar\'e map $P_{W^1,W^2}$ constructed as above is absolutely continuous with respect to $\lambda_{W^1}$ and $\lambda_{W^2}$, i.e. $\lambda_{W^1} \approx \lambda_{W^2}\circ P_{W^1,W^2}$.
\end{definition}

Let us formulate the main theorem of this article, which is basically taken from \citep{Katok86}. As usual let us denote the Lebesgue measure on $\R^d$ by $\lambda$.

\begin{theorem} \label{thm:ACT2}
Let $\Delta^l$ be given as above. 
\begin{enumerate}
\item There exist numbers $0 < q_{\Delta^l} < \delta_{\Delta^l}/2$ and $\eps_{\Delta^l} > 0$ such that for every $(\omega,x) \in \Delta^l$ and $0 < q \leq q_{\Delta^l}$ the family $\collLSM{x}{q}$ is absolutely continuous with $\eps_{\Delta^l_{\omega}}(x,q) = \eps_{\Delta^l}$ uniformly on $\Delta^l$.

\item For every $C' \in (0,1)$ there exist numbers $0 < q_{\Delta^l}(C') < \delta_{\Delta^l}/2$ and $\eps_{\Delta^l}(C') > 0$ such that for each $(\omega,x) \in \Delta^l$ with $\lambda(\Delta_{\omega}^l) > 0$ and $x$ is a density point of $\Delta_{\omega}^l$ with respect to $\lambda$, and each two submanifolds $W^1$ and $W^2$ transversal to $\collLSM{x}{q_{\Delta^l}(C')}$ satisfying $\norm{W^{i}} \leq \eps_{\Delta^l}(C')$, $i = 1,2$, the Poincar\'e map $P_{W^1,W^2}$ is absolutely continuous and the Jacobian $J(P_{W^1,W^2})$ satisfies the inequality
\begin{align*}
\abs{J(P_{W^1,W^2})(y) - 1} \leq C'
\end{align*}
for $\lambda_{W^1}$-almost all $y \in W^1 \cap \localLSM{x}{q_{\Delta^l}(C')}$.
\end{enumerate}
\end{theorem}

\section{Preparations for the Proof of the Absolute Continuity Theorem} \label{sec:ACT}

Before presenting the formal proof of the absolute continuity theorem we will shortly outline the approach, which is based on the idea of Anosov and Sinai \cite{Anosov67} and follows the proof of \cite{Katok86} for deterministic dynamical systems on a compact manifold.

The basic idea is that for fixed $(\omega,x) \in \Delta^l$ and some proper $q_{\Delta^l}$ and sufficiently large $n$ we apply the mapping $f^n_\omega$ to the subsets $\localLSM{x}{q_{\Delta^l}} \cap W^i$, $i = 1,2$, of the transversal manifolds. Because of the contraction in the stable directions, which is stronger than the one in other directions, the set $f^n_\omega\left(\localLSM{x}{q_{\Delta^l}} \cap W^1\right)$ lies within an exponentially small distance of the set $f^n_\omega\left(\localLSM{x}{q_{\Delta^l}} \cap W^2\right)$. By this we are able compare the Lebesgue measures of these sets and show that their ratio is close to $1$ (this is basically Proposition \ref{prop:CompD1DbarD2}). Finally comparing the Lebesgue volume of the pullbacks of these sets under the mapping $\left(f^n_\omega\right)^{-1}$ (see Lemma \ref{lem:CompPullingBack}) we obtain the desired result. The main problem here is that although $W^i$, $i = 1,2$ is the graph of a $C^1$ function, this is in general not true for $f^n_\omega(W^i)$ for $n \geq 0$. Thus in the following sections we will construct a proper covering of $f^n_\omega(W^i)$, $i = 1,2$, which will provide a local representation by functions that itself and their derivative can be controlled.

\subsection{Preliminaries}

Fix once and for all $(\omega,x) \in \Delta^l$ and let $n \in \N$. Then we define the following balls in the stable respectively unstable tangent spaces with respect to the usual Euclidean norm and the Lyapunov norm. For both objects we will use the same symbols, but a $\sim$ above the symbole indicates in the Lyapunov case. For $r > 0$, $z \in \Delta^l_\omega$ and $n\geq 0$ let
\begin{align*}
\sball[\bar \xi]{r}{z,n}&:= \left\{ \xi \in E_n(\omega,z) : \abs{\bar \xi-\xi} \leq r \right\},\\
\uball[\bar \eta]{r}{z,n}&:= \left\{ \eta \in H_n(\omega,z)  : \abs{\bar\eta- \eta} \leq r \right\},\\
\sLball[\bar \xi]{r}{z,n}&:= \left\{ \xi \in E_n(\omega,z)  : \LnormAt[n]{z}{\bar\xi- \xi} \leq r \right\},\\
\uLball[\bar \eta]{r}{z,n}&:= \left\{ \eta \in H_n(\omega,z) : \LnormAt[n]{z}{\bar\eta- \eta} \leq r \right\},
\end{align*}
where $\bar \xi \in E_n(\omega,z)$, $\bar \eta \in H_n(\omega,z)$, and
\begin{align*}
\ball[\bar \zeta]{r}{z,n} &:= \sball[\bar \xi]{r}{z,n} \times \uball[\bar \eta]{r}{z,n},\\
\Lball[\bar \zeta]{r}{z,n} &:= \sLball[\bar \xi]{r}{z,n} \times \uLball[\bar \eta]{r}{z,n},
\end{align*}
where $\bar \zeta = \bar \xi + \bar \eta$. If we consider the ball around the origin in $T_{f^n_\omega z}\R^d$, we will omit to specify the center of the ball, e.g. we will abbreviate $\sball{r}{z,n} := \sball[0]{r}{z,n}$. Let us emphasize that we have fixed $(\omega,x)$ in the beginning and thus in the following we will often omit to specify the dependence on $(\omega,x)$ or $\omega$ explicitely.

Let us consider $z \in \Delta_\omega^l \cap \pLball[\Delta,\omega]{x}{\delta_{\Delta^l}/2}$ and choose $y \in W_{loc}(\omega,z) \cap \pLball[\Delta,\omega]{z}{\delta_{\Delta^l}/2}$ on the local stable manifold. Then we will denote its representation in $T_z\R^d$ by
\begin{align*}
(\xi_0,\eta_0) := \exp_{z}^{-1}(y) \in \exp_{z}^{-1}(W_{loc}(\omega,z)) \cap \Lball{\delta_{\Delta^l}/2}{z,0}
\end{align*}
with $\xi_0 \in E_0(\omega,z)$ and $\eta_0 \in H_0(\omega,z)$ and
\begin{align*}
(\xi_n,\eta_n) := F^n_0(\omega,z) 
(\xi_0,\eta_0) = \exp_{f^n_\omega z}^{-1} (f^n_\omega y),
\end{align*}
where $\xi_n \in E_n(\omega,z)$ and $\eta_n \in H_n(\omega,z)$. In the future, when we have fixed the points $z$ and $y$ and thus the point $(\xi_0,\eta_0) \in \exp_{z}^{-1}(W_{loc}(\omega,z)) \cap \Lball{\delta_{\Delta^l}/2}{z,0}$, we will use the notation $\xi_n$ and $\eta_n$ exclusively in the sense defined above, without additional explanation.

The following proposition will allow us to compare Lyapunov norms at different points.

\begin{proposition} \label{prop:EstimateOnLnorms}
For every $z, z' \in \Delta^l_\omega$, every $z^1, z^2 \in \R^d$ and any $n \geq 0$ we have
\begin{align*}
&\LnormAt[n]{z}{\exp^{-1}_{f^n_\omega z}\left(f^n_\omega z^1\right) - \exp^{-1}_{f^n_\omega z}\left(f^n_\omega z^2\right)} \\
&\hspace{25ex}\leq 2 A e^{2\eps n}\LnormAt[n]{z'}{\exp^{-1}_{f^n_\omega z'}\left(f^n_\omega z^1\right) - \exp^{-1}_{f^n_\omega z'}\left(f^n_\omega z^2\right)},
\end{align*}
where $A$ was defined in Lemma \ref{lem:EstimatesOnLnorm}.
\end{proposition}

\begin{proof}
Fix $n \geq 0$, $z, z' \in \Delta^l_\omega$ and $z^1, z^2 \in \R^d$. For $\zeta \in T_{f^n_\omega z'}\R^d$ we have since the exponential map is a simple translation on $\R^d$
\begin{align*}
\abs{D_\zeta\left(\exp_{f^n_\omega z}^{-1}\circ\exp_{f^n_\omega z'}\right)} = 1.
\end{align*}
Denote by $L$ the line in $T_{f^n_\omega z'}\R^d$ connecting the points $\exp^{-1}_{f^n_\omega z'}(f^n_\omega z^1)$ and $\exp^{-1}_{f^n_\omega z'}(f^n_\omega z^2)$. By the mean value theorem and Lemma \ref{lem:EstimatesOnLnorm} we get
\begin{align*}
&\LnormAt[n]{z}{\exp^{-1}_{f^n_\omega z}\left(f^n_\omega z^1\right) - \exp^{-1}_{f^n_\omega z}\left(f^n_\omega z^2\right)} 
\leq A e^{2\eps n} \abs{\exp^{-1}_{f^n_\omega z}\left(f^n_\omega z^1\right) - \exp^{-1}_{f^n_\omega z}\left(f^n_\omega z^2\right)}\\
&\hspace{5ex}= A e^{2\eps n} \abs{\left(\exp^{-1}_{f^n_\omega z}\circ \exp_{f^n_\omega z'}\right)\circ  \exp^{-1}_{f^n_\omega z'}\left(f^n_\omega z^1\right) - \left(\exp^{-1}_{f^n_\omega z}\circ \exp_{f^n_\omega z'}\right) \circ \exp^{-1}_{f^n_\omega z'}\left(f^n_\omega z^2\right)}\\
&\hspace{5ex}\leq A e^{2\eps n} \sup_{\zeta \in L} \abs{D_\zeta\left(\exp_{f^n_\omega z}^{-1}\circ\exp_{f^n_\omega z'}\right)} \abs{\exp^{-1}_{f^n_\omega z'}\left(f^n_\omega z^1\right) - \exp^{-1}_{f^n_\omega z'}\left(f^n_\omega z^2\right)}\\
&\hspace{5ex}\leq 2 A e^{2\eps n} \LnormAt[n]{z'}{\exp^{-1}_{f^n_\omega z'}\left(f^n_\omega z^1\right) - \exp^{-1}_{f^n_\omega z'}\left(f^n_\omega z^2\right)}.
\end{align*}
\end{proof}

\subsection{Local Representation of Iterated Transversal Manifolds} \label{sec:ChoiceOfDelta}

From the main theorem of this section, Theorem \ref{thm:ExsitenceOfPsi}, we will deduce that the iterated transversal manifolds can be locally represented as the graph of some functions, which satisfy some invariance property and certain growth estimates. 

Let us fix some $C \in (0,1)$ and define the constant $\const{q}{1}_C$ by
\begin{align*}
\const{q}{1}_C := \min \left\{\frac{r_0}{2A}; \frac{1}{2c_0}\left(e^{b-2\eps} - e^{a+12\eps}\right); \frac{C}{4c_0}\left(e^{b-9d\eps} - e^{a+2\eps}\right); \delta_{\Delta^l} \right\},
\end{align*}
where $r_0$ and $c_0$ are defined in the proof of Theorem \ref{thm:localStableManifold} and $A$ in Lemma \ref{lem:EstimatesOnLnorm}. Further let $0 < q \leq \const{q}{1}_C$ and choose $z \in \Delta_\omega^l \cap \pLball[\Delta,\omega]{x}{q/2}$ and $y \in W_{loc}(\omega,z) \cap \pLball[\Delta,\omega]{z}{q/2}$.

From the proof of Theorem \ref{thm:localStableManifold} (see \rref{eq:EstiamteOnStableManifold}), it follows since $(\xi_0,\eta_0) \in \exp_z^{-1}\left(W_{loc}(\omega,z)\right)$ and $\LnormAt[0]{z}{(\xi_0,\eta_0)} \leq r_0$ that
\begin{align*} 
\LnormAt[n]{z}{(\xi_n,\eta_n)} 
= \LnormAt[n]{z}{F_0^n(\omega,z)(\xi_0,\eta_0)}
\leq e^{(a+6\eps)n} \LnormAt[0]{z}{(\xi_0,\eta_0)}.
\end{align*}

Then we have the following theorem, the main theorem of this section, which is basically \citep[Lemma II.6.1]{Katok86}.

\begin{theorem} \label{thm:ExsitenceOfPsi}
Let $z \in \Delta^l_\omega$, $0 < q \leq \const{q}{1}_C$ and $0 < \delta_{0} \leq q/4$, $(\xi_0,\eta_0) \in \exp_z^{-1} (W_{loc}(\omega,z))$ with $\LnormAt[0]{z}{(\xi_0,\eta_0)} \leq q/4$ and define $\delta'_{n} := \delta_{0} e^{(a+11\eps)n}$. Further let $\psi_{(\omega,z),0}: \uLball[\eta_0]{\delta_{0}}{z,0} \to E_0(\omega, z)$ be a mapping of class $C^1$ such that $\psi_{(\omega,z),0}(\eta_0) = \xi_0$ and
\begin{align}
\max_{\eta \in \uLball[\eta_0]{\delta_{0}}{z,0}} \LnormAt[0]{z}{\psi_{(\omega,z),0}(\eta)} &\leq \frac{q}{4} \label{eq:psi0Est}\\
\max_{\eta \in \uLball[\eta_0]{\delta_{0}}{z,0}} \LnormAt[0]{z}{D_{\eta}\psi_{(\omega,z),0}} &\leq C.\label{eq:Dpsi0Est}
\end{align}
Then there exists a unique sequence $\{\psi_{(\omega,z),n}\}_{n \geq 1}$ of mappings of class $C^1$ with
\begin{align*}
\psi_{(\omega,z),n}:\uLball[\eta_{n}]{\delta'_{n}}{z,n} \to E_n(\omega,z),
\end{align*}
such that for every $n \geq 0$ one has
\begin{align}
\psi_{(\omega,z),n}(\eta_n) &= \xi_n,\label{eq:PsiConsistancy}\\
\graph(\psi_{(\omega,z),n+1}) &\subseteq F_{(\omega,z),n}(\graph(\psi_{(\omega,z),n})),\label{eq:InvarianceOfW}
\end{align}
and
\begin{align}
\max_{\eta \in \uLball[\eta_{n}]{\delta'_{n}}{z,n}} \LnormAt[n]{z}{\psi_{(\omega,z),n}(\eta)} &\leq \left(\frac{1}{4} + C\right) qe^{(a+7\eps)n}\label{eq:psiEst}\\
\max_{\eta \in \uLball[\eta_{n}]{\delta'_{n}}{z,n}} \LnormAt[n]{z}{D_{\eta}\psi_{(\omega,z),n}} &\leq Ce^{-7d\eps n}.\label{eq:DpsiEst}
\end{align}
\end{theorem}

\begin{proof}
Although this is basically \citep[Lemma II.6.1]{Katok86} we will state the proof here for several reasons. In contrast to \citep{Katok86} we need to achieve a rate of convergence that involves the dimension $d$ in \rref{eq:DpsiEst} and this proof here includes the results from the proof of Theorem \ref{thm:localStableManifold} of \citep{Liu95} for the random case.

We will prove this theorem by induction. So let us show that for any $n \geq 0$ \rref{eq:psiEst} allows to define the mapping $\psi_{(\omega,z),n+1}$ satisfying the properties \rref{eq:InvarianceOfW}, \rref{eq:psiEst} and \rref{eq:DpsiEst} for $n+1$. The base of induction, for $n = 0$, follows directly from \rref{eq:psi0Est} and \rref{eq:Dpsi0Est}.

Let us assume the statement is true for some $n \geq 0$. Then the map $F_{(\omega,z),n}$ can be represented in coordinate form on $E_n(\omega,z) \oplus H_n(\omega,z)$ by
\begin{align*}
F_{(\omega,z),n}(\xi,\eta) = \left(A_{(\omega,z),n}\xi + a_{(\omega,z),n}(\xi,\eta),B_{(\omega,z),n}\eta+ b_{(\omega,z),n}(\xi,\eta)\right),
\end{align*}
where $\xi\in E_n(\omega,z)$, $\eta \in H_n(\omega,z)$,
\begin{align*}
A_{(\omega,z),n} = D_{0}F_{(\omega,z),n}\big|_{E_n(\omega,z)},\\
B_{(\omega,z),n} = D_{0}F_{(\omega,z),n}\big|_{H_n(\omega,z)},
\end{align*}
and $a_{(\omega,z),n}$, $b_{(\omega,z),n}$ are $C^1$ mappings with $a_{(\omega,z),n}(0,0) = 0$, $b_{(\omega,z),n}(0,0) = 0$ and their derivatives satisfy $D_{(0,0)}a_{(\omega,z),n} = 0$ and $D_{(0,0)}b_{(\omega,z),n} = 0$. By Lemma \ref{lem:EstimatesOnLnorm} we have
\begin{align} \label{eq:AundBestimate}
\LnormAt[n+1]{z}{A_{(\omega,z),n} \xi} &\leq e^{a+2\eps} \LnormAt[n]{z}{\xi}\quad \text{ for any } \xi \in E_n(\omega,z)\notag\\ 
\LnormAt[n+1]{z}{B_{(\omega,z),n} \eta} &\geq e^{b-2\eps} \LnormAt[n]{z}{\eta}\quad \text{ for any } \eta \in H_n(\omega,z).
\end{align}
Let $t_{(\omega,z),n} = \left(a_{(\omega,z),n},b_{(\omega,z),n}\right)$. The following proposition gives an estimate on $t_{(\omega,z),n}$ assuming the induction hypothesis (see \citep[Proposition II.6.3]{Katok86}).

\begin{proposition} \label{prop:tnestimate}
For every $\eta^1,\eta^2 \in \uLball[\eta_n]{\delta_n'}{z,n}$ we have
\begin{align*}
&\LnormAt[n+1]{z}{t_{(\omega,z),n}\left(\psi_{(\omega,z),n}(\eta^1),\eta^1\right) - t_{(\omega,z),n}\left(\psi_{(\omega,z),n}(\eta^2),\eta^2\right)} \\
&\hspace{50ex}\leq 2qc_0 e^{(a+14\eps)n}\LnormAt[n]{z}{\eta^1 - \eta^2},
\end{align*}
where $c_0$ is defined in the proof of Theorem \ref{thm:localStableManifold}.
\end{proposition}

\begin{proof}
This is basically the proof of \citep[Proposition II.6.3]{Katok86}.

By the mean value theorem we have 
\begin{align*}
&\LnormAt[n+1]{z}{t_{(\omega,z),n}(\psi_{(\omega,z),n}(\eta^1),\eta^1) - t_{(\omega,z),n}(\psi_{(\omega,z),n}(\eta^2),\eta^2)}\\
&\hspace{3ex}\leq \sup_{\zeta \in I}\LnormAt[n+1]{z}{D_\zeta t_{(\omega,z),n}}\max\left\{ \LnormAt[n]{z}{\psi_{(\omega,z),n}(\eta^1) - \psi_{(\omega,z),n}(\eta^2)}; \LnormAt[n]{z}{\eta^1- \eta^2} \right\},
\end{align*}
where $I$ denotes the line in $T_{f^n_\omega}\R^d$ that connects $(\psi_{(\omega,z),n}(\eta^1),\eta^1)$ and $(\psi_{(\omega,z),n}(\eta^2),\eta^2)$. For $\zeta \in I$ we have by induction hypothesis and $q \leq \const{q}{1}_C$
\begin{align}
\LnormAt[n]{z}{\zeta} &\leq \max_{i = 1,2}\left\{ \LnormAt[n]{z}{\psi_{(\omega,z),n}(\eta^i)};  \LnormAt[n]{z}{\eta^i}\right\}\notag \\
&\leq \left( \frac{1}{4} + C \right)q e^{(a + 7\eps)n} + \LnormAt[n]{z}{\eta_n} + \delta'_n\notag\\
&\leq \left( \frac{1}{4} + C \right)q e^{(a + 7\eps)n} + e^{(a+6\eps)n}\LnormAt[0]{z}{(\xi_0,\eta_0)} + \delta_0 e^{(a+11\eps)n}\notag\\
&\leq 2 q e^{(a + 11\eps)n} \leq r_0 e^{-3 \eps n}. \label{eq:EstimateOnZeta}
\end{align}
Because of $D_\zeta t_{(\omega,z),n} = D_\zeta F_{(\omega,z),n} - D_0 F_{(\omega,z),n}$ we can apply \rref{eq:EstimatesOfStabManfThm} and thus we get for $\zeta \in I$ by \rref{eq:EstimateOnZeta}
\begin{align*}
\LnormAt[n\to n+1]{z}{D_\zeta t_{(\omega,z),n}} \leq c_0e^{3\eps n} \LnormAt[n]{z}{\zeta} \leq 2qc_0e^{(a + 14\eps) n}.
\end{align*}
And by assumption \rref{eq:psiEst} and the mean value theorem we have
\begin{align*}
&\max \left\{ \LnormAt[n]{z}{\psi_{(\omega,z),n}(\eta^1) - \psi_{(\omega,z),n}(\eta^2)}; \LnormAt[n]{z}{\eta^1- \eta^2} \right\}\\
&\hspace{25ex}\leq \max\{Ce^{-7d\eps n}; 1\}\LnormAt[n]{z}{\eta^1- \eta^2} = \LnormAt[n]{z}{\eta^1- \eta^2},
\end{align*}
which finally yields the assertion.
\end{proof}

By Proposition \ref{prop:tnestimate} and \rref{eq:AundBestimate} the mapping $\beta_n: \uLball[\eta_n]{\delta'_n}{z,n} \to H_{n+1}(\omega,z)$ defined by 
\begin{align*}
 \beta_n(\eta) = B_{(\omega,z),n}\eta + b_{(\omega,z),n}(\psi_{(\omega,z),n}(\eta),\eta)
\end{align*}
satisfies for $\eta^1, \eta^2 \in \uLball[\eta_n]{\delta'_n}{z,n}$ since $q \leq \const{q}{1}_C$
\begin{align}
&\LnormAt[n+1]{z}{\beta_n(\eta^1) - \beta_n(\eta^2)} \notag\\ 
&\hspace{10ex}\geq \LnormAt[n+1]{z}{B_{(\omega,z),n}(\eta^1-\eta^2)} \notag\\
&\hspace{20ex}- \LnormAt[n+1]{z}{b_{(\omega,z),n}(\psi_{(\omega,z),n}(\eta^1),\eta^1) - b_{(\omega,z),n}(\psi_{(\omega,z),n}(\eta^2),\eta^2)} \notag\\
&\hspace{10ex}\geq \left(e^{b-2\eps} - 2qc_0e^{(a + 14\eps) n} \right)\LnormAt[n]{z}{\eta^1-\eta^2}\notag\\
&\hspace{10ex}\geq e^{a+12\eps}\LnormAt[n]{z}{\eta^1-\eta^2} \label{eq:EstimateOnbeta}.
\end{align}
Thus $\beta_n$ is an $C^1$ injective immersion and its image contains the ball of radius $e^{a+12\eps}\delta'_n > e^{a+11\eps}\delta'_n = \delta'_{n+1}$ around (using \rref{eq:PsiConsistancy} for $n$)
\begin{align*}
 \beta_n(\eta_n) = B_{(\omega,z),n}\eta_n+ b_{(\omega,z),n}(\psi_{(\omega,z),n}(\eta_n),\eta_n) = B_{(\omega,z),n}\eta_n+ b_{(\omega,z),n}(\xi_n,\eta_n) = \eta_{n+1}.
\end{align*}
In particular $\beta^{-1}_n$ is well defined and $C^1$ on $\uLball[\eta_{n+1}]{\delta'_{n+1}}{z,n+1}$. This allows us to express $\psi_{(\omega,z),n+1}$ as
\begin{align*}
\psi_{(\omega,z),n+1} = \pi_{E_{n+1}(\omega,z)} \circ F_{(\omega,z),n} \circ (\psi_{(\omega,z),n}\times\id_{H_n(\omega,z)}) \circ \beta_n^{-1},
\end{align*}
where $\pi_{E_{n+1}(\omega,z)}$ denotes the orthogonal projection of $T_{f^n_\omega}\R^d$ to $E_{n+1}(\omega,z)$ with respect to $\langle\cdot,\cdot\rangle_{(\omega,z),n}$ and $\id_{H_n(\omega,z)}$ the identity map in $H_n(\omega,z)$. Then we immediately get
\begin{align*}
&\left\{\left(\psi_{(\omega,z),n+1}(\eta),\eta\right):\eta \in \uLball[\eta_{n+1}]{\delta'_{n+1}}{z,n+1}\right\} \\
&\hspace{10ex}\subseteq F_{(\omega,z),n}\left(\left\{\left(\psi_{(\omega,z),n}(\eta),\eta\right):\eta \in \uLball[\eta_{n}]{\delta'_{n}}{z,n}\right\}\right),
\end{align*}
which is \rref{eq:InvarianceOfW} and $\psi_{(\omega,z),n+1}(\eta_{n+1}) = \xi_{n+1}$, which is \rref{eq:PsiConsistancy}. In the next step we need to achieve the estimate in \rref{eq:DpsiEst} for $n+1$. Our aim is to estimate (for ease of notation we will abbreviate $\LnormAt[n]{z}{\cdot}$ by $\norm{\cdot}_n$ and $\psi_{(\omega,z),n}$ by $\psi_n$)
\begin{align*}
\frac{\norm{\psi_{n+1}(\eta +  \tau) - \psi_{n+1}(\eta)}_{n+1}}{\norm{\tau}_{n+1}},
\end{align*}
for $\eta, \eta+\tau \in \uLball[\eta_{n+1}]{\delta'_{n+1}}{z,n+1}$. Let $\tilde \eta := \beta_n^{-1}(\eta)$ and $\tilde \eta + \tilde \tau := \beta_n^{-1}(\eta + \tau)$. Because of \rref{eq:EstimateOnbeta} we have $\tilde \eta, \tilde \eta + \tilde \tau \in \uLball[\eta_{n}]{\delta'_{n}}{z,n}$. By definition of $\beta_n$ we have
\begin{align*}
\tau = \beta_n(\tilde \eta + \tilde \tau) - \beta_n(\tilde \eta) = B_n\tilde\tau + b_n(\psi_n(\tilde \eta + \tilde\tau),\tilde \eta + \tilde\tau) - b_n(\psi_n(\tilde \eta),\tilde \eta).
\end{align*}
Since $F_{(\omega,z),n}(\psi_n(\tilde \eta),\tilde \eta) = (\psi_{n+1}(\eta),\eta)$ and $F_{(\omega,z),n}(\psi_n(\tilde \eta + \tilde \tau),\tilde \eta+ \tilde \tau) = (\psi_{n+1}(\eta +\tau),\eta+\tau)$ we get
\begin{align*}
\psi_{n+1}(\eta) &= A_n\psi_n(\tilde \eta) + a_n(\psi_n(\tilde \eta),\tilde \eta),\\
\psi_{n+1}(\eta + \tau) &= A_n\psi_n(\tilde \eta + \tilde \tau) + a_n(\psi_n(\tilde \eta + \tilde \tau),\tilde \eta + \tilde \tau).
\end{align*}
By choice of $q \leq \const{q}{1}_C$ we have that $2qc_0 < e^{b-2\eps}$. Thus applying Proposition \ref{prop:tnestimate} and \rref{eq:AundBestimate} we get
\begin{align*}
&\frac{\norm{\psi_{n+1}(\eta +  \tau) - \psi_{n+1}(\eta)}_{n+1}}{\norm{\tau}_{n+1}}\\
&\hspace{5ex}= \frac{\norm{A_n\left(\psi_n(\tilde \eta + \tilde \tau) - \psi_n(\tilde \eta)\right) + a_n(\psi_n(\tilde \eta + \tilde \tau),\tilde \eta + \tilde \tau) - a_n(\psi_n(\tilde \eta),\tilde \eta)}_{n+1}}{\norm{B_n\tilde\tau + b_n(\psi_n(\tilde \eta + \tilde\tau),\tilde \eta + \tilde\tau) - b_n(\psi_n(\tilde \eta),\tilde \eta)}_{n+1}}\\
&\hspace{5ex}\leq \frac{e^{a+2\eps}\norm{\psi_n(\tilde \eta + \tilde \tau) - \psi_n(\tilde \eta)}_n + \norm{a_n(\psi_n(\tilde \eta + \tilde \tau),\tilde \eta + \tilde \tau) - a_n(\psi_n(\tilde \eta),\tilde \eta)}_{n+1}}{\norm{B_n\tilde\tau}_{n+1} - \norm{b_n(\psi_n(\tilde \eta + \tilde\tau),\tilde \eta + \tilde\tau) - b_n(\psi_n(\tilde \eta),\tilde \eta)}_{n+1}}\\
&\hspace{5ex}\leq \frac{e^{a + 2\eps} \frac{\norm{\psi_n(\tilde \eta + \tilde \tau) - \psi_n(\tilde \eta)}_n}{\norm{\tilde \tau}_n} + 2qc_0 e^{(a+14\eps)n}}{e^{b-2\eps} - 2qc_0 e^{(a+14\eps)n}}.
\end{align*}
Since $\norm{\tau}_{n+1} \to 0$ implies by continuity of $\beta_n$ that $\norm{\tilde \tau}_{n} \to 0$ so by the induction hypothesis we get
\begin{align*}
&\sup_{\eta\in\uLball[\eta_{n+1}]{\delta'_{n+1}}{z,{n+1}}} \limsup_{\norm{\tau}_{n+1} \to 0} \frac{\norm{\psi_{n+1}(\eta +  \tau) - \psi_{n+1}(\eta)}_{n+1}}{\norm{\tau}_{n+1}} \\
&\hspace{45ex}\leq \frac{e^{a + 2\eps} C e^{-7d\eps n} + 2qc_0 e^{(a+14\eps)n}}{e^{b-2\eps} - 2qc_0 e^{(a+14\eps)n}}\\
&\hspace{45ex}\leq e^{-7d\eps n} \frac{e^{a + 2\eps} C  + 2qc_0 e^{(a+21d\eps)n}}{e^{b-2\eps} - 2qc_0 e^{(a+14\eps)n}}\\
&\hspace{45ex}\leq e^{-7d\eps n} \frac{e^{a + 2\eps} C  + 2qc_0}{e^{b-2\eps} - 2qc_0}.
\end{align*}
Since $q \leq \const{q}{1}_C$ we have
\begin{align*}
\max_{\eta \in \uLball[\eta_{n+1}]{\delta'_{n+1}}{z,n+1}}\!\! \norm{D_{\eta}\psi_{n+1}}_{n+1} &\leq \!\!\!\sup_{\eta\in\uLball[\eta_{n+1}]{\delta'_{n+1}}{z,{n+1}}} \limsup_{\norm{\tau}_{n+1} \to 0} \frac{\norm{\psi_{n+1}(\eta +  \tau) - \psi_{n+1}(\eta)}_{n+1}}{\norm{\tau}_{n+1}}\\ &\leq C e^{-7d\eps (n+1)}.
\end{align*}
The last step is to verify \rref{eq:psiEst} for $n+1$. Observe that for $\eta \in \uLball[\eta_{n+1}]{\delta'_{n+1}}{z,n+1}$
\begin{align*}
\norm{\psi_{n+1}(\eta)}_{n+1} &\leq \norm{\psi_{n+1}(\eta_{n+1})}_{n+1} + \norm{\psi_{n+1}(\eta) - \psi_{n+1}(\eta_{n+1})}_{n+1} \\
&\leq \norm{(\xi_{n+1},\eta_{n+1})}_{n+1} + \sup_{\eta \in \uLball[\eta_{n+1}]{\delta'_{n+1}}{z,n+1}} \norm{D_\eta\psi_{n+1}}_{n+1} \norm{\eta_{n+1} - \eta}_{n+1}\\
&\leq e^{(a+6\eps)(n+1)}\norm{(\xi_0,\eta_0)}_0 + \delta'_{n+1} C e^{-7d\eps n}\\
&\leq \frac{q}{4}e^{(a+6\eps)(n+1)} + \frac{q}{4} C e^{(a + 11\eps)(n+1)}e^{-7d\eps n}\\
&\leq \left(\frac{1}{4} + C\right)qe^{(a+7\eps)(n+1)},
\end{align*}
which proves \rref{eq:psiEst} for $n+1$ by taking the supremum over all $\eta \in \uLball[\eta_{n+1}]{\delta'_{n+1}}{z,n+1}$.
\end{proof}

Since $E_0(\omega,z)$ and $H_0(\omega,z)$ depend continuously on $(\omega,z)$ we can choose an orthonormal basis $\{\zeta_i(\omega,z) : i = 1, \dots, d\}$ of $T_z\R^d$ with respect to $\left\langle\cdot,\cdot\right\rangle_{(\omega,z),0}$ such that $\{\zeta_i(\omega,z) : i = 1, \dots, k\}$ is a basis of $E_0(\omega,z)$ and which also depends continuously on $(\omega,z) \in \Delta^l$. Let us define for each $(\omega,z) \in \Delta^l$ the linear map
\begin{align*}
A(\omega,z): \R^d \to T_z\R^d, \quad A(\omega,z)e_i = \zeta_i(\omega,x),
\end{align*}
where $e_i$ denotes the $i\textsuperscript{th}$ unit vector in $\R^d$. Since $\zeta_i(\omega,z)$ depends continuously on $(\omega,z)$ the same is true for $A(\omega,z)$. Then for $(\omega,z),(\omega',z') \in \Delta^l$ let us denote the map
\begin{align*}
I_{(\omega,z),(\omega',z')}  : \R^d \to \R^d, \quad I_{(\omega,z),(\omega',z')} = A(\omega',z')^{-1}\circ \exp_{z'}^{-1} \circ \exp_z\circ A(\omega,z).
\end{align*}
The function $I_{(\omega,z),(\omega',z')}$ describes the change of basis from $T_z\R^d$ to $T_{z'}\R^d$ equipped with the orthonormal basis with respect to the Lyapunov metric. Then we have the following lemma, which is \citep[Proposition 7.1]{Katok86}.

\begin{lemma} \label{lem:BoundOnIdentity}
There exists a continuous nondecreasing function $R: [0,\infty) \to \R$ with $R(0) = 0$, $R(q) > 0$ for $q > 0$ such that for any $(\omega,z) \in \Delta^l$ and $(\omega',z') \in V_{\Delta^l}((\omega,z),q)$ and for every $v \in \R^d$ with $\abs{v} \leq 1$ we have
\begin{align*}
\abs{D_vI_{(\omega,z),(\omega',z')} - \id} \leq R(q).
\end{align*}
\end{lemma}

\begin{proof}
Since $A(\omega,z)$ is linear and depends continuously on $(\omega,z)$ the function
\begin{align*}
\left((\omega,z),(\omega',z'),v\right) \mapsto D_vI_{(\omega,z),(\omega',z')}
\end{align*}
is continuous and hence uniformly continuous on the compact set $\Delta^l \times \Delta^l \times \{v \in \R^d: \abs{v} \leq 1\}$. Thus let us define
\begin{align*}
R(q) := \sup_{(\omega,z),(\bar \omega,\bar z) \in \Delta^l} \sup_{\substack{(\omega',z') \in V_{\Delta^l}((\omega,z),q) \\ (\bar\omega',\bar z') \in V_{\Delta^l}((\bar\omega,\bar z),q)}} \sup_{\substack{v, \bar v\in \R^d\\\abs{v} \leq 1, \abs{v-\bar v}\leq q}}\abs{D_vI_{(\omega,z),(\omega',z')} - D_{\bar v}I_{(\bar\omega,\bar z),(\bar\omega',\bar z')}}.
\end{align*}
Clearly $0 \leq R(q) < +\infty$ for $q \geq 0$ and if one chooses $(\omega',z') = (\bar \omega,\bar z) = (\bar \omega', \bar z')$ and $v = \bar v$ then this is exactly the desired.
\end{proof}

Now let $0 < \const{q}{2} \leq \delta_{\Delta^l}$ be such that $0 < R(\const{q}{2}) < \frac{1}{5}$ and let $W$ be a transversal submanifold of $\pLball[\Delta,\omega]{x}{\const{q}{2}}$ with $\norm{W} \leq 1/2$. Then by choice of $\delta_{\Delta^l}$ for all $(\omega',x') \in V_{\Delta^l}((\omega,x),\const{q}{2}/2)$ the local stable manifold $W_{loc}(\omega',x') \cap \pLball[\Delta,\omega]{x}{\const{q}{2}}$ is the graph of a function $\phi$ (see Section \ref{sec:TheoremAbsoluteContinuity}) with
\begin{align} \label{eq:EstimateOnRepresentation}
 \sup \left\{\Lnorm[0]{D_\xi \phi} : \xi \in \sTS{0}, \Lnorm[0]{\xi} < \const{q}{2}  \right\} \leq\frac{1}{3}.
\end{align}
Because of \rref{eq:EstimateOnRepresentation} and $\norm{W} \leq 1/2$ the submanifold $W \cap \pLball[\Delta,\omega']{x'}{\const{q}{2}}$ can be represented by a $C^1$ function $\psi_{(\omega',x')}$, i.e. there exists an open subset $O_{(\omega',x')}$ of $H_0(\omega',x')$ and a function $\psi_{(\omega',x')} : O_{(\omega',x')} \to E_0(\omega',x')$ whose graph represents $W$, i.e.
\begin{align*}
W \cap \pLball[\Delta,\omega']{x'}{\const{q}{2}} = \exp_{x'} \left( \left\{\left(\psi_{(\omega',x')}(\eta),\eta\right):\eta \in O_{(\omega',x')} \right\}\right).
\end{align*}
Then we have the following proposition \citep[Corollary II.7.1]{Katok86}.

\begin{proposition} \label{prop:EstimateOnDerivativeOfPsi}
For every $z \in \Delta^l_\omega \cap \pLball[\Delta,\omega]{x}{\const{q}{2}}$ we have
\begin{align*}
\sup_{\eta \in O_{(\omega,z)}} \LnormAt[0]{z}{D_\eta\psi_{(\omega,z)}} \leq 2(\norm{W} + R(\const{q}{2})).
\end{align*}
\end{proposition}

\begin{proof}
Let us define 
\begin{align*}
\hat \psi_{(\omega,z)} := A(\omega,z)^{-1} \circ \psi_{(\omega,z)} \circ A(\omega,z)|_{\Span(e_{k+1},\dots, e_d)},
\end{align*}
where it makes sense. Then one can easily check that with
\begin{align*}
I_{(\omega,x),(\omega,z)} = \left(I^s_{(\omega,x),(\omega,z)},I^u_{(\omega,x),(\omega,z)}\right): \R^k \times \R^{d-k} \to \R^k \times \R^{d-k}
\end{align*}
we have for those $v \in \R^k$ where it makes sense
\begin{align*}
\hat \psi_{(\omega,z)} \circ I^u_{(\omega,x),(\omega,z)}(\hat\psi(v),v) = I^s_{(\omega,x),(\omega,z)}(\hat\psi(v),v)
\end{align*}
with
\begin{align*}
\hat\psi := A(\omega,x)^{-1} \circ \psi \circ A(\omega,x)|_{\Span(e_{k+1},\dots, e_d)},
\end{align*}
where $\psi: H_0(\omega,x) \to E_0(\omega,x)$ is the function that represents the transversal manifold $W$ by definition. Now the proof of \citep[Proposition II.7.2]{Katok86} combined with Lemma \ref{lem:BoundOnIdentity} and the fact that $R(\const{q}{2}) < 1/5$ and $\norm{W} \leq 1/2$ yields
\begin{align*}
\sup_{v \in A(\omega,z)^{-1}\left(O_{(\omega,z)}\right)} \abs{D_v \hat \psi_{(\omega,z)}} \leq 2(\norm{W} + R(\const{q}{2})).
\end{align*}
Since $A(\omega,z)$ is an orthogonal map from $(\R^d, \abs{\cdot})$ to $(T_z\R^d, \LnormAt[0]{z}{\cdot})$ we immediately get
\begin{align*}
\sup_{\eta \in O_{(\omega,z)}} \LnormAt[0]{z}{D_\eta \psi_{(\omega,z)}} \leq 2(\norm{W} + R(\const{q}{2})).
\end{align*}
\end{proof}

Now choose constants $\const{q}{3}_C$ and $\eps_C$ such that 
\begin{align*}
&0 < \eps_C < \frac{1}{2},\\
&0 < \const{q}{3}_C < \min\left\{\frac{\const{q}{1}_C}{16A};\const{q}{2}\right\},\\
&\eps_C + R(\const{q}{3}_C) < \frac{C}{2}
\end{align*}
and consider a transversal manifold $W$ of $\pLball[\Delta,\omega]{x}{\const{q}{3}_C}$ with $\norm{W} \leq \eps_C$. Choose a point $z \in  \Delta^l_\omega \cap \pLball[\Delta,\omega]{x}{\const{q}{3}_C/2} $ be such that $W_{loc}(\omega, z) \cap W \cap \pLball[\Delta,\omega]{x}{\const{q}{3}_C} \neq \emptyset$. This intersection consists by transversality of exactly one point, which we will denote by $y$. As usual denote $(\xi_0,\eta_0) = \exp_z^{-1}(y)$. Let $\psi_{(\omega,z)}$ and $O_{(\omega,z)}$ be as constructed before. Then we define 
\begin{align} \label{eq:DefinitionOfqCzW}
q_C(z,W) &:= \sup \bigg\{\delta_0 : \delta_0 \leq \frac{\const{q}{3}_C}{4}, \uLball[\eta_0]{\delta_0}{z,0} \subseteq O_{(\omega,z)} \cap \uLball{\const{q}{3}_C}{z,0} \notag\\
&\hspace{20ex}\text{ and } \exp_z\left(\Lball[(\xi_0,\eta_0)]{\delta_0}{z,0}\right) \subseteq \pLball[\Delta,\omega]{x}{\const{q}{3}_C}\bigg\}.
\end{align}
Lemma \ref{lem:BoundOnIdentity} guarantees that the first inclusion holds for positive $\delta_0$, whereas since $W$ is a submanifold of $\pLball[\Delta,\omega]{x}{\const{q}{3}_C}$ and because of \rref{eq:EstimateOnRepresentation} this is also true for the second inclusion. Thus $q_C(z,W) > 0$ and one can even see that for fixed $W$ both remarks hold uniformly in $z \in \Delta^l_\omega \cap \pLball[\Delta,\omega]{x}{\const{q}{3}_C/2}$. By definition of $\psi_{(\omega,z)}$ we clearly have $\psi_{(\omega,z)}(\eta_0) = \xi_0$ and for $0< \delta_0 < q_C(z,W)$ we get by Proposition \ref{prop:EstimateOnLnorms}
\begin{align*}
 \LnormAt[0]{z}{(\xi_0,\eta_0)} &= \LnormAt[0]{z}{\exp_z^{-1}(y) - \exp_z^{-1}(z)} \leq 2A \LnormAt[0]{x}{\exp_x^{-1}(y) - \exp_x^{-1}(z)}\\
&\leq 2A \left(\LnormAt[0]{x}{\exp_x^{-1}(y)} + \LnormAt[0]{x}{\exp_x^{-1}(z)}\right) \leq 4A \const{q}{3}_C \leq \frac{1}{4} \const{q}{1}_C
\end{align*}
and similarly since $\exp_z(\psi_{(\omega,z)}(\eta)) \in \pLball[\Delta,\omega]{x}{\const{q}{3}_C}$ for each $\eta \in \uLball[\eta_0]{\delta_0}{z,0}$
\begin{align*}
\sup_{v \in \uLball[\eta_0]{\delta_0}{z,0}} \LnormAt[0]{z}{\psi_{(\omega,z)}(v)} &\leq \sup_{v \in \uLball[\eta_0]{\delta_0}{z,0}} \LnormAt[0]{z}{\psi_{(\omega,z)}(v) - \exp_z^{-1}(z)} \\
&\leq 2A \sup_{v \in \uLball[\eta_0]{\delta_0}{z,0}} \LnormAt[0]{x}{\exp_x^{-1}(\exp_z(\psi_{(\omega,z)}(v))) - \exp_x^{-1}(z)}\\
&\leq 4A \const{q}{3}_C \leq \frac{1}{4} \const{q}{1}_C.
\end{align*}
Finally from Proposition \ref{prop:EstimateOnDerivativeOfPsi} and choice of $\const{q}{3}_C$ we get
\begin{align*}
 \sup_{\eta \in \uLball[\eta_0]{\delta_0}{z,0}} \LnormAt[0]{z}{D_\eta\psi_{(\omega,z)}} \leq 2(\norm{W} + R(\const{q}{3}_C)) \leq 2(\eps_C + R(\const{q}{3}_C))\leq C.
\end{align*}
Thus for $q = \const{q}{1}_C$, $0 < \delta_0 < q_C(z,W)$ and $\psi_0 := \psi_{(\omega,z)}|_{\uLball[\eta_0]{\delta_0}{z,0}}$ the assumptions of Theorem \ref{thm:ExsitenceOfPsi} are fullfilled and we obtain for each $n \geq 0$  mappings
\begin{align*}
\psi_{(\omega,z),n} : \uLball[\eta_n]{\delta_n'}{z,n} \to H_n(\omega,z),
\end{align*}
which satisfy
\begin{align*}
\psi_{(\omega,z),n}(\eta_n) &= \xi_n,\\
\graph(\psi_{(\omega,z),n+1}) &\subseteq F_{(\omega,z),n}(\graph(\psi_{(\omega,z),n})),
\end{align*}
and the estimates
\begin{align*}
\max_{\eta \in \uLball[\eta_{n}]{\delta'_{n}}{n}} \LnormAt[n]{z}{\psi_{(\omega,z),n}(\eta)} &\leq \left(\frac{1}{4} + C\right) qe^{(a+7\eps)n},\\
\max_{\eta \in \uLball[\eta_{n}]{\delta'_{n}}{n}} \LnormAt[n]{z}{D_{\eta}\psi_{(\omega,z),n}} &\leq Ce^{-7d\eps n}.
\end{align*}

With this sequence of maps we are able to define the $(d-k)$-dimensional submanifold of $\R^d$, which will play an important role in the following. For any $n \geq 0$ and $0< r < q_C(z,W)e^{(a + 11\eps)n}$ let us define
\begin{align*}
 \tilde W_n(z,y,r) := \exp_{f^n_\omega z} \left\{ (\psi_{(\omega,z),n}(\eta),\eta) : \eta \in \uLball[\eta_n]{r}{z,n} \right\}.
\end{align*}
In particular, for $0 < \delta_0 < q_C(z,W)$ and $\delta'_n = \delta_0 e^{(a+11\eps)n}$ we can consider the submanifolds $\tilde W_n(z,y,\delta'_n)$. By Theorem \ref{thm:ExsitenceOfPsi} we immediately get
\begin{align} \label{eq:InvarianceOfRepresentation}
\tilde W_n(z,y,\delta'_n) \subset f_n(\omega) \left(\tilde W_{n-1}(z,y,\delta'_{n-1})\right),
\end{align}
which is a very important property for the future. Let us emphasize that if one uses the Euclidean metric on the tangent spaces instead of the Lyapunov metric then this property is not true in general anymore.

\subsection{Projection Lemmas} \label{sec:ProjectionLemmas}

For $n \geq 0$, $z' \in f^n_\omega(W)$ and $q > 0 $ we will denote by $Q(z',q)$ the closed ball in $f^n_\omega(W)$ of radius $q$ centered at $z'$ with respect to the induced Euclidean metric on $f^n_\omega(W)$. For fixed $\delta_0 > 0$ let us define 
\begin{align*}
d_0 := \frac{\delta_0}{12 A}  \quad \text{ and } \quad d_n := d_0 e^{(a+9\eps)n}
\end{align*}
for $n \geq 0$. Then we have the following proposition, which compares the Euclidean balls in $f^n_\omega(W)$ with the submanifolds constructed at the end of the previous section. As before let $z \in  \Delta^l_\omega \cap \pLball[\Delta,\omega]{x}{\const{q}{3}_C/2} $ be such that $W_{loc}(\omega, z) \cap W \cap \pLball[\Delta,\omega]{x}{\const{q}{3}_C} \neq \emptyset$ and denote this intersection by $y$. Further let $0 < \delta_0 < q_C(z,W)$. Then we have \citep[Porposition II.8.1]{Katok86}.

\begin{proposition} \label{prop:ComparsionWandQ}
For any $n\geq 0$ we have
\begin{enumerate}
\item [a)] if $z' \in \tilde W_n(z,y,\frac{1}{2}\delta'_n)$ then $Q(z', 3d_n) \subset \tilde W_n(z,y,\frac{3}{4}\delta'_n)$;
\item [b)] if $z' \in \tilde W_n(z,y,\frac{3}{4}\delta'_n)$ then $Q(z', 3d_n) \subset \tilde W_n(z,y, \delta'_n)$.
\end{enumerate}
\end{proposition}

\begin{proof}
This is \citep[Proposition II.8.1]{Katok86}.
\end{proof}

Let $F$ be a $k$-dimensional subspace of $T_{f^n_\omega z}\R^d$ transversal to the subspace $H_n(\omega,z)$ such that
\begin{align} \label{eq:AngleOfSubpaces}
\gamma(F,H_n(\omega,z)) \geq {l'}^{-1}e^{-\eps n},
\end{align}
where $\gamma(\cdot,\cdot)$ denotes  the angle between two subspaces with respect to the Euclidean scalar product and $l'$ was fixed in the beginning of Section \ref{sec:TheoremAbsoluteContinuity}. Two examples that will be considered in the following are the Riemannian orthogonal complement $H^{\bot}_n(\omega,z)$ and $E_n(\omega,z)$, which satisfies \rref{eq:AngleOfSubpaces} because of Lemma \ref{lem:ExistenceOfl}.

Let us denote by $\pi^n_F$ the projection of $T_{f^n_\omega z}\R^d$ onto $H_n(\omega,z)$  parallel to the subspace $F$. Further let $\hat Q(z',q) := \exp_{f^n_\omega z}^{-1}(Q(z',q))$ and for $z' \in \R^d$ let $\hat z' := \exp_{f^n_\omega z}^{-1}(z')$. Then we have the following projection lemma (see \citep[Lemma II.8.1]{Katok86}), which compares the projection along the subspace $F$ of an Euclidean ball in $f^n_\omega(W)$ with an Euclidean ball in $H_n(\omega,z)$ for large $n$.

\begin{lemma} \label{lem:ProjectionLemma}
For every $\alpha \in (0,1)$  there exists $\const{N}{1} = \const{N}{1}(\alpha)$ such that for any $n \geq \const{N}{1}$, any $z' \in \tilde W(z,y,\frac{3}{4}\delta'_n)$, any $0 < q \leq 3d_n$, and any subspace $F \subset T_{f^n_\omega z}\R^d$ which satisfies \rref{eq:AngleOfSubpaces} we have
\begin{align*}
B^u_{z,n}(\pi^n_F(\hat z'),(1-\alpha)q) \subset \pi^n_F(\hat Q(z',q)) \subset B^u_{z,n}(\pi^n_F(\hat z'),(1+\alpha)q).
\end{align*}
\end{lemma}

\begin{proof}
This is \citep[Lemma II.8.1]{Katok86}.
\end{proof}

\begin{remark}
The quantity $\const{N}{1}(\alpha)$ can be assumed to be decreasing in $\alpha$. 
\end{remark}

As an immediate consequence of this lemma and the properties of the function $\psi_{(\omega,z),n}$ we get the following corollary.

\begin{corollary} \label{cor:ExistenceOfPsiPiN}
There exists a number $\const{N}{2}$ such that for any $n \geq \const{N}{2}$ and each $z' \in \tilde W(z,y,\frac{3}{4}\delta'_n)$ there exists a $C^1$ map $\Psi_{\pi,n}: B^u_{z,n}\left(\pi^n_F(\hat z'), \frac{8}{3}d_n\right) \to H_n(\omega,z)$ such that
\begin{align*}
\hat Q\left(z', \frac{7}{3}d_n\right) \subset \graph(\Psi_{\pi,n}) \subset \hat Q(z',3d_n)
\end{align*}
and the derivative satisfies for any $y' \in B^u_{z,n}\left(\pi^n_F(\hat z'), \frac{8}{3}d_n\right)$
\begin{align*}
 \abs{D_{y'}\Psi_{\pi,n}} \leq 2 A e^{-5\eps n}.
\end{align*}
\end{corollary}

\begin{proof}
Because of Proposition \ref{prop:ComparsionWandQ} the function $\psi_{(\omega,z),n}$ is well defined on $\pi_F^n\left(\hat Q(z',3d_n\right)$. Thus by Lemma \ref{lem:ProjectionLemma} there exists $\const{N}{2} := \max \left\{\const{N}{1}(1/9); \const{N}{1}(1/7)\right\}$ such that for $n \geq \const{N}{2}$ we have
\begin{align*}
\pi^n_F\left(\hat Q\left(z', \frac{7}{3}d_n\right)\right) \subset B^u_{z,n}\left(\pi^n_F(\hat z'),\frac{8}{3}d_n\right) \subset \pi^n_F(\hat Q(z',3d_n)).
\end{align*}
Thus we can define $\Psi_{\pi,n} := \psi_{(\omega,z),n}|_{B^u_{z,n}\left(\pi^n_F(\hat z'),\frac{8}{3}d_n\right)}$, which satisfies because of Lemma \ref{lem:EstimatesOnLnorm} and \rref{eq:DpsiEst} for any $y' \in B^u_{z,n}\left(\pi^n_F(\hat z'), \frac{8}{3}d_n\right)$
\begin{align*}
\abs{D_{y'}\Psi_{\pi,n}} \leq 2Ae^{2\eps n} \LnormAt{z}{D_{y'}\Psi_{\pi,n}} \leq 2Ae^{-5\eps n}.
\end{align*}
\end{proof}

For $n \geq 0$ let us denote by $\lambda_n$ and $\hat \lambda_n$ the $(d-k)$-dimensional Riemannian volume on $\tilde W(z,y,\delta'_n)$ and $\hat W(z,y,\delta'_n):= \exp_{f^n_\omega z}^{-1}\left(W(z,y,\delta'_n)\right)$ respectively. For $z' \in \tilde W(z,y,\frac{3}{4}\delta'_n)$ and $\theta \in (0,1/6)$ let
\begin{align*}
A_n(z',\theta) := \left\{ y' \in f^n_\omega(W) : 2d_n(1-\theta) \leq \tilde d(y',z') \leq 2d_n \right\}
\end{align*}
the $\theta$-boundary of $Q(z',2d_n)$, where $\tilde d$ denotes the induced Euclidean metric on $f^n_\omega(W)$. By Proposition \ref{prop:ComparsionWandQ} we get that $A(z',\theta) \subset \tilde W_n(z,y,\delta'_n)$ and thus $\lambda_n(A_n(z',\theta))$ is well defined. The next lemma compares the volume of $A_n(z',\theta)$ to $Q(z',d_n)$, this is basically \citep[Lemma II.8.2]{Katok86}.

\begin{lemma} \label{lem:EstimateOnTheVolume}
There exists a constant $\const{C}{1}$ such that for any $\theta \in (0,1/6)$ there exists a number $\const{N}{3} = \const{N}{3}(\theta)$ such that for every $n \geq \const{N}{3}$ and every $z' \in \tilde W(z,y,\frac{3}{4}\delta'_n)$ we have
\begin{align*}
\frac{\lambda_n(A_n(z',\theta))}{\lambda_n(Q(z,d_n))} \leq \const{C}{1} \theta.
\end{align*}
\end{lemma}

\begin{proof} This is basically taken from \citep[Lemma II.8.2]{Katok86}, but some things are adapted to our situation. The proof bases on several applications of Lemma \ref{lem:ProjectionLemma}. Let us fix some $n \geq 0$ then since $\exp_{f^n_\omega z}$ is a simple translation on $\R^d$ it is sufficient to show
\begin{align*}
\frac{\hat \lambda_n(\hat A_n(z',\theta))}{\hat \lambda_n(\hat Q(z,d_n))} \leq \const{C}{1} \theta,
\end{align*}
where $\hat A_n(z',\theta) := \exp_{f^n_\omega z}^{-1}(A_n(z',\theta))$. Because of Lemma \ref{lem:ProjectionLemma} applied to $\alpha = 2\theta - \theta^2$, $F = H_n(\omega,z)^\bot$ and $q = d_n$ there exists $\const{N}{3,1}$ such that for all $n \geq \const{N}{3,1}$
\begin{align} \label{eq:firstInclusion}
B^u_{z,n}(\pi^n_F(\hat z'), (1-\theta)^2d_n) \subset \pi^n_F(\hat Q(\hat z', d_n)).
\end{align}
Since the exponential function $\exp_{f^n_\omega z}$ is again a simple translation on $\R^d$ we have for any $n \geq 0$
\begin{align*}
 \hat A_n(z',\theta) = \left\{\hat y' \in \hat W_n(z,y,\delta'_n) : 2d_n(1-\theta) \leq \hat d(\hat y',\hat z') \leq 2d_n \right\},
\end{align*}
where again $\hat z' := \exp_{f^n_\omega}^{-1}(z')$ and $\hat d$ denotes the induced Euclidean metric on $\hat W_n(z,y,\delta'_n)$. Thus we have (again let $F = H_n(\omega,z)^\bot$)
\begin{align} \label{eq:secondInclusion}
\pi^n_F(\hat A_n(z',\theta)) \subset B^u_{z,n}(\pi^n_F(\hat z'), 2d_n).
\end{align}
By definition of $A_n(z',\theta)$ we have
\begin{align} \label{eq:thirdInclusion}
\hat Q(z', 2d_n(1-\theta)^2) \subset \hat A_n(z',\theta)^c.
\end{align}
Let us again apply Lemma \ref{lem:ProjectionLemma} with $\alpha = \theta/(1-\theta)$, $F = H_n(\omega,z)^\bot$ and $q = 2d_n(1-\theta)^2$ then there exists $\const{N}{3,2}$ such that for any $n \geq \const{N}{3,2}$
\begin{align*}
B^u_{z,n}(\pi^n_F(\hat z'), 2d_n(1-\theta)(1-2\theta)) \subset \pi^n_F(\hat Q(z', 2d_n(1-\theta)^2))
\end{align*}
which yields by \rref{eq:thirdInclusion}
\begin{align} \label{eq:forthInclusion}
B^u_{z,n}(\pi^n_F(\hat z'), 2d_n(1-\theta)(1-2\theta)) \subset \pi^n_F(\hat A_n(z',\theta)^c).
\end{align}
Combining \rref{eq:secondInclusion} and \rref{eq:forthInclusion} we get
\begin{align} \label{eq:fifthInclusion}
 \pi^n_F(\hat A_n(z',\theta)) &\subset \big\{\eta \in H_n(\omega,z) : 2d_n(1-\theta)(1-2\theta) \leq \abs{\pi^n_F(\hat z') - \eta} \leq 2d_n \big\}=: R_n(z',\theta).
\end{align}

By Corollary \ref{cor:ExistenceOfPsiPiN} there exists $\const{N}{3,3} \geq \const{N}{2}$ such that $\abs{D_{y'}\Psi_{\pi,n}} \leq 1$ for all $n \geq \const{N}{3,3}$. Proposition \ref{prop:Appendix1} then implies that for every $n \geq \const{N}{3,3}$ and any measurable subset $V \subset \pi^n_F(\hat Q(z',\frac{7}{3}d_n))$ for $z' \in \tilde W_n(z,y,\frac{3}{4}\delta'_n)$ we have
\begin{align} \label{eq:sixthInclusion}
\vol(V) \leq \lambda_n\left(\hat Q\left(z', \frac{7}{3}d_n\right) \cap (\pi^n_F)^{-1}(V)\right) \leq 2^{(d-k)/2} \vol(V),
\end{align}
where $\vol(V)$ denotes the $(d-k)$-dimensional Lebesgue measure of $V$ induced by the Euclidean scalar product in $T_{f^n_\omega z}\R^d$. Let us observe that for $a,b \geq 0$, $p\in \N$ we have the factorization $a^p - b^p = (a-b)\sum_{i=1}^pa^{p-1}b^{i-1}$. Now combining \rref{eq:firstInclusion}, \rref{eq:fifthInclusion} and \rref{eq:sixthInclusion} we get for $n \geq \max\{\const{N}{3,1};\const{N}{3,2};\const{N}{3,3}\} =: \const{N}{3}$
\begin{align*}
\frac{\hat \lambda_n(\hat A_n(z',\theta))}{\hat \lambda_n(\hat Q(z,d_n))} &\leq \frac{\hat\lambda\left( \hat Q(z',2d_n) \cap (\pi^n_F)^{-1}(R_n(z',\theta))\right)}{\hat \lambda_n\left( \hat Q(z',d_n) \cap (\pi^n_F)^{-1}\left(B^u_{z,n}(\pi^n_F(\hat z'),d_n(1-\theta)^2)\right)\right)} \\
&\leq 2^{(d-k)/2} \frac{\vol(R_n(z',\theta))}{\vol\left(B^u_{z,n}(\pi^n_F(\hat z'),\frac{25}{36}d_n)\right)} \\
&= 2^{(d-k)/2} \frac{\vol\left(B^u_{z,n}(\pi^n_F(\hat z'),2d_n)\right)-\vol\left(B^u_{z,n}(\pi^n_F(\hat z'),2d_n(1-\theta)(1-2\theta))\right)}{\vol\left(B^u_{z,n}(\pi^n_F(\hat z'),\frac{25}{36}d_n)\right)} \\
&\leq 4\cdot2^{(d-k)/2} \frac{(2d_n)^{d-k} - (2d_n(1-\theta)(1-2\theta))^{d-k}}{d_n^{d-k}} \\
&\leq 4(d-k)2^{3(d-k)/2} \left( 1 - (1-\theta)(1-2\theta) \right)\\
&\leq 12(d-k)2^{3(d-k)/2} \theta.
\end{align*}
Now the result follows with $\const{C}{1} := 12(d-k)2^{3(d-k)/2}$.
\end{proof}

\subsection{Construction of a Covering} \label{subsec:covering}

As before let $W$ be a transversal submanifold. As before if $P \in W$ we will denote by $Q(P,h)$ the closed ball in $W$, with respect to the Euclidean metric induced on $W$, centered at $P$ and of radius $h$. When $h > 0$ is small enough, i.e. $0 < h < h_P$, the ball $Q(P,h)$ satisfies $Q(P,h) \subset W$.

Let us recall that $\Delta^l$ is a compact set and hence if $(\omega,x) \in \Delta^l$ then $\Delta^l_\omega$ is compact. Let us define for $0 < q < \delta_{\Delta^l}$ the closed ball in the tangent space of $x$ of radius $q$
\begin{align*}
\tilde U_{\Delta,\omega}^{cls}(x,q) := \exp_x\left\{\zeta \in T_x\R^d : \Lnorm{\zeta} \leq q\right\}.
\end{align*}
Then we have $\Int(\tilde U_{\Delta,\omega}^{cls}(x,q)) = \pLball[\Delta,\omega]{x}{q}$ and $\tilde U_{\Delta,\omega}^{cls}(x,q)$ is compact for any $q > 0$. Thus by choice of $\delta_{\Delta^l}$ the local stable manifolds $W_{loc}(\omega,z) \cap \tilde U_{\Delta,\omega}^{cls}(x,q)$ are compact for any $0 < q < \delta_{\Delta^l}$ and hence
\begin{align*}
\tilde \Delta^{l,cls}_\omega(x,q) = \bigcup_{z \in \Delta^l_\omega \cap \tilde U_{\Delta,\omega}^{cls}(x,q/2)} W_{loc}(\omega,z) \cap \tilde U_{\Delta,\omega}^{cls}(x,q)
\end{align*}
is compact. For $P \in W$ and $0 < h < h_P$ let us denote $D(P,h) := \tilde \Delta^{l,cls}_\omega(x,\const{q}{3}_{C}) \cap Q(P,h)$. As $W$ is relatively compact in $\R^d$, then $Q(P,h)$ is compact and consequently $D(P,h)$ is also a compact subset of $\R^d$. The next lemma now gives a covering of $D(P,h)$ by the local representation of the iterated transversal as constructed at the end of Section \ref{sec:ChoiceOfDelta}. Although this is basically \citep[Lemma II.8.3]{Katok86}, we here have a slightly weaker result, since the quantity $\delta_{P,\beta,h}$ in our theorem does depend on $h$.

\begin{lemma} \label{lem:FirstInclusion}
 For every $P \in W$, every $0 < \beta < h_P$ and $ 0 < h < h_P - \beta$ there exists $\delta_{P,\beta,h} > 0$ such that for every $0 < \delta_0 < \delta_{P,\beta,h}$ and every $n \geq 1$ there exists $\const{M}{1} = \const{M}{1}(n,P,\beta,\delta_0,h)$ and points $z_i \in \Delta^l_\omega\cap\pLball[\Delta,\omega]{x}{\const{q}{3}_C/2}$ for $1 \leq i \leq \const{M}{1}$, such that for every $i$ one has
\begin{align*}
 W_{loc}(\omega,z_i) \cap W \neq \emptyset.
\end{align*}
Let us denote $y_i = W_{loc}(\omega,z_i) \cap W$. The submanifolds $\tilde W_n(z_i,y_i,\delta'_n)$ are well defined and we have
\begin{align*}
 f^n_\omega\left(D(P,h)\right) &\subset \overline W_n(1/2) := \bigcup_{i = 1}^{\const{M}{1}} \tilde W_n\left(z_i, y_i, \frac{1}{2}\delta'_n\right)\\ &\subset \overline W_n(1) := \bigcup_{i = 1}^{\const{M}{1}} \tilde W_n\left(z_i, y_i, \delta'_n\right) \subset  f^n_\omega\left(Q(P,h + \beta)\right).
\end{align*}
\end{lemma}

\begin{proof}
Because of Lemma \ref{lem:EstimatesOnLnorm} the Lyapunov norm can be bounded by the Euclidean Norm uniformly for all $z \in \Delta^l_\omega \cap \pLball[\Delta,\omega]{x}{\const{q}{3}_C/2}$. Thus there exists a constant $\bar h_0$ and a function $t$ depending (both only on $a,b,k,\eps,l',r'$ and $C'$) with $0<t(h) \leq h$ for $0 < h < \bar h_0$  such that for every $z \in \Delta^l_\omega \cap \pLball[\Delta,\omega]{x}{\const{q}{3}_C/2}$ with $W_{loc}(\omega,z) \cap W \neq \emptyset$ and $y = W_{loc}(\omega,z) \cap W$ we have for any $0 < h < \min\{q_C(z,W);\bar h_0; h_y\}$
\begin{align*}
\tilde W_0(z,y, t(h)) \subset Q(y,h).
\end{align*}
Let us define for fixed $P \in W$ and $0 < h < h_P$ the number
\begin{align*}
A_{P,h} = \inf\{q_C(z,W) : z \in \Delta^l_\omega \cap \pLball[\Delta,\omega]{x}{\const{q}{3}_C/2} \text{ and }  W_{loc}(\omega,z) \cap W \in Q(P,h)\}.
\end{align*}
By the remark after the definition of $q_C(z,W)$ (see \rref{eq:DefinitionOfqCzW}) this quantity is strictly positive for all $P \in W$ and $0 < h < h_P$. Now let us define
\begin{align*}
\delta_{P,\beta,h} := \min\left\{t\left(\min\left\{\frac{\beta}{4}; \bar h_0\right\}\right);A_{P,h}\right\}.
\end{align*}
and fix numbers $n \geq 1$, $0 < \beta < h_P$, $0<h<h_P - \beta$ and $0 < \delta_0 < \delta_{P,\beta,h}$. Then for the set $f^n_\omega(D(P,h))$ we can consider the open covering
\begin{align*}
\left\{\Int\tilde W_n\left(z,y,\frac{1}{2}\delta'_n\right): z \in \Delta^l_\omega \cap \pLball[\Delta,\omega]{x}{\const{q}{3}_C/2} \text{ and } W_{loc}(\omega,z) \cap W \in Q(P,h) \right\},
\end{align*}
where the interior is meant in the induced metric on the submanifold $f^n_\omega (W)$. By definition of $\delta_{P,\beta,h}$ and since $0 < \delta_0 < \delta_{P,\beta,h} \leq A_{P,h}$ these sets are well defined. Since $D(P,h)$ is compact and $f^n_\omega$ a diffeomorphism, $f^n_\omega(D(P,h))$ is compact as well. Thus for the fixed parameter $P,\beta,h,\delta_0$ and $n$ there exists a finite covering, say
\begin{align*}
\left\{\Int\tilde W_n\left(z_i,y_i,\frac{1}{2}\delta'_n\right)\right\}_{1 \leq i \leq \const{M}{1}}.
\end{align*}
Now it only remains to prove that 
\begin{align*}
\overline W_n(1) := \bigcup_{i = 1}^{\const{M}{1}} \tilde W_n\left(z_i, y_i, \delta'_n\right) \subset  f^n_\omega\left(Q(P,h + \beta)\right),
\end{align*}
which is equivalent to that for all $1 \leq i \leq \const{M}{1}$
\begin{align} \label{eq:finalInclusion}
(f^n_\omega)^{-1}\left(\tilde W_n(z_i,y_i,\delta'_n)\right) \subset Q(P, h + \beta).
\end{align}
If this would not be true, then there exists some $1 \leq i \leq \const{M}{1}$ and a point $z'$ such that $z' \in (f^n_\omega)^{-1}\left(\tilde W_n(z_i,y_i,\delta'_n)\right)$ but $z' \notin Q(P,h+\beta)$. Because of 
\begin{align}\label{eq:SecondContradiction}
\emptyset \neq (f^n_\omega)^{-1}\left(\tilde W_n(z_i,y_i,\delta'_n)\right) \cap D(P,h) \subset (f^n_\omega)^{-1}\left(\tilde W_n(z_i,y_i,\delta'_n)\right) \cap Q(P,h) 
\end{align}
and the connectivity of $(f^n_\omega)^{-1}\left(\tilde W_n(z_i,y_i,\delta'_n)\right)$ there exists a point
\begin{align} \label{eq:FirstContradiction}
z'' \in (f^n_\omega)^{-1}\left(\tilde W_n(z_i,y_i,\delta'_n)\right) \cap \partial Q(P,h+\beta).
\end{align}
By \rref{eq:InvarianceOfRepresentation}, the definition of $\delta_{P,\beta,h}$ and the properties of the function $t$ we have
\begin{align*} 
(f^n_\omega)^{-1}\left(\tilde W_n(z_i,y_i,\delta'_n)\right) \subset \tilde W_0(z_i,y_i,\delta'_0) \subset Q\left(y_i,\frac{\beta}{4}\right).
\end{align*}
This implies on the one hand via \rref{eq:FirstContradiction}
\begin{align*}
z'' \in \partial Q(P,h+\beta) \cap Q\left(y_i,\frac{\beta}{4}\right) \neq \emptyset
\end{align*}
and on the other hand via \rref{eq:SecondContradiction}
\begin{align*}
D(P,h) \cap Q\left(y_i,\frac{\beta}{4}\right) \neq \emptyset.
\end{align*}
Since the distance between $D(P,h)$ and $\partial Q(P,h+\beta)$ is because of $D(P,h) \subset Q(P,h)$ greater than $\beta$ and $\diam\left(Q\left(y_i,\frac{\beta}{4}\right)\right) \leq \frac{\beta}{2}$ this yields a contradiction and hence \rref{eq:finalInclusion} is true for all $1 \leq i \leq \const{M}{1}$, which finishes the proof.
\end{proof}

The next step of the construction of a proper covering of $f^n_\omega\left(D(P,h)\right)$ is the following lemma. The main part of the following lemma, is to give a bound on the multiplicity of the covering. Here multiplicity is defined as follows: Let $\{A_i\}_{i \in I}$ be a family of subsets of the set $X$ and let $Y \subset X$ with $Y \subset \bigcup_{i \in I} A_i$. We will say that the {\it multiplicity} of the covering $\{A_i\}_{i \in I}$ of $Y$ is not bigger than some number $L$ if for any $y \in Y$ the number of covering elements is smaller than $L$, i.e. $\#\{i\in I:y \in A_i\} \leq L$.

\begin{lemma} \label{lem:Woverlines}
Let $P \in W$, $0 < \beta < h_P$, $0 < h < h_P - \beta$ and $0 < \delta_0 < \delta_{P,\beta,h}$. Then there exists $d_0 \in (0, \delta_0)$, $L> 0$, $\const{N}{4} = \const{N}{4}(P,\beta,\delta_0, h)$ such that for every $n \geq \const{N}{4}$ there exists $\const{M}{2} = \const{M}{2}(n,P,\beta,\delta_0, h)$ and points $\{\bar z_j\}_{1 \leq j\leq \const{M}{2}} \subset f^n_\omega(W)$ with:
\begin{enumerate}
\item for every $ 1 \leq j \leq \const{M}{2}$ there exists $1 \leq i \leq \const{M}{1}$ such that $Q(\bar z_j, 2d_n) \subset \tilde W(z_i,y_i,\delta'_n)$;
\item we have
\begin{align*}
\overline W_n(1/2) &= \bigcup_{i = 1}^{\const{M}{1}} \tilde W_n\left(z_i, y_i, \frac{1}{2}\delta'_n\right) \subset \bigcup_{j=1}^{\const{M}{2}} Q(\bar z_j, d_n)\\
&\subset \bigcup_{j=1}^{\const{M}{2}} Q(\bar z_j, 2d_n) \subset \overline W_n(1) = \bigcup_{i = 1}^{\const{M}{1}} \tilde W_n\left(z_i, y_i, \delta'_n\right);
\end{align*}
\item the multiplicity of the covering of $\overline W_n(1/2)$ by the balls $Q(\bar z_j, d_n)$, $1 \leq j \leq \const{M}{2}$, is not bigger than $L$.
\end{enumerate}
\end{lemma}

\begin{proof}
Although this is \citep[Lemma II.8.4]{Katok86} we will state the proof for sake of completeness of the covering construction. As in Section \ref{sec:ProjectionLemmas} define $d_0 := \frac{\delta_0}{12A}$ and let $n \geq 0$ be fixed for the moment. As before we will denote by $\tilde d$ the induced Euclidean metric on $f^n_\omega(W)$. As $\overline W_n(1/2)$ is compact, we can find a finite set of points $\{\bar z_j\}_{1\leq j\leq \const{M}{2}}$ such that $\tilde d(\bar z_i, \bar z_j) \geq d_n$ for all $1 \leq i,j \leq \const{M}{2}, i\neq j$, and that for any point $z' \in \overline W_n(1/2)$ there exists some $j$, $1 \leq j \leq \const{M}{2}$ such that $\tilde d(z',\bar z_j) < d_n$. Observe that such a set is not unique and its cardinality may depend on the choice of points.

Property {\it i)} follows directly from Propostion \ref{prop:ComparsionWandQ} by the choice of $d_0$. The first inclusion in {\it ii)} is satisfied by construction, the second one is obvious and the third one follows from property {\it i)}.

Thus it is left to show property {\it iii)}. For some $j$, $1\leq j \leq \const{M}{2}$, let us consider $Q(\bar z_j,d_n)$ with $\bar z_j \in \tilde W_n(z_i,y_i,\frac{1}{2}\delta'_n)$ for some  $i = i(j), 1\leq i\leq \const{M}{1}$. We will show that
\begin{align*}
\#\{1 \leq l \leq \const{M}{2}: Q(\bar z_l, d_n) \cap Q(\bar z_j, d_n) \neq \emptyset\}
\end{align*}
is bounded by some constant $K$ independently of $j$ and $n$ sufficiently large, then $L = K+1$ satisfies the desired. Since the diameter satisfies $\diam(Q(\bar z_l, d_n)) \leq 2d_n$ for any $1 \leq l \leq \const{M}{2}$ we get that if
\begin{align*}
Q(\bar z_l, d_n) \cap Q(\bar z_j, d_n) \neq \emptyset
\end{align*}
then
\begin{align*}
Q(\bar z_l, d_n) \subset Q(\bar z_j, 3d_n).
\end{align*}
Thus to prove property {\it iii)} is suffices to show that
\begin{align*}
&\#\{1 \leq l \leq \const{M}{2}: Q(\bar z_l, d_n) \cap Q(\bar z_j, d_n) \neq \emptyset\} \\
&\hspace{15ex}\leq \#\{1 \leq l \leq \const{M}{2}: Q(\bar z_l, d_n) \subset Q(\bar z_j, 3d_n) \neq \emptyset\} =: K(n,j)
\end{align*}
is bounded by some constant $K$. Since by construction we have for each $1 \leq l \leq \const{M}{2}$, $l\neq j$,
\begin{align*}
 Q\left(\bar z_l, \frac{d_n}{3}\right) \cap Q\left(\bar z_j, \frac{d_n}{3}\right) = \emptyset
\end{align*}
thus we will show that there exists $\const{N}{4}$ such that for all $n \geq \const{N}{4}$ and any $j$, $1 \leq j\leq \const{M}{2}$, the number $K(n,j)$ can be bounded by the number of disjoint balls of radius $d_n/3$ contained in $Q(\bar z_j, 3d_n)$. Thus let $z'$ such that $Q(z',\frac{d_n}{3}) \subset Q(\bar z_j, 3d_n)$. Since $\bar z_j \in \tilde W_n(z_i,z_i,\frac{1}{2} \delta'_n)$ by Proposition \ref{prop:ComparsionWandQ} we have
\begin{align*}
Q\left(z',\frac{d_n}{3}\right) \subset Q(\bar z_j, 3d_n) \subset \tilde W_n\left(z_i,y_i,\frac{3}{4}\delta'_n\right).
\end{align*}
Hence we can apply Lemma \ref{lem:ProjectionLemma} with $\alpha = \frac{1}{2}$ to $Q(z', \frac{d_n}{3})$ and $Q(z',3d_n)$ which yields that for all $n \geq \const{N}{4}:= \const{N}{1}(1/2)$ (where $\const{N}{1}$ is chosen accordinly to Lemma \ref{lem:ProjectionLemma})
\begin{align*}
B^u_{z,n}\left(\pi^n_{E_n(\omega,z)}\left(\hat z'\right), \frac{d_n}{6}\right) &\subset \pi^n_{E_n(\omega,z)}\left(\hat Q\left(z', \frac{d_n}{3}\right)\right) \\
&\subset \pi^n_{E_n(\omega,z)}\left(\hat Q\left(z', 3d_n\right)\right) \subset B^u_{z,n}\left(\pi^n_{E_n(\omega,z)}\left(\hat z'\right), \frac{9}{2}d_n\right).
\end{align*}
Thus
\begin{align*}
K(n,j) \leq \frac{\vol\left(B^{d-k}(\frac{9}{2}d_n)\right)}{\vol\left(B^{d-k}(\frac{d_n}{6})\right)} = 27^{d-k} =: K,
\end{align*}
where $B^{d-k}(r)$ denotes the $(d-k)$-dimensional Euclidean ball of radius $r$ and $\vol\left(B^{d-k}(r)\right)$ its volume.
\end{proof}

\subsection{Comparison of Volumes}

Let us consider two submanifolds $W^1$ and $W^2$ transversal to the family $\collLSM{x}{\const{q}{3}_C}$ satisfying $\norm{W^i} \leq \eps_C$, where $\eps_C$ was defined in Section \ref{sec:ChoiceOfDelta}.

Let $z \in  \Delta^l_\omega \cap \pLball[\Delta,\omega]{x}{\const{q}{3}_C/2} $ then by transversality $W_{loc}(\omega, z) \cap W^i \cap \pLball[\Delta,\omega]{x}{\const{q}{3}_C} \neq \emptyset$. Let us denote the intersection of $W^1$ and $W^2$ with the local stable manifold $W_{loc}(\omega,z)$ by $y^1 = \exp_z (\xi_0^1,\eta_0^1)$ and $y^2 = \exp_z  (\xi_0^2,\eta_0^2)$ respectively, i.e. $y^i = W_{loc}(\omega,z) \cap W^i$, where as usually $\xi_0^i \in E_0(\omega, z)$ and $\eta_0^i \in H_0(\omega, z)$, $ i= 1,2$. Clearly we have $ y^{i} \in \pLball[\Delta,\omega]{x}{\const{q}{3}_C}$ for $i = 1,2$. Let us now fix two numbers $\delta_{i,0}$ for $i = 1,2$ such that
\begin{align*}
 0 < \delta_{i,0} < \frac{1}{2} \min\left(q_C(z,W^1),q_C(z,W^2)\right) =: q_C(z,W^1,W^2).
\end{align*}

Now we can apply to the manifolds $W^1$ and $W^2$ the construction described in Section \ref{subsec:covering} and obtain for $i = 1,2$ and $n \geq 0$ the maps $\psi^i_n$ (see Lemma \ref{thm:ExsitenceOfPsi}) and the manifolds
\begin{align*}
\tilde W^i_n := \tilde W^i_n(z,y^i,\delta'_{i,n}) = \exp_{f^n_\omega z} \left\{\left(\psi^i_n(\eta),\eta\right):\eta \in \uLball[\eta^i_n]{\delta'_{i,n}}{z,n}\right\},
\end{align*}
where $\delta'_{i,n} = \delta'_{i,0}e^{(a + 11\eps)n}$ and $\eta^i_n = \pi_{E_n(\omega,z)}\circ F_{(\omega,z),n-1}\circ \dots \circ F_{(\omega,z),0}(\xi^i_0,\eta^i_0) \in H_n(\omega,z)$. Here $\pi_{E_n(\omega,n)}$ again denotes the projection of $T_{f^n_\omega z}\R^d$ to $H_n(\omega,z)$ parallel to $E_n(\omega,x)$. Let us further define for $i = 1,2$
\begin{align*}
 \hat W^i_n := \exp^{-1}_{f^n_\omega z} \left(\tilde W^i_n(z,y^i,\delta'_{i,n})\right)
\end{align*}
and for $z' \in \left(f^n_\omega \right)^{-1}\left(\tilde W^i_n\right)$ and $j = 0, 1, \dots, n$ let $\hat z'_j = \exp^{-1}_{f^j_\omega z}(f^j_\omega z')$ and $T^i_j(z') := T_{\hat  z'_j}\hat W^i_j$. As in the proof of Theorem \ref{thm:localStableManifold} let
\begin{align*}
F^n_0(\omega,z) = F_{(\omega,z),n} \circ \dots \circ F_{(\omega,z),0}.
\end{align*}
We will denote its inverse by $F^{-n}_0(\omega,z)$. Let $E$ and $E'$ be two real vector spaces of the same finite dimension, equipped with the scalar products $\langle \cdot, \cdot\rangle_E$ and $\langle \cdot, \cdot\rangle_{E'}$ respectively. If $E_1 \subset E$ is a linear subspace of $E$ and $B: E \to E'$ a linear mapping, then we can define the determinante of $B|_{E_1}$ to be
\begin{align*}
\abs{\determinante\left(B|_{E_1}\right)} := \frac{\vol_{E'_1}(B(U))}{\vol_{E_1}(U)},
\end{align*}
where $U$ is an arbitrary open and bounded subset of $E_1$ and $E'_1$ is a arbitrary linear subspace of $E'$ of the same dimension as $E_1$ with $B(U) \subset E'_1$ (see \citep[Section II.3]{Katok86}). Then we have the following lemma on the comparsion of the determinants of the pullbacks in the direction tangent to the transversal manifolds. This is basically \citep[Lemma II.9.2]{Katok86}.

\begin{lemma} \label{lem:ComparionsOfDeterminantes}
 There exists a positive constant $\const{C}{2}$ such that for any number $n \in \N$ and every $z^1 \in \left(f^n_\omega \right)^{-1}\left(\tilde W^1_n\right)$, $z^2 \in \left(f^n_\omega \right)^{-1}\left(\tilde W^2_n\right)$ we have
\begin{align*}
 \abs{\frac{\abs{\determinante\left(D_{\hat z^1_n}F_0^{-n}(\omega,z)\big|_{T^1_n(z^1)}\right)}}{\abs{\determinante\left(D_{\hat z^2_n}F_0^{-n}(\omega,z)\big|_{T^2_n(z^2)}\right)}} - 1} \leq \const{C}{2}C.
\end{align*}
\end{lemma}

\begin{proof}
This is basically \citep[Lemma II.9.2]{Katok86}, but we will state the proof here, since some estimates differ from the proof there.

As before let us denote by $y^1$ and $y^2$ the intersection of the transversal manifolds $W^1$ and $W^2$ respectively with the local stable manifold $W_{loc}(\omega,z)$. Since
\begin{align*}
&\frac{\abs{\determinante\left(D_{\hat z^1_n}F_0^{-n}(\omega,z)\big|_{T^1_n(z^1)}\right)}}{\abs{\determinante\left(D_{\hat z^2_n}F_0^{-n}(\omega,z)\big|_{T^2_n(z^2)}\right)}} =
\frac{\abs{\determinante\left(D_{\hat z^1_n}F_0^{-n}(\omega,z)\big|_{T^1_n(z^1)}\right)}}{\abs{\determinante\left(D_{\hat y^1_n}F_0^{-n}(\omega,z)\big|_{T^1_n(y^1)}\right)}} 
\cdot \frac{\abs{\determinante\left(D_{\hat y^1_n}F_0^{-n}(\omega,z)\big|_{T^1_n(y^1)}\right)}}{\abs{\determinante\left(D_{\hat y^2_n}F_0^{-n}(\omega,z)\big|_{T^2_n(y^2)}\right)}} \\
&\hspace{37ex}\cdot\frac{\abs{\determinante\left(D_{\hat y^2_n}F_0^{-n}(\omega,z)\big|_{T^2_n(y^2)}\right)}}{\abs{\determinante\left(D_{\hat z^2_n}F_0^{-n}(\omega,z)\big|_{T^2_n(z^2)}\right)}}
\end{align*}
the problem can be reduced to estimate the quotient in the following two cases:
\begin{enumerate}
\item the transversal manifolds $W^1$ and $W^2$ coincide, i.e. first and third multiplier
\item $z^1,z^2 \in W_{loc}(\omega,z)$, i.e. the second multiplier with $y^1 = z^1$ and $y^2 =z^2$.
\end{enumerate}
Because of the general inequality for $a,b,c > 0$
\begin{align*}
\abs{abc - 1} \leq \abs{a -1}bc + \abs{b-1}c + \abs{c-1}
\end{align*}
the assertion follows, if we can bound each quotient separately.

{\it Case i).} Without loss of generality let us assume that $z^1, z^2 \in W^1$. The same proof is true if $z^1, z^2 \in W^2$. By the chain rule we have
\begin{align} \label{eq:tobeestimated}
L_n(z^1,z^2) &:= \frac{\abs{\determinante\left(D_{\hat z^1_n}F_0^{-n}(\omega,z)\big|_{T^1_n(z^1)}\right)}}{\abs{\determinante\left(D_{\hat z^2_n}F_0^{-n}(\omega,z)\big|_{T^1_n(z^2)}\right)}}
=\prod_{j=1}^n \frac{\abs{\determinante\left(D_{\hat z^1_j}F_{(\omega,z),j-1}^{-1}\big|_{T^1_j(z^1)}\right)}}{\abs{\determinante\left(D_{\hat z^2_j}F_{(\omega,z),j-1}^{-1}\big|_{T^1_j(z^2)}\right)}}\notag\\
&\leq \prod_{j=1}^n \left(1+ \frac{\abs{\abs{\determinante\left(D_{\hat z^1_j}F_{(\omega,z),j-1}^{-1}\big|_{T^1_j(z^1)}\right)} - \abs{\determinante\left(D_{\hat z^2_j}F_{(\omega,z),j-1}^{-1}\big|_{T^1_j(z^2)}\right)}}}{\abs{\determinante\left(D_{\hat z^2_j}F_{(\omega,z),j-1}^{-1}\big|_{T^1_j(z^2)}\right)}} \right)
\end{align}
We will estimate the numerator and the enumerator in the last expression separately. By definition we have
\begin{align*}
\hat W^1_j := \left\{ \left(\psi^1_j(\eta),\eta\right) : \eta \in \uLball[\eta_j]{\delta'_{1,j}}{z,j} \right\} \subset T_{f^j_\omega z}\R^d.
\end{align*}
Because of $z^i \in \left(f^n_\omega\right)^{-1}(\tilde W^1_n)$ and $F^{-1}_{(\omega,z),l}(\hat W^1_{l+1}) \subset \hat W^1_l$, $l \in \N$ and $i = 1,2$, we get for $0 \leq j \leq n$ that for $i = 1,2$
\begin{align*}
\hat z_j^i = F^j_0(\omega,z)\left(\exp_z^{-1}(z^i)\right) \in \hat W^1_j.
\end{align*}
By Lemma \ref{lem:BasicLemmaOnDeterminante} there exists a constant $\const{C}{2,1} = \const{C}{2,1}(k) > 0$ such that
\begin{align*}
&\abs{\abs{\determinante\left(D_{\hat z^1_j}F_{(\omega,z),j-1}^{-1}\big|_{T^1_j(z^1)}\right)} - \abs{\determinante\left(D_{\hat z^2_j}F_{(\omega,z),j-1}^{-1}\big|_{T^1_j(z^2)}\right)}} \\
&\hspace{15ex}\leq \const{C}{2,1} \sup_{\hat z' \in \hat W^1_j}\abs{D_{\hat z'}F_{(\omega,z),j-1}^{-1}}^{d-k} \cdot\\
&\hspace{25ex}\left(\abs{D_{\hat z^1_j}F_{(\omega,z),j-1}^{-1} - D_{\hat z^2_j}F_{(\omega,z),j-1}^{-1}} + \Gamma_{\abs{\cdot}}\left(T^1_j(z^1),T^1_j(z^2)\right)\right),
\end{align*}
where $\Gamma_{\abs{\cdot}}$ denotes the aparture between to linear spaces with respect to the Euclidean norm, i.e. for two such spaces $E$ and $E'$ the aparture is defined by
\begin{align*}
\Gamma_{\abs{\cdot}}(E,E') := \sup_{\substack{e \in E\\\abs{e} = 1}} \inf_{e' \in E'} \abs{e-e'}.
\end{align*}
Let us first observe that by Lemma \ref{lem:EstimatesOnLnorm}, the properties of $\psi^1_j$ (see Theorem \ref{thm:ExsitenceOfPsi}) and \rref{eq:EstiamteOnStableManifold} there exists some constant $\const{C}{2,2}$ such that for $1 \leq j \leq n$
\begin{align} \label{eq:EstimateOnDiam}
\sup_{\hat z' \in \hat W^1_j} \abs{\hat z'} &\leq 2 \sup_{\hat z' \in \hat W^1_j} \LnormAt[j]{z}{\hat z'} \notag\\
&\leq 2 \sup_{\eta \in \uLball[\eta_j]{\delta'_{1,j}}{z,j}} \LnormAt[j]{z}{\left(\psi^1_j(\eta),\eta\right)}\notag\\
&\leq 2 \max \left\{\left(\frac{1}{4} + C\right)\const{q}{3}_Ce^{(a+7\eps)j}; \delta_{1,j} + \LnormAt[j]{z}{\eta_j} \right\}\notag\\
&\leq 2 \max \left\{\left(\frac{1}{4} + C\right)\const{q}{3}_Ce^{(a+7\eps)j}; \delta_{1,0}e^{(a+11\eps)j} + \const{q}{3}_Ce^{(a+6\eps)j} \right\}\notag\\
&\leq \const{C}{2,2} e^{(a+11\eps)j}. 
\end{align}
Then we have by Lemma \ref{lem:ExistenceOfr}, Lemma \ref{lem:DerivativeEstimate} and \rref{eq:EstimateOnDiam} for $1 \leq j \leq n$
\begin{align} \label{eq:estimateOfTheInverseDerivative}
\sup_{\hat z' \in \hat W^1_j}\abs{D_{\hat z'}F_{(\omega,z),j-1}^{-1}} &\leq \sup_{\hat z' \in \hat W^1_j}\abs{D_{\hat z'}F_{(\omega,z),j-1}^{-1} - D_0F_{(\omega,z),j-1}^{-1}} + \abs{D_0F_{(\omega,z),j-1}^{-1}}\notag\\
&\leq r' e^{\eps (j-1)} \sup_{\hat z' \in \hat W^1_{j-1}} \abs{\hat z'}+ C' e^{\eps (j-1)}\notag\\
&\leq r'\const{C}{2,2} e^{\eps j}  e^{(a+11\eps)(j-1)} + C' e^{\eps (j-1)}\notag\\
&\leq \const{C}{2,3} e^{\eps (j-1)}.
\end{align}
By Lemma \ref{lem:ExistenceOfr} and Theorem \ref{thm:ExsitenceOfPsi} we get for $1 \leq j \leq n$
\begin{align} \label{eq:EstimateOfTheDiameter}
&\abs{D_{\hat z^1_j}F_{(\omega,z),j-1}^{-1} - D_{\hat z^2_j}F_{(\omega,z),j-1}^{-1}} \leq r' e^{\eps (j-1)} \abs{\hat z^1_{j-1} - \hat z^2_{j-1}} \notag\\
&\hspace{5ex}\leq 2  r' e^{\eps (j-1)} \LnormAt[(j-1)]{z}{\hat z^1_{j-1} - \hat z^2_{j-1}} \leq 2r'e^{\eps (j-1)} \sup_{\hat z', \hat z'' \in \hat W^1_{j-1}} \LnormAt[j-1]{z}{\hat z' - \hat z''}\notag\\
&\hspace{5ex}\leq 2r'e^{\eps (j-1)} \!\!\!\!\!\!\!\!\!\sup_{\eta, \eta' \in \uLball[\eta_{j-1}]{\delta'_{1,j-1}}{z,j-1}} \!\!\!\!\!\!\ \max\left\{\LnormAt[j-1]{z}{\eta - \eta'};\LnormAt[j-1]{z}{\psi^1_{j-1}(\eta) - \psi^1_{j-1}(\eta')}\right\}\notag\\
&\hspace{5ex}\leq 2r'e^{\eps (j-1)}  \max\left\{2\delta'_{1,j-1}; 2\delta'_{1,j-1} \sup_{\eta \in \uLball[\eta_{j-1}]{\delta'_{1,j-1}}{{j-1}}} \norm{D_\eta\psi^1_{j-1}}\right\}\notag\\
&\hspace{5ex}\leq 2r'e^{\eps (j-1)}  \max\left\{2\delta'_{1,j-1}; 2\delta'_{1,j-1} C e^{-7d\eps (j-1)}\right\}\notag\\
&\hspace{5ex}= 4r'\delta_{1,0} e^{(a+12\eps) (j-1)}.
\end{align}
The aparture between $T^1_j(z^1)$ and $T^1_j(z^2)$ can be bounded via Lemma \ref{lem:BasicLemmaOnAperture}
\begin{align*}
\Gamma_{\abs{\cdot}}\left(T^1_j(z^1),T^1_j(z^2)\right) &\leq 2A e^{2\eps j} \Gamma_{\LnormAt[j]{z}{\cdot}}\left(T^1_j(z^1),T^1_j(z^2)\right) \\
&\leq 8A e^{2\eps j} \sup_{\eta \in \uLball[\eta_j]{\delta'_{1,j}}{z,j}} \LnormAt[j]{z}{D_{\eta} \psi^1_j}\\
&\leq 8A e^{2\eps j} C e^{-7d\eps j},
\end{align*}
where $\Gamma_{\LnormAt[j]{z}{\cdot}}$ denotes the aparture with respect to the Lyapunov norm. So finally we get
\begin{align} \label{eq:finalEstimate1}
&\abs{\abs{\determinante\left(D_{\hat z^1_j}F_{(\omega,z),j-1}^{-1}\big|_{T^1_j(z^1)}\right)} - \abs{\determinante\left(D_{\hat z^2_j}F_{(\omega,z),j-1}^{-1}\big|_{T^1_j(z^2)}\right)}} \notag\\
&\hspace{15ex}\leq \const{C}{2,1} (\const{C}{2,3})^{d-k} e^{\eps j(d-k)} \left(4r'\delta_{1,0} e^{(a+12\eps)(j-1)}  + 8A C e^{-5d\eps j}\right)\notag\\
&\hspace{15ex}\leq \const{C}{2,4}(\delta_{1,0} + C)e^{-4d \eps j}
\end{align}
with a constant $\const{C}{2,4} > 0$. Finally we have to estimate the denominator in \rref{eq:tobeestimated}. We have analogously to \rref{eq:estimateOfTheInverseDerivative}
\begin{align} \label{eq:finalEstimate2}
\determinante\left(D_{\hat z^2_j}F_{(\omega,z),j-1}^{-1}\big|_{T^1_j(z^2)}\right)^{-1} &= \determinante\left(D_{\hat z^2_{j-1}}F_{(\omega,z),j-1}\big|_{T^1_{j-1}(z^2)}\right)\notag\\
&\leq \abs{D_{\hat z^2_{j-1}}F_{(\omega,z),j-1}}^{d-k}\notag\\
&\leq \sup_{\hat z' \in \hat W^i_{j-1}} \abs{D_{\hat z'}F_{(\omega,z),j-1}}^{d-k}\notag\\
&\leq (\const{C}{2,3})^{d-k} e^{(d-k)\eps (j-1)}.
\end{align}
Thus by combining \rref{eq:finalEstimate1} and \rref{eq:finalEstimate2} there exists a constant $\const{C}{2,5}$ such that
\begin{align*}
L_n(z^1,z^2) &\leq \prod_{j=1}^{n}\left(1 + \const{C}{2,4}(\delta_{1,0} + C)(\const{C}{2,3})^{d-k} e^{(d-k)\eps j}e^{-4d \eps j} \right)\\
&\leq \prod_{j=1}^{n}\left(1 + \const{C}{2,5}(\delta_{1,0} + C) e^{-3d\eps  j}\right).
\end{align*}
Let us observe that for any $\theta \in (0,1)$ and $a \in (0,2\const{C}{2,5})$ we have
\begin{align*}
\prod_{j=0}^{+\infty} \left(1 + a \theta^j\right) &\leq \exp\left(a \sum_{j=0}^{+\infty}\theta^j \right) = \exp\left(\frac{a}{1 - \theta}\right) \\
&\leq 1 + a\left(\frac{1}{1-\theta} + \exp\left(\frac{2\const{C}{2,5}}{1-\theta}\right)\right) =: 1 + \const{C}{2,6}a
\end{align*}
and thus with $\theta = e^{-3d\eps}$ and $a = \const{C}{2,5}(\delta_{1,0} + C)$ we get
\begin{align*}
L_n(z^1,z^2) \leq 1 + \const{C}{2,5}\const{C}{2,6}(\delta_{1,0} + C).
\end{align*}
Since $z^1$ and $z^2$ appear symmetrically in all our considerations we get
\begin{align*}
\frac{1}{L_n(z^1,z^2)} = L_n(z^2,z^1) \leq 1 + \const{C}{2,5}\const{C}{2,6}(\delta_{1,0} + C)
\end{align*}
and thus finally because of $1/(1+x) \geq 1-x$, $x \geq 0$ and $\delta_{1,0} \leq C$ we achieve
\begin{align*}
\abs{L_n(z^1,z^2) - 1} \leq \const{C}{2,5}\const{C}{2,6}(\delta_{1,0} + C) \leq 2\const{C}{2,5}\const{C}{2,6} C =: \const{C}{2} C.
\end{align*}

{\it Case ii).} The proof of this case follows the same line as in case i), except we have to find an analog bound in \rref{eq:EstimateOfTheDiameter} for for $\abs{y_{j-1}^1 - y_{j-1}^2}$. Let us note that $z, y^1, y^2 \in W_{loc}(\omega,z) \cap \pLball[\Delta,\omega]{x}{\const{q}{3}_C}$, then we have by Proposition \ref{prop:EstimateOnLnorms}
\begin{align*}
\abs{\hat y^1_j - \hat y^2_j} &\leq 2 \LnormAt[j]{z}{\hat y^1_j - \hat y^2_j} = 2 \LnormAt[j]{z}{F^j_0(\omega,z)(\exp_z^{-1}(y^1)) - F^j_0(\omega,z)(\exp_z^{-1}(y^2))}\\
&\leq 2r_0 e^{(a + 6 \eps)j} \left( \LnormAt[0]{z}{\exp_z^{-1}(y^1)} + \LnormAt[0]{z}{\exp_z^{-1}(y^2)}\right) \\
&= 2r_0 e^{(a + 6 \eps)j} \left( \LnormAt[0]{z}{\exp_z^{-1}(y^1) - \exp_z^{-1}(z)} + \LnormAt[0]{z}{\exp_z^{-1}(y^2)- \exp_z^{-1}(z)}\right) \\
&\leq 4r_0A e^{2\eps} e^{(a + 6 \eps)j} \Big( \LnormAt[0]{x}{\exp_x^{-1}(y^1) - \exp_x^{-1}(z)} \\
&\hspace{38ex}+ \LnormAt[0]{x}{\exp_x^{-1}(y^2)- \exp_x^{-1}(z)}\Big) \\
&\leq 16r_0A e^{2\eps} e^{(a + 6 \eps)j} \const{q}{3}_C.
\end{align*}
By definition of we have $\const{q}{3}_C \leq \const{q}{1} \leq C$ and thus we finally get analogously to \rref{eq:EstimateOfTheDiameter}
\begin{align*}
&\abs{D_{\hat y^1_j}F_{(\omega,z),j-1}^{-1} - D_{\hat y^2_j}F_{(\omega,z),j-1}^{-1}} \leq 16r'r_0 A e^{2\eps} e^{(a + 7 \eps)j} C \leq  \const{C}{2,7} C e^{(a+5\eps)j},
\end{align*}
which gives the analog bound for \rref{eq:EstimateOfTheDiameter} and thus finishes the proof.
\end{proof}

Let us denote by $\lambda^i_n$ the $(d-k)$-dimensional volume on $\tilde W^i_n(z,y^i,\delta'_{i,n})$ induced by the Euclidean norm. Then we have the following result (see \citep[Lemma II.9.3]{Katok86}) on the comparsion of volumes under the pull back of the diffeomorphisms, which is a direct result from Lemma \ref{lem:ComparionsOfDeterminantes}.

\begin{lemma} \label{lem:CompPullingBack}
There exists a constant $\const{C}{3}$ such that for any $\tau \in (0,1)$ and $n \geq 1$ if $A^i \subset \tilde W^i_n(z,y^i,\delta'_{i,n})$ for $i = 1,2$ with $\lambda^2_n(A^2) >0$ and
\begin{align*}
\abs{\frac{\lambda_n^1(A^1)}{\lambda^2_n(A^2)} - 1} < \tau
\end{align*}
then this implies
\begin{align*}
\abs{\frac{\lambda_0^1\left((f^n_\omega)^{-1}(A^1)\right)}{\lambda^2_0\left((f^n_\omega)^{-1}(A^2)\right)} - 1} \leq \const{C}{3}(\tau+C).
\end{align*}
\end{lemma}

\begin{proof}
Basically this is a direct result from Lemma \ref{lem:ComparionsOfDeterminantes}. Let us observe that for any $x \in \R^d$ the exponential function $\exp_x$ as a function on $\R^d$ is translation. Hence the Lebesgue measure $\hat \lambda^i_n$ on $\hat W^i_n(z,y^i,\delta'_{i,n}) = \exp_{f^n_\omega z}^{-1}(\tilde W^i_n(z,y^i,\delta'_{i,n}))$ coincides with $\lambda^i_n \circ \exp_{f^n_\omega z}$. So if we define for $i = 1,2$ the sets $\hat A^i := \exp_{f^n_\omega z}^{-1}(A^i)$ then we immediately get
\begin{align*}
\frac{\lambda_n^1(A^1)}{\lambda^2_n(A^2)} = \frac{\hat\lambda_n^1(\hat A^1)}{\hat\lambda^2_n(\hat A^2)}.
\end{align*}
For $i=1,2$ we have $\lambda^i_0\left((f^n_\omega)^{-1}(A^i)\right) = \hat \lambda^i_0\left(F^{-n}_0(\omega,z)(\hat A^i)\right)$. Thus by change of variables and the mean value theorem we get
\begin{align*}
\hat\lambda_0^i\left(F_0^{-n}(\omega,z)(\hat A^i)\right) &= \int_{\hat A^i} \abs{\determinante\left(D_\zeta F^{-n}_0(\omega,z)\big|_{T_\zeta\hat W^i_n}\right)} \dx \hat\lambda^i_n(\zeta)\\
&= \abs{\determinante\left(D_{\zeta_n^i}F^{-n}_0(\omega,z)\big|_{T_{\zeta_n^i}\hat W^i_n}\right)} \hat\lambda^i_n(\hat A^i)
\end{align*}
for some points $\zeta_n^i \in \hat A^i$, $i = 1,2$. By Lemma \ref{lem:ComparionsOfDeterminantes} we finally get
\begin{align*}
\abs{\frac{\lambda_0^1\left((f^n_\omega)^{-1}(A^1)\right)}{\lambda^2_0\left((f^n_\omega)^{-1}(A^2)\right)} - 1} = \abs{\frac{\hat\lambda_0^1\left(F_0^{-n}(\omega,z)(\hat A^1)\right)}{\hat\lambda_0^2\left(F_0^{-n}(\omega,z)(\hat A^2)\right)} - 1} \leq \const{C}{2} C(1 + \tau) + \tau \leq \const{C}{2}(C + \tau)
\end{align*}
with $\const{C}{3} := \const{C}{2}$.
\end{proof}

\subsection{Construction of the Final Covering} \label{sec:finalCovering}

Fix two submanifolds $W^1$ and $W^2$ transversal to $\collLSM{x}{\const{q}{3}_C}$. We will now apply the covering construction presented in the Section \ref{subsec:covering} to $W^1$. Let us fix $P \in W^1$, $0 < \beta < h_P$, $0<h<h_P-\beta$  and $0 < \delta_0 < \delta_{P,\beta,h}$. Now Lemma \ref{lem:Woverlines} implies that for $n \geq \const{N}{4}$, which will be as well fixed for the moment, there exists $\const{M}{1}_n$ and $\const{M}{2}_n$ and corresponding points $\{z_i\}_{1 \leq i \leq \const{M}{1}_n}$ and $\{\bar z_j\}_{1 \leq j \leq \const{M}{2}_n}$. For the moment let us fix some $j$, $1 \leq j \leq \const{M}{2}_n$. We will consider the submanifolds $\tilde W^1_n(z_i,y^1_i,\delta'_n)$, the sets $\overline W^1_n(1/2)$, $\overline W^1_n(1)$ and $Q(\bar z_j,d_n) \subset \tilde W^1_n(z_i,y^1_i,\delta'_n)$ without any further explanation (for details see Section \ref{subsec:covering}).

By Lemma \ref{lem:Woverlines} there exists $i = i(j)$, $1 \leq i \leq \const{M}{1}_n$, such that we have $Q(\bar z_j,d_n) \cap \tilde W^1_n\left(z_i, y^1_i, \frac{1}{2}\delta'_n\right) \neq \emptyset$ and $Q(\bar z_j,2d_n) \subset \tilde W^1_n\left(z_i, y^1_i,\delta'_n\right)$.

As before for $z' \in \tilde W^1_n(z_i,y^1_i,\delta'_n)$ let us set $\hat z' := \exp^{-1}_{f^n_\omega z_i}(z')$ and $\pi^n_{z_i} := \pi^n_{E_n(\omega,z_i)}$ denotes the projection of $T_{f^n_\omega z_i} \R^d$ onto $H_n(\omega,z_i)$ parallel to the subspace $E_n(\omega,z_i)$. 

Now we will start the final step before presenting the proof of the absolute continuity theorem, which will allow us to formulate and prove Lemma \ref{lem:CompOfWns}.

Fix $\theta \in (0,1/6)$ and let us consider the covering of the ball $\uballat[\pi^n_{z_i}(\hat{\bar z}_j)]{2(1-\theta)d_n}{n}{z_i} \subset H_{n}(\omega,z_i)$ by the closed $(d-k)$-dimensional cubes $\hat D_{j,m} \subset H_{n}(\omega,z_i)$, $1 \leq m \leq N_j$, of diameter $\theta d_n$ (with respect to the Euclidean norm) with disjoint interiors.

If $l$ is the length of an edge of the cube $\hat D_{j,m}$, then we will denote by $\left(\hat D_{j,m}\right)_{\bar l}$ the concentric cube with length of the edge $l+{\bar l}$. Let $0 < \alpha_0 < \frac{\theta d_0}{\sqrt{d-k}}$ and define $\alpha_n := \alpha_0 e^{(a + 9\eps)n}$ for $n \geq 0$. If we denote by $\vol$ the $(d-k)$-dimensional volume in $H_n(\omega,z_i)$ then we have
\begin{align} \label{eq:EstimateOnTheDVolumes}
\abs{\frac{\vol\left(\left(\hat D_{j,m}\right)_{\alpha_n}\right)}{\vol\left(\hat D_{j,m}\right)} - 1} \leq 2^{d-k} \sqrt{d-k} \frac{\alpha_0}{\theta d_0}.
\end{align}
By the choice of $\alpha_0$ we have
\begin{align*}
\hat D_{j,m} \subset \left(\hat D_{j,m}\right)_{\alpha_n} \subset \uballat[\pi^n_{z_i}(\hat{\bar z}_j)]{2d_n}{n}{z_i}.
\end{align*}
Because of $2Ad_0 \leq \delta_0 < \delta_{P,\beta}$, $\diam\left(\left(\hat D_{j,m}\right)_{\alpha_n}\right) \leq 2\theta d_n$ and $\bar z_j \in \tilde W^1_n(z_i, y_i, \frac{1}{2}\delta'_n)$ Lemma \ref{lem:ProjectionLemma} and Proposition \ref{prop:ComparsionWandQ} imply for $n \geq \max\left\{\const{N}{1}(1/3);\const{N}{4}\right\}$ that
\begin{align} \label{eq:hatDisFine}
\left(\hat D_{j,m}\right)_{\alpha_n} &\subset \uballat[\pi^n_{z_i}(\hat{\bar z}_j)]{2d_n}{n}{z_i} \subset \pi_{z_i}^n(\hat Q (\bar z_j, 3d_n)) \\
&\subset \pi_{z_i}^n(\hat W^1_n(z_i,y^1_i,\delta'_n))
= \uLball[\eta^1_{i,n}]{\delta'_n}{z_i,n},
\end{align}
where $\eta^1_{i,n} = \pi^n_{z_i}(F^n_0(\omega,z_i)y^1_i)$. 
Thus for $n \geq \max\left\{\const{N}{1}(1/3);\const{N}{4}\right\}$ the function $\psi^2_{z_i,n}$ is well defined on $\left(\hat D_{j,m}\right)_{\alpha_n}$ and analogously one can see that $\psi^1_{z_i,n}$ is well defined on $\hat D_{j,m}$, where $\psi^k_{z_i,n}$, $k = 1,2$, are the functions which are constructed in Theorem \ref{thm:ExsitenceOfPsi} for $W^k$, $k = 1,2$ with respect to $z_i$. So let us finally define
\begin{align*}
D^1_{j,m} &:= \exp_{f^n_\omega z_i} \left\{ \left(\psi^1_{z_i,n}(\eta),\eta\right) : \eta \in \hat D_{j,m} \right\}\\
\bar D^2_{j,m} &:= \exp_{f^n_\omega z_i} \left\{ \left(\psi^2_{z_i,n}(\eta),\eta\right) : \eta \in \left(\hat D_{j,m}\right)_{\alpha_n} \right\}.
\end{align*}

Then we have the following important lemma, which basically states that the pullback of the set $D^1_{j,m}$ is mapped by the Poincar\'e map $P_{W^1,W^2}$ (defined in Section \ref{sec:TheoremAbsoluteContinuity}) into the pullback of the set $\bar D^2_{j,m}$. Later this will give us the possibility to compare the Lebesgue measures under the Poincar\'e map on $W^1$ with the one on $W^2$.

\begin{lemma} \label{lem:CompOfWns}
For every $\alpha_0 >0$ there exists $\const{N}{6} = \const{N}{6}(\alpha_0)\geq \max\left\{\const{N}{1}(1/3);\const{N}{4}\right\}$ such that for any $n \geq \const{N}{6}$, $1 \leq j \leq \const{M}{2}_n$ and $1 \leq m \leq N_j$ we have
\begin{align*}
P_{W^1,W^2}\left((f^n_\omega)^{-1}(D^1_{j,m}) \cap \localLSM{x}{\const{q}{3}_C}\right) \subset (f^n_\omega)^{-1}(\bar D^2_{j,m}).
\end{align*}
\end{lemma}

\begin{proof}
Let $n \geq \max\left\{\const{N}{1}(1/3);\const{N}{4}\right\}$ and $y^1 \in (f^n_\omega)^{-1}(D^1_{j,m}) \cap \localLSM{x}{\const{q}{3}_C}$. Then there exists $z' \in \Delta^l_\omega \cap \pLball[\Delta,\omega]{x}{\const{q}{3}_C/2}$ such that $y^1 \in W(\omega, z')$. Since $W^2$ is also transversal to $\collLSM{x}{\const{q}{3}_C}$ there exists a unique point $y^2 = W^2 \cap W(\omega,z') \cap \pLball[\Delta,\omega]{x}{\const{q}{3}_C}$. Thus we only need to check that for $n$ large $y^2 \in (f^n_\omega)^{-1}(\bar D^2_{j,m})$ or equivalent
\begin{align*}
\exp_{f^n_\omega z_i}^{-1} \left(f^n_\omega y^2\right) \in \exp_{f^n_\omega z_i}^{-1} \left(\bar D^2_{j,m} \right) =  \left\{ \left(\psi^2_{z_i,n}(\eta),\eta\right) : \eta \in \left(\hat D_{j,m}\right)_{\alpha_n} \right\}.
\end{align*}
If we denote $(\xi^{1}_0, \eta^{1}_0) := \exp_{z_i}^{-1}(y^1)$ and $(\xi^2_0,\eta^2_0) := \exp_{z_i}^{-1}(y^2)$ and
\begin{align*}
(\xi^k_n, \eta^k_n)&:= \exp_{f^n_\omega {z_i}}^{-1}(f^n_\omega y^k) = F_0^n(\omega,x)(\xi^k_0, \eta^k_0),
\end{align*}
for $k = 1,2$, then it suffices to prove that $\eta^2_n \in \left(\hat D_{j,m}\right)_{\alpha_n}$ for large $n$. By Lemma \ref{lem:EstimatesOnLnorm} and Proposition \ref{prop:EstimateOnLnorms} 
we have because of $z_i, z' \in \Delta^l_\omega$
\begin{align*}
\abs{\eta^1_n - \eta^2_n} &\leq 2 \LnormAt[n]{z_i}{\eta^1_n - \eta^2_n} \leq 2 \LnormAt[n]{z_i}{(\xi^1_n,\eta^1_n) - (\xi^2_n,\eta^2_n)}\\
&=2 \LnormAt[n]{z_i}{\exp_{f^n_\omega z_i}^{-1}(f^n_\omega y^1) - \exp_{f^n_\omega z_i}^{-1}(f^n_\omega y^2)}\\
&\leq 2A e^{2\eps n} \LnormAt[n]{z'}{\exp_{f^n_\omega z'}^{-1}(f^n_\omega y^1) - \exp_{f^n_\omega z'}^{-1}(f^n_\omega y^2)}.
\end{align*}
Let us denote $(\hat \xi^k_n ,\hat \eta^k_n) := \exp_{f^n_\omega z'}^{-1}(f^n_\omega y^k)$ where $\hat \xi^k_n \in E_n(\omega,z')$ and $\hat \eta^k_n \in H_n(\omega,z')$ for $k=1,2$. By the choice of $\const{q}{1}_C$ and $\const{q}{3}_C$ and since $z',y^1, y^2 \in \pLball[\Delta,\omega]{x}{\const{q}{3}_C}$ we have for $k = 1,2$
\begin{align*}
\LnormAt[0]{z'}{(\hat \xi^k_0 ,\hat \eta^k_0)} &= \LnormAt[0]{z'}{\exp_{z'}^{-1}(y^k)} = \LnormAt[0]{z'}{\exp_{z'}^{-1}(y^k) - \exp_{z'}^{-1}(z')}\\
&\leq A \LnormAt[0]{x}{\exp_{x}^{-1}(y^k) - \exp_{x}^{-1}(z')} \leq 2A\const{q}{3}_C \leq r_0.
\end{align*}
Thus because of $(\hat \xi^k_0 ,\hat \eta^k_0) = \exp^{-1}_{z'}(y^k) \in \exp^{-1}_{z'}\left(W_{loc}(\omega,z')\right)$ for $k=1,2$ we get with \rref{eq:EstiamteOnStableManifold}
\begin{align*}
\abs{\bar \eta^1_n - \bar \eta^2_n} &\leq 2A e^{2\eps n}\left(\LnormAt[n]{z'}{(\hat \xi^1_n ,\hat \eta^1_n)} + \LnormAt[n]{z'}{(\hat \xi^2_n ,\hat \eta^2_n)} \right) \leq 4A e^{2\eps n}r_0 e^{(a+6\eps)n} \\
&= \left(4Ar_0 e^{-\eps n}\right)e^{(a+9\eps)n}.
\end{align*}
By choosing $\const{N}{6} = \const{N}{6}(\alpha)$ so large such that $4Ar_0 e^{-\eps \const{N}{6}} \leq \frac{\alpha_0}{2}$ we get that for $n \geq \const{N}{6}$
\begin{align*}
\abs{\bar \eta^1_n - \bar \eta^2_n} \leq \frac{\alpha_n}{2}.
\end{align*}
This implies since $\bar \eta^1_n \in \hat D_{j,m}$ that $\bar \eta^2_n \in \left(\hat D_{j,m}\right)_{\alpha_n}$, which proves the lemma.
\end{proof}

Further we have the following lemma, which compares these sets with the set $Q(\bar z_j,r)$ for $d_n \leq r \leq 2d_n$. It is a stronger result than in \citep[Proposition II.10.1]{Katok86} because of the second inclusion in the proposition, which is an important ingredient for the proof of Lemma \ref{lem:CompSumUnion}.

\begin{proposition} \label{prop:QD}
Let $\theta \in (0, \frac{1}{6})$. For all $n \geq \max\left\{\const{N}{1}(\theta/2);\const{N}{4}\right\}$ and all $1 \leq j \leq \const{M}{2}_n$ one has
\begin{align*}
Q(\bar z_j, d_n) \subset Q(\bar z_j, 2(1-2\theta)d_n) \subset \bigcup_{m=1}^{N_j} D^1_{j,m} \subset Q(\bar z_j, 2d_n).
\end{align*}
\end{proposition}

\begin{proof}
The idea is basically taken from \citep[Proposition II.10.1]{Katok86}. By the remark after Lemma \ref{lem:ProjectionLemma} we have $\const{N}{1}(\theta/2) \geq \const{N}{1}(\theta) \geq \const{N}{1/3}$. Let us recall that for $\bar z_j \in \tilde W^1_n\left(z_i,y^1_i,\delta'_{n}\right)$ we denote $\hat Q(\bar z_j,r) := \exp_{f^n_\omega z_i}^{-1}(Q(\bar z_j,r))$. If we are able to show
\begin{align} \label{eq:firstInclusion1}
\pi^n_{z_i}(\hat Q(\bar z_j, d_n)) \subset \pi^n_{z_i}(\hat Q(\bar z_j, 2(1-2\theta)d_n)) \subset \bigcup_{m=1}^{N_j} \hat D_{j,m} \subset \pi^n_{z_i}(\hat Q(\bar z_j, 2d_n))
\end{align}
then the application of $\psi^1_{z_i,n}$ to both sides yields the assertion. The first inclusion is obvious since $\theta \in (0,1/6)$. For the second inclusion in \rref{eq:firstInclusion1} let us apply Lemma \ref{lem:ProjectionLemma} with $\alpha = \frac{\theta}{1-2\theta} \geq \theta$, $F = E_n(\omega,z_i)$ and $q = 2(1-2\theta)d_n$ then we have that for $n \geq \max\left\{\const{N}{1}(\theta);\const{N}{4}\right\}$
\begin{align*}
\pi^n_{z_i}(\hat Q(\bar z_j,2(1-2\theta)d_n)) \subset \uballat[\pi^n_{z_i}(\hat{\bar z}_j)]{2(1-\theta)d_n}{n}{z_i}.
\end{align*}
Since $\{\hat D_{j,m}\}_{1 \leq m \leq N_j}$ form a covering of $\uballat[\pi^n_{z_i}(\hat{\bar z}_j)]{2(1-\theta)d_n}{n}{z_i}$ and $\theta \in (0,1/6)$ we get for $n \geq \max\left\{\const{N}{1}(\theta);\const{N}{4}\right\}$
\begin{align*}
\pi^n_{z_i}(\hat Q(\bar z_j,2(1-2\theta)d_n)) \subset \uballat[\pi^n_{z_i}(\hat{\bar z}_j)]{2(1-\theta)d_n}{n}{z_i} \subset \bigcup_{m = 1}^{N_j} \hat D_{j,m},
\end{align*}
which proves the second inclusion in \rref{eq:firstInclusion1}. For the third one observe that $\diam\left(\hat D_{j,m}\right) = \theta d_n$ and since $\hat D_{j,m} \cap \uballat[\pi^n_{z_i}(\hat{\bar z}_j)]{2(1-\theta)d_n}{n}{z_i} \neq \emptyset$ for any $1 \leq m \leq N_j$ we have
\begin{align*}
\bigcup_{m = 1}^{N_j} \hat D_{j,m} \subset \uballat[\pi^n_{z_i}(\hat{\bar z}_j)]{(2-\theta)d_n}{n}{z_i}.
\end{align*}
If we again apply Lemma \ref{lem:ProjectionLemma} to $\alpha = \frac{\theta}{2}$, $F = E_n(\omega,z_i)$ and $q = 2d_n$ then we get for any $n \geq  \max\left\{\const{N}{1}(\theta/2);\const{N}{4}\right\}$ 
\begin{align*}
\uballat[\pi^n_{z_i}(\hat{\bar z}_j)]{(2-\theta)d_n}{n}{z_i} \subset \pi^n_{z_i}(\hat Q(\bar z_j, 2d_n)),
\end{align*}
which gives the third inclusion in \rref{eq:firstInclusion1}.
\end{proof}

By Proposition \ref{prop:QD} and Lemma \ref{lem:Woverlines} we immediately get
\begin{align} \label{eq:PropertiesOfDis}
\overline W^1_n(1/2) \subset \bigcup_{j=1}^{\const{M}{2}} \bigcup_{m=1}^{N_j} D^1_{j,m} \subset \overline W^1_n(1).
\end{align}
Since $\Int(D^1_{j,m}) \cap \Int(D^1_{j,m'}) = \emptyset$ for $m \neq m'$, it follows from Lemma \ref{lem:Woverlines} that there exists some number $L' >0$ such that for every $n \geq \max\left\{\const{N}{1}(\theta/2);\const{N}{4}\right\}$ the covering of $\overline W^1_n(1/2)$ by the sets $\left\{D^1_{j,m}\right\}_{\substack{1 \leq j\leq \const{M}{2}\\1\leq m\leq N_j}}$ is of multiplicity at most $L'$. We will denote this covering by $\mathcal{A}$. Let us remark that $L'$ is the number $L$, which originally comes from Lemma \ref{lem:Woverlines}, and additionally the multiplicity of the covering $\{D^1_{j,m}\}_{1 \leq m \leq N_j}$. Since in following lemma we are interested in the comparision of the sum of the Lebesgue measures with the Lebesgue measure of the union the second multiplicity is neglectable, since its Lebesgue measure is $0$.

We will now choose a subcover of $\mathcal{A}$ which has multiplicity one, except on a set of very small measure. To obtain this we proceed consecutively from the ball $Q(\bar z_j,2d_n)$ to the ball $Q(\bar z_{j+1},2d_n)$ for $j = 1, 2,\dots, \const{M}{2}-1$: in the $(j+1)\textsuperscript{th}$ step we eleminate all sets $D^1_{j+1,m}$ with $D^1_{j+1,m} \subset \bigcup_{k = 1}^j \bigcup_{m=1}^{N_k} D^1_{k,m}$ or $D^1_{j+1,m} \subset Q(\bar z_{j+1},2(1-2\theta)d_n)^c$. Let $\left\{D^1_i\right\}_{1\leq i\leq N}$ be the covering of $\overline W^1_n(1/2)$ formed by all remaining elements of $\mathcal{A}$. Then we have the following lemma, which is \citep[Lemma II.10.2]{Katok86}.

\begin{lemma} \label{lem:CompSumUnion}
There exists a constant $\const{C}{4}$ such that for every $0 < \theta < \min\left\{\frac{1}{18};\frac{1}{3\const{C}{1}}\right\}$ there exists $\const{N}{7} = \const{N}{7}(\theta)\geq \max\left\{\const{N}{1}(\theta/2);\const{N}{4}\right\}$ such that for every $n \geq \const{N}{7}$ we have
\begin{align*}
\abs{\frac{\sum_{i=1}^N\lambda^1_0\left((f^n_\omega)^{-1}(D^1_i)\right)}{\lambda^1_0\left((f^n_\omega)^{-1}(\bigcup_{i=1}^N D^1_i)\right)} -1 } \leq \const{C}{4}(\theta + C).
\end{align*}
\end{lemma}

\begin{proof}
This is basically \citep[Lemma II.10.1]{Katok86}, but varies at some point, inparticular the definition of {\it good} and {\it bad} sets.

Let us consider $n \geq \max\left\{\const{N}{1}(\theta/2);\const{N}{4}\right\}$. Our first aim is to divide the set $\{1, \dots, N\}$ into a {\it bad} set $B$ and a {\it good} one $G$, in the sense that for $i \in G$ we have $\Int(D_i^1 \cap D^1_{i'}) = \emptyset$ for all $i' \neq i$. By the properties of the function $\psi^1_{z_i,n}$ (cf. Theorem \ref{thm:ExsitenceOfPsi}) let us first observe that $\diam(D^1_i) \leq 2\theta d_n$. The consecutive construction of the covering $\{D^1_i\}_{1 \leq i \leq N}$ and the second inclusion of Proposition \ref{prop:QD} imply that non-empty intersection of the interiors only occurs around the boundary of the sets $Q(\bar z_j, 2(1- 2\theta)d_n)$. Let us define
\begin{align*}
\begin{cases}
i \in B &\text{if there exists $j$ such that } D^1_i \cap Q(\bar z_j, 2(1- 2\theta)d_n)^c \cap Q(\bar z_j, 2(1- \theta)d_n) \neq \emptyset\\
i \in G &\text{otherwise}.
\end{cases}
\end{align*}
Then $i \in G$ satisfies $\Int(D_i^1 \cap D^1_{i'}) = \emptyset$ for all $i' \neq i$. Because of $\diam(D^1_i) \leq 2\theta d_n$ we get
\begin{align} \label{eq:firstEstimate1}
\bigcup_{i\in B} D^1_i \subset \bigcup_{j = 1}^{\const{M}{2}} \left\{z' \in Q(\bar z_j, 2d_n) : \tilde d( z', \partial Q(\bar z_j, 2d_n)) \leq 6\theta d_n \right\} = \bigcup_{j = 1}^{\const{M}{2}} A(\bar z_j, 3\theta)
\end{align}
where $\tilde d$ is the induced metric on $f^n_\omega(W^1)$ by the Euclidean metric and $A(z_j,3\theta)$ is defined before Lemma \ref{lem:EstimateOnTheVolume}. As mentioned above the multiplicity of the covering $\{D^1_i\}_{1 \leq i \leq N}$ does not exceed $L'$ we have
\begin{align*}
\sum_{i=1}^N \lambda^1_0\left((f^n_\omega)^{-1}(D^1_i)\right) &= \sum_{i\in G} \lambda^1_0\left((f^n_\omega)^{-1}(D^1_i)\right) + \sum_{i\in B} \lambda^1_0\left((f^n_\omega)^{-1}(D^1_i)\right) \\
&\leq \sum_{i\in G} \lambda^1_0\left((f^n_\omega)^{-1}(D^1_i)\right) + L'  \lambda^1_0\left(\bigcup_{i\in B}(f^n_\omega)^{-1}(D^1_i)\right) \\
&= \lambda^1_0\left(\bigcup_{i\in G}(f^n_\omega)^{-1}(D^1_i)\right) + L'  \lambda^1_0\left(\bigcup_{i\in B}(f^n_\omega)^{-1}(D^1_i)\right) \\
&\leq \lambda^1_0\left(\bigcup_{i=1}^N(f^n_\omega)^{-1}(D^1_i)\right) + L'  \lambda^1_0\left(\bigcup_{i\in B}(f^n_\omega)^{-1}(D^1_i)\right).
\end{align*}
Hence we get
\begin{align} \label{eq:secondEstimate1}
1 \leq \frac{\sum_{i=1}^N \lambda^1_0\left((f^n_\omega)^{-1}(D^1_i)\right)}{\lambda^1_0\left(\bigcup_{i=1}^N(f^n_\omega)^{-1}(D^1_i)\right)} \leq 1 + L' \frac{\lambda^1_0\left(\bigcup_{i\in B}(f^n_\omega)^{-1}(D^1_i)\right)}{\lambda^1_0\left(\bigcup_{i=1}^N(f^n_\omega)^{-1}(D^1_i)\right)}
\end{align}
and it suffices to estimate the last term in \rref{eq:secondEstimate1}. Because of \rref{eq:firstEstimate1}, Proposition \ref{prop:QD} and the fact that the multiplicity of the covering $\{Q(\bar z_j, d_n)\}_j$ is bounded by $L$ we have
\begin{align} \label{eq:thirdEstimate1}
\frac{\lambda^1_0\left(\bigcup_{i\in B}(f^n_\omega)^{-1}(D^1_i)\right)}{\lambda^1_0\left(\bigcup_{i=1}^N(f^n_\omega)^{-1}(D^1_i)\right)} &\leq \frac{\lambda^1_0\left(\bigcup_{j = 1}^{\const{M}{2}} (f^n_\omega)^{-1}(A(\bar z_j, 3\theta))\right)}{\lambda^1_0\left(\bigcup_{j=1}^{\const{M}{2}}(f^n_\omega)^{-1}(Q(\bar z_j, d_n))\right)} \notag \\
&\leq L \frac{\sum_{j = 1}^{\const{M}{2}} \lambda^1_0\left((f^n_\omega)^{-1}(A(\bar z_j, 3\theta))\right)}{\sum_{j=1}^{\const{M}{2}}\lambda^1_0\left((f^n_\omega)^{-1}(Q(\bar z_j, d_n))\right)}.
\end{align}
If numbers $a_1,\dots,a_N, b_1, \dots, b_N > 0$ satisfy $\frac{a_i}{b_i} \leq h$ for all $i$, then clearly we have $\frac{\sum_i a_i}{\sum_i b_i} \leq h$. By this remark it suffices to estimate each fractional in \rref{eq:thirdEstimate1} on its own. So let us fix some $j$, $1 \leq j \leq \const{M}{2}$, and denote $A^1 := A(\bar z_j,3\theta) \cup Q(\bar z_j, d_n)$ and $A^2 := Q(\bar z_j, d_n)$. Choosing $\theta < \frac{1}{18}$ from Lemma \ref{lem:EstimateOnTheVolume} we obtain a constant $\const{C}{1}$ such that for every $n \geq \const{N}{7}(\theta) := \max\left\{\const{N}{3}(3\theta);\const{N}{1}(\theta/2);\const{N}{4}\right\}$ we have
\begin{align*}
1 \leq \frac{\lambda^1_n(A^1)}{\lambda^1_n(A^2)} = 1 + \frac{\lambda^1_n(A(\bar z_j,3\theta))}{\lambda^1_n(Q(\bar z_j, d_n))} \leq 1 + 3\const{C}{1}\theta,
\end{align*}
which yields
\begin{align*}
\abs{\frac{\lambda^1_n(A^1)}{\lambda^1_n(A^2)} - 1} \leq 3\const{C}{1}\theta.
\end{align*}
Thus by application of Lemma \ref{lem:CompPullingBack} we achieve a constant $\const{C}{3}$ such that for $\tau = 3\const{C}{1}\theta < 1$ we have for $n \geq \const{N}{7}(\theta)$
\begin{align*}
\abs{\frac{\lambda^1_0((f^n_\omega)^{-1}(A^1))}{\lambda^1_0((f^n_\omega)^{-1}(A^2))} - 1} \leq \const{C}{3}(3\const{C}{1}\theta + C).
\end{align*}
By definition of $A^1$ and $A^2$ this implies for $n \geq \const{N}{7}$
\begin{align*}
\frac{\lambda^1_0\left((f^n_\omega)^{-1}(A(\bar z_j, 3\theta))\right)}{\lambda^1_0\left((f^n_\omega)^{-1}(Q(\bar z_j, d_n))\right)} \leq \const{C}{3}(3\const{C}{1} \theta + C),
\end{align*}
which finally finishes the proof with $\const{C}{4} := 3L'L\const{C}{3}\const{C}{1}$.
\end{proof}

The next proposition is the last one before we will start to prove the absolute continuity theorem, we will state the proof for sake of completeness although it is basically \citep[Proposition II.10.2]{Katok86}.

\begin{proposition} \label{prop:CompD1DbarD2}
There exists a constant $\const{C}{5}$ such that for any $\theta \in (0,1)$ there exists $\const{N}{8}= \const{N}{8}(\theta)\geq \max\left\{\const{N}{1}(\theta/2);\const{N}{4}\right\}$ such that for any $0 < \alpha_0 < \frac{\theta d_0}{\sqrt{d-k}}$, $n \geq \const{N}{8}$ and $1 \leq i \leq N$ one has
\begin{align*}
\abs{\frac{\lambda^2_n(\bar D^2_i)}{\lambda^1_n(D^1_i)} -1} \leq \const{C}{5}\left(2\theta + \frac{\alpha_0}{\theta d_0}(1 + \theta)\right).
\end{align*}
\end{proposition}

\begin{proof}
Let us fix some $n \geq \max\left\{\const{N}{1}(\theta/2);\const{N}{4}\right\}$ and $1 \leq i \leq N$. Then there exists $i'$, $1\leq i' \leq \const{M}{1}$, and $j'$, $1 \leq j' \leq \const{M}{2}$ and $m$, $1\leq m \leq N_{j'}$ such that
\begin{align*}
D^1_i = D^1_{j',m} &= \exp_{f^n_\omega z_{i'}} \left(\left\{ (\psi^1_{z_{i'},n}(v),v) : v \in \hat D_{j',m} \right\}\right)\\
\bar D^2_i = \bar D^2_{j',m} &=\exp_{f^n_\omega z_{i'}} \left(\left\{ (\psi^2_{z_{i'},n}(v),v) : v \in \left(\hat D_{j',m}\right)_{\alpha_n} \right\}\right).
\end{align*}
Let us denote
\begin{align*}
\hat D^1_i := \exp^{-1}_{f^n_\omega z_{i'}}( D^1_i) \quad \text{ and } \quad \hat{\bar D}^2_i := \exp^{-1}_{f^n_\omega z_{i'}}(\bar D^2_i).
\end{align*}
Then we clearly have
\begin{align*}
\frac{\lambda^2_n(\bar D^2_i)}{\lambda^1_n(D^1_i)} = \frac{\lambda^2_n(\bar D^2_i)}{\hat\lambda_n^2(\hat{\bar D}^2_i)} \cdot \frac{\hat\lambda_n^2(\hat{\bar D}^2_i)}{\vol((\hat D_{j',m})_{\alpha_n})} \cdot \frac{\vol((\hat D_{j',m})_{\alpha_n})}{\vol(\hat D_{j',m})} \cdot \frac{\vol(\hat D_{j',m})}{\hat\lambda_n^1(\hat D^1_i)} \cdot \frac{\hat\lambda_n^1(\hat D^1_i)}{\lambda^1_n(D^1_i)},
\end{align*}
where $\hat \lambda^k_n$ denotes the induced Lebesgue measure on $\hat W^k_n(z_{i'},y^k_{i'},\delta'_n)$ for $k = 1,2$ and $\vol$ the $(d-k)$-dimensional volume on $H_n(\omega,z_{i'})$. Since the exponential function is a simple translation on $T_{f^n_\omega z_{i'}}\R^d$ we have
\begin{align*}
\frac{\lambda^2_n(\bar D^2_i)}{\hat\lambda_n^2(\hat{\bar D}^2_i)} = \frac{\hat\lambda_n^1(\hat D^1_i)}{\lambda^1_n(D^1_i)} = 1.
\end{align*}
For $n \geq \max\left\{\const{N}{1}(\theta/2);\const{N}{4}\right\}$ we have because of \rref{eq:hatDisFine} that $\left(\hat D_{j',m}\right)_{\alpha_n} \subset \uLball[\eta^k_{i',n}]{\delta'_n}{z_{i'},n}$, where as before $\eta^k_{i',n} = \pi^n_{z_{i'}}(F^n_0(\omega,z_{i'}) y^k_{i'})$ for $k = 1,2$ and thus because of Lemma \ref{lem:EstimatesOnLnorm} and Theorem \ref{thm:ExsitenceOfPsi}
\begin{align*}
\sup_{\eta \in \left(\hat D_{j',m}\right)_{\alpha_n}} \abs{D_\eta\psi^2_{z_{i'},n}} &\leq 2Ae^{2\eps n} \sup_{\eta \in \left(\hat D_{j',m}\right)_{\alpha_n}} \LnormAt[n]{z_{i'}}{D_\eta\psi^2_{z_{i'},n}}\\
&\leq 2Ae^{2\eps n} \sup_{\eta \in \uLball[\eta^2_{i',n}]{\delta'_n}{z_{i'},n}} \LnormAt[n]{z_{i'}}{D_\eta\psi^2_{z_{i'},n}}\\
&\leq 2Ae^{2\eps n} e^{-7 d \eps n}\\
&\leq 2A e^{-5 \eps n}.
\end{align*}
Choosing $\const{N}{8}(\theta) \geq \max\left\{\const{N}{1}(\theta/2);\const{N}{4}\right\}$ such that $2A e^{-5 \eps n} \leq \theta$ for all $n \geq \const{N}{8}$ then we can estimate the second term via Proposition \ref{prop:Appendix1} by
\begin{align*}
1 \leq \frac{\hat\lambda_n^2(\hat{\bar D}^2_i)}{\vol((\hat D_{j',m})_{\alpha_n})} \leq 1 + 2^{d-k}\theta.
\end{align*}
Analogously we we can estimate the forth term by
\begin{align*}
1 - 2^{d-k}\theta \leq \frac{\vol(\hat D_{j',m})}{\hat\lambda_n^1(\hat D^1_i)} \leq 1.
\end{align*}
The estimate on the third term is \rref{eq:EstimateOnTheDVolumes}. Alltogether this implies with $\abs{abc - 1} \leq \abs{a -1}bc + \abs{b-1}c + \abs{c-1}$ the desired, i.e. for $n \geq \const{N}{8}(\theta)$ we have
\begin{align*}
\frac{\lambda^2_n(\bar D^2_i)}{\lambda^1_n(D^1_i)} \leq \const{C}{5} \left(2\theta + \frac{\alpha_0}{\theta d_0}(1 + \theta)\right),
\end{align*}
where $\const{C}{5} := 2^{d-k}\sqrt{d-k}$.
\end{proof}

\section{Proof of the Absolute Continuity Theorem} \label{sec:FinalProofACT}

Now we are able to sate the main proof of the absolute continuity theorem. Let us repeat its formulation.

\begin{thmACT}
Let $\Delta^l$ be given as above. 
\begin{enumerate}
\item There exist numbers $0 < q_{\Delta^l} < \delta_{\Delta^l}/2$ and $\eps_{\Delta^l} > 0$ such that for every $(\omega,x) \in \Delta^l$ and $0 < q \leq q_{\Delta^l}$ the familiy $\collLSM{x}{q}$ is absolutely continuous.

\item For every $\tilde C \in (0,1)$ there exist numbers $0 < q_{\Delta^l}(\tilde C) < \delta_{\Delta^l}/2$ and $\eps_{\Delta^l}(\tilde C) > 0$ such that for each $(\omega,x) \in \Delta^l$ with $\lambda(\Delta_{\omega}^l) > 0$ and $x$ is a density point of $\Delta_{\omega}^l$ with respect to $\lambda$, and each two submanifolds $W^1$ and $W^2$ transversal to $\collLSM{x}{q_{\Delta^l}(\tilde C)}$ and satisfying $\norm{W^{i}} \leq \eps_{\Delta^l}(\tilde C)$, $i = 1,2$, the Poincar\'e map $P_{W^1,W^2}$ is absolutely continuous and the Jacobian $J(P_{W^1,W^2})$ satisfies the inequality
\begin{align*}
\abs{J(P_{W^1,W^2})(y) - 1} \leq \tilde C
\end{align*}
for $\lambda_{W^1}$-almost all $y \in W^1 \cap \localLSM{x}{q_{\Delta^l}(\tilde C)}$.
\end{enumerate}
\end{thmACT}

\begin{proof}
{\it Case i)}
Fix once and for all $(\omega,x) \in \Delta^l$ and some $C \in (0,1)$. Then set $q_{\Delta^l} := \const{q}{3}_C$ and $\eps_{\Delta^l} := \eps_C$, both defined in Section \ref{sec:ChoiceOfDelta}. Let us remark that if the family $\collLSM{x}{q_{\Delta^l}}$ is absolutely continuous then $\collLSM{x}{q}$ is absolutely continuous for any $0 < q \leq q_{\Delta^l}$.

For any $P \in W^1$ and samll $ 0 < h <h_P$ we denote as before by $Q(P,h)$ the closed ball in $W^1$ centered at $P$ of radius $h$. We will show that there exist constant $\const{C}{6}$ such that for any two submanifolds $W^1$ and $W^2$ transversal to $\collLSM{x}{q_{\Delta^l}}$ satisfying $\norm{W^i}\leq \eps_{\Delta^l}$ we have
\begin{align}\label{eq:proofdone}
 \lambda_{W^2}\left(\poincare{Q(P,h) \cap \localLSM{x}{q_{\Delta^l}}}\right) \leq (1+\const{C}{6}C) \lambda_{W^1}(Q(P,h)).
\end{align}
This implies that $\lambda_{W^2}\left(\poincare{\cdot \cap \localLSM{x}{q_{\Delta^l}}}\right) \ll \lambda_{W^1}(\cdot)$, which implies the assertion since $\mathcal{B}\left( W^1 \cap \Delta^l_\omega(x,q_{\Delta^l})\right) \subseteq \mathcal{B}\left( W^1\right)$.

Now fix $P \in  W^1 \cap \localLSM{x}{q_{\Delta^l}}$, $0 < \beta < h_P$ and $0 < h < h_P - \beta$. We will use the covering of the transversal manifolds presented in Section \ref{subsec:covering} and \ref{sec:finalCovering}. For the fixed parameters $P$, $\beta$, $h$ and the transversal manifolds there exists according to Lemma \ref{lem:FirstInclusion} some $\delta_{P,\beta,h} > 0$. Now let us fix $0 < \delta_0 < \delta_{P,\beta,h}$,  $0 < \theta < \min\left\{\frac{1}{18};\frac{1}{3\const{C}{1}}\right\}$ (where $\const{C}{1}$ is the one from Lemma \ref{lem:EstimateOnTheVolume}) and $0 < \alpha_0 < \frac{\theta d_0}{\sqrt{d-k}}$, where $d_0 = \frac{\delta_0}{12A}$ as in Section \ref{sec:ProjectionLemmas}. For $n\geq \const{N}{9}(\alpha_0,\theta) := \max\left\{\const{N}{6}(\alpha_0); \const{N}{7}(\theta); \const{N}{8}(\theta)\right\}$ we can apply the covering construction of the previous sections to obtain a covering $\left\{D_i^1\right\}_{1 \leq i\leq N}$ of $f^n_\omega\left(D(P,h)\right)$, where $D(P,h) := Q(P,h) \cap \localLSM{x}{q_{\Delta^l}}$ and sets $\left\{\bar D_i^2\right\}_{1 \leq i\leq N}$. These satisfy by Lemma \ref{lem:CompOfWns} for all $1 \leq i \leq N$ 
\begin{align*}
\poincare{\left(f^n_\omega\right)^{-1}\left(D^1_i\right) \cap \localLSM{x}{q_{\Delta^l}}} \subset \left(f^n_\omega\right)^{-1}\left(\bar D^2_i\right).
\end{align*}
Then since by Lemma \ref{lem:Woverlines} and \rref{eq:PropertiesOfDis} for $n \geq \const{N}{9}(\alpha_0,\theta)$
\begin{align*} 
 \poincare{D(p,h)} &= \poincare{\left(f^n_\omega\right)^{-1}f^n_\omega \left( D(p,h) \right) \cap \localLSM{x}{q_{\Delta^l}}} \notag\\
&\subseteq  \poincare{\left(f^n_\omega\right)^{-1} \left( \bigcup_{i=1}^N  D^1_i \right) \cap \localLSM{x}{q_{\Delta^l}}} \notag\\
&= \bigcup_{i=1}^N  \poincare{\left(f^n_\omega\right)^{-1} \left( D^1_i \right) \cap \localLSM{x}{q_{\Delta^l}}} \notag\\
&\subseteq \bigcup_{i=1}^N \left(f^n_\omega\right)^{-1}\left(\bar D^2_i\right),
\end{align*}
we get
\begin{align} \label{eq:FirstEstimate}
\lambda_{W^2}\left(\poincare{D(p,h)}\right) &\leq \lambda_{W^2}\left(\bigcup_{i=1}^N \left(f^n_\omega\right)^{-1}\left(\bar D^2_i\right)\right)\notag \\
&\leq \sum_{i=1}^N \lambda_{W^2}\left(\left(f^n_\omega\right)^{-1}\left(\bar D^2_i\right)\right).
\end{align}
Now let $\alpha_0 := \frac{\theta^2 d_0}{\sqrt{d-k}}$ and let $\theta < \min\left\{\frac{1}{18};\frac{1}{3\const{C}{1}};\frac{1}{4\const{C}{5}}\right\}$ then
\begin{align*} 
\const{C}{5}\left(2\theta + \frac{\alpha_0}{\theta d_0}(1 + \theta)\right) \leq 4\const{C}{5}\theta =: \tau < 1.
\end{align*}
The assumptions of Lemma \ref{lem:CompPullingBack} are satisfied because of Proposition \ref{prop:CompD1DbarD2}, such that we get for all $n \geq \const{N}{10}(\theta) := \const{N}{9}\left(\frac{\theta^2 d_0}{\sqrt{d-k}},\theta\right)$
\begin{align} \label{eq:PulledBack}
\lambda_{W^2}\left(\left(f^n_\omega\right)^{-1}\left(\bar D^2_i\right)\right) \leq \left(1 + \const{C}{3}\left(\tau + C\right)\right) \lambda_{W^1}\left(\left(f^n_\omega\right)^{-1}\left( D^1_i\right)\right).
\end{align}
Combining \rref{eq:FirstEstimate} and \rref{eq:PulledBack} and applying Lemma \ref{lem:CompSumUnion} we get for all $n \geq \const{N}{10}(\theta)$
\begin{align}\label{eq:FinalEstimate}
\lambda_{W^2}\left(\poincare{D(p,h)}\right) &\leq \left(1 + \const{C}{4}\left(\theta + C\right)\right) \left(1 + \const{C}{3}\left( \tau + C \right)\right) \lambda_{W^1}\left(\left(f^n_\omega\right)^{-1}\left( \bigcup_{i=1}^N D^1_i\right)\right)\notag\\
&\leq (1+\const{C}{6}(\theta+C)) \lambda_{W^1}\left(\left(f^n_\omega\right)^{-1}\left(\bigcup_{i=1}^N D^1_i\right)\right),
\end{align}
with $\const{C}{6}:=\const{C}{4} + 4\const{C}{3}\const{C}{5} + 2\const{C}{3}\const{C}{4}$. By the choice of the covering we get from Lemma \ref{lem:FirstInclusion} that
\begin{align*}
\bigcup_{i=1}^N D^1_i \subseteq \bigcup_{i=1}^{\const{M}{2}} \tilde W^1_n\left(z^1_i, y^1_i, \delta_n'\right) \subseteq f^n_\omega\left(Q(p,h+\beta)\right),
\end{align*}
which implies by \rref{eq:FinalEstimate} for $n \geq \const{N}{10}(\theta)$ 
\begin{align*}
\lambda_{W^2}\left(\poincare{D(p,h)}\right) &\leq (1+\const{C}{6}(\theta + C)) \lambda_{W^1}\left(Q(p,h+\beta)\right).
\end{align*}
Since $\beta >0$ and $\theta > 0$ can be chosen arbitrarily small, this finally implies \rref{eq:proofdone}.

\end{proof}

\begin{proof}[Proof of the Absolute Continuity Theorem ii)]

Now we will proof the second part of Theorem \ref{thm:ACT2}. Fix once and for all $(\omega,x) \in \Delta^l$ such that $\lambda(\Delta^l_\omega) > 0$ and $x \in \Delta^l_\omega$ is a density point of $\Delta^l_\omega$ with respect to the Lebesgue measure $\lambda$. For $C \in (0,1)$ let $\const{q}{3}_C$ and $\eps_C$ as in Section \ref{sec:ChoiceOfDelta}.

For each $\xi \in E_0(\omega,x)$ with $\LnormAt[0]{x}{\xi} < \const{q}{3}_C$ let us define the submanifold $W_\xi$ by the formula
\begin{align*}
W_\xi := \exp_x \left\{ (\xi, \eta): \eta \in H_0(\omega,x); \LnormAt[0]{x}{\eta} < \const{q}{3}_C \right\} \subset \pLball[\Delta,\omega]{x}{\const{q}{3}_C}) .
\end{align*}
Clearly each $W_\xi$ is a transversal submanifold to the familiy $\collLSM{x}{\const{q}{3}_C}$. Since $x$ is a density point of $\Delta^l_\omega$ we have $\lambda(\Delta^l_\omega \cap \pLball[\Delta,\omega]{x}{\const{q}{3}_C/2}) > 0$. Since by Fubini's theorem
\begin{align*}
0 < \lambda\left(\Delta^l_\omega \cap \pLball[\Delta,\omega]{x}{\const{q}{3}_C/2}\right) = \int_{\sLball{\const{q}{3}_C/2}{x,0}} \lambda_{W_\xi}\left(W_{\xi} \cap \Delta^l_\omega \right) \dx \lambda_{\sLball{\const{q}{3}_C/2}{x,0}}(\xi)
\end{align*}
and thus there exists $\xi \in \sLball{\const{q}{3}_C/2}{x,0}$ such that $\lambda_{W_\xi}\left(W_{\xi} \cap \Delta^l_\omega \right) > 0$. And because of $\localLSM{x}{\const{q}{3}_C} \supseteq \Delta^l_\omega \cap \pLball[\Delta,\omega]{x}{\const{q}{3}_C/2}$ we have $\lambda_{W_\xi}\left(W_{\xi} \cap \localLSM{x}{\const{q}{3}_C} \right) > 0$. Let $W^1$ and $W^2$ two transversal manifolds to $\collLSM{x}{\const{q}{3}_C}$ then consider the Poincar\'e maps $P_{ W^1,  W_\xi}$ and $P_{ W^{2},  W_\xi} = P_{ W_\xi,  W^2}^{-1}$. Clearly we have
\begin{align*}
P_{ W^1,  W^2} = P_{ W_\xi,  W^2} \circ P_{ W^1,  W_\xi}.
\end{align*}
Because these maps are absolutely continuous by i) of Theorem \ref{thm:ACT2}, we have for $i=1,2$
\begin{align*}
 \lambda_{W^i}\left( W^i \cap \localLSM{x}{\const{q}{3}_C} \right) > 0.
\end{align*}

The following construction is due to the following apllication to $P_{W^1,W^2}$ {\it and} its inverse $P_{W^1,W^2}^{-1} = P_{W^2,W^1}$. So let us consider the set $\mathcal{T}$ of all points $y \in  W^1 \cap \localLSM{x}{\const{q}{3}_C}$ such that $y$ is a density point of $ W^1 \cap \localLSM{x}{\const{q}{3}_C}$ with respect to $\lambda_{ W^1}$ and $\poincare{y}$ is a density point of $ W^2 \cap \localLSM{x}{\const{q}{3}_C}$ with respect to $\lambda_{ W^2}$. As $\lambda_{ W^1}$-almost all points of $ W^1 \cap \localLSM{x}{\const{q}{3}_C}$ are of density and as $P_{ W^1,  W^2}^{-1}$ is absolutely continuous, we have that $\lambda_{W^2}$-almost all points of $ W^2 \cap \localLSM{x}{\const{q}{3}_C}$ belong to $\poincare{\mathcal{T}}$.

Now let us take $y \in \mathcal{T}$. By the definition of a point of density for every $\kappa > 0 $ there exists $0 < h(\kappa) < h_y$ such that for every $0 < h < h(\kappa)$ one has
\begin{align*}
\lambda_{W^1}(Q(y,h)) \leq (1 + \kappa)\lambda_{W^1}(\mathcal{T} \cap Q(y,h)),
\end{align*}
where $Q(y,h)$ as before denotes the closed ball in $W^1$ with center $y$ and radius $h > 0$ with respect to the Euclidean metric. Since $\lambda_{ W^2}$-almost all points of $ W^2 \cap \localLSM{x}{\const{q}{3}_C}$ belong to $\poincare{\mathcal{T}}$ and because of \rref{eq:proofdone} we have for every $0 < h < h(\kappa)$
\begin{align*}
\lambda_{ W^2}\left(\poincare{\mathcal{T} \cap Q(y,h)}\right) &= \lambda_{ W^2}\left(\poincare{\localLSM{x}{\const{q}{3}_C} \cap Q(y,h)}\right) \\
&\leq (1 + \const{C}{6}C)\lambda_{ W^1}( Q(y,h))\\
&\leq (1 + \kappa)(1 + \const{C}{6}C)\lambda_{ W^1}(\mathcal{T} \cap Q(y,h)),
\end{align*}
i.e.
\begin{align} \label{eq:EstimateOnJac}
\frac{\lambda_{ W^2}\left(\poincare{\mathcal{T} \cap Q(y,h)}\right)}{\lambda_{ W^1}(\mathcal{T} \cap Q(y,h))} \leq (1 + \kappa)(1 + \const{C}{6}C)
\end{align}
Since $y$ is a density point the Lebesgue-Vitali theorem implies for $h \to 0$ that 
\begin{align*}
J(P_{ W^1,  W^2})(y) \leq (1 + \kappa)(1 + \const{C}{6}C),
\end{align*}
where $J(P_{ W^1,  W^2})$ denotes the Jacobian of the poincar\'e map, and since $\kappa > 0$ can be chosen arbitrarily samll we finally get
\begin{align*}
J(P_{ W^1,  W^2})(y) \leq 1 + \const{C}{6}C.
\end{align*}

As $y \in \mathcal{T}$ then $\poincare{y}$ is a density point of $ W^2 \cap \localLSM{x}{\const{q}{3}_C}$. Since in our cosideration and in particular in \rref{eq:EstimateOnJac} $P_{ W^1,  W^2}$ and $P_{ W^1,  W^2}^{-1}$ play completely symmetrical roles we get
\begin{align*}
J(P_{ W^1,  W^2}^{-1})(\poincare{y}) \leq 1 + \const{C}{6}C.
\end{align*}
Because of
\begin{align*}
J(P_{ W^1,  W^2})(y) = \frac{1}{J(P_{ W^1,  W^2}^{-1})(\poincare{y})}
\end{align*}
we have
\begin{align*}
J(P_{ W^1,  W^2})(y) \geq \frac{1}{1+\const{C}{6}C} \geq 1 - \const{C}{6}C.
\end{align*}
Choosing additionally $0 < C < \frac{1}{\const{C}{6}}$ we finally get
\begin{align*}
\abs{J(P_{ W^1,  W^2})(y) - 1} \leq \const{C}{6}C.
\end{align*}
Now let $\tilde C \in (0,1)$ as in the theorem then we define
\begin{align*}
q_{\Delta^l}(\tilde C) = \const{q}{3}_{\tilde C/\const{C}{6}} \quad \text{and} \quad 
\eps_{\Delta^l}(\tilde C) = \eps_{\tilde C/\const{C}{6}}
\end{align*}
and this finishes the proof of Theorem \ref{thm:ACT2} part ii).

\end{proof}
\appendix

\section{Appendix}

In the Appendix we will state some basic results from \citep{Katok86}. The first one gives an estimate of the volume of the graph a function.

\begin{proposition} \label{prop:Appendix1}
Let $p \in \N$, $U \subset \R^p$ be an open bounded set and $H$ some finite Hilbert space. Then for a $C^1$ mapping $g: U \to H$ with $\sup_{v \in U} \norm{D_vf} \leq a$ we have
\begin{align*}
 \vol_p(U) \leq m_p(\graph(f)) \leq (1 + a^2)^\frac{p}{2} \vol_p(U).
\end{align*}
Here $\vol_p$ denotes the $p$-dimensional Lebesgue measure and $m_p$ the $p$-dimensional Hausdorff measure in $\R^p \oplus H$, which coincides while restricted to a $p$-dimensional submanifold of $\R^p \oplus H$ since $H$ is a finite Hilbert space with the $p$-dimensional volume (Lebesgue measure) on this submanifold.
\end{proposition}

\begin{proof}
 This is \citep[Proposition II.3.2]{Katok86}
\end{proof}

Let $E$ and $E'$ be two real vector spaces of the same finite dimension, equipped with the scalar products $\langle \cdot, \cdot\rangle_E$ and $\langle \cdot, \cdot\rangle_{E'}$ respectively. If $E_1 \subset E$ is a linear subspace of $E$ and $A: E \to E'$ a linear mapping, then we can define the determinante of $A|_{E_1}$ to be
\begin{align*}
\abs{\determinante\left(A|_{E_1}\right)} := \frac{\vol_{E'_1}(A(U))}{\vol_{E_1}(U)},
\end{align*}
where $U$ is an arbitrary open and bounded subset of $E_1$ and $E'_1$ is a arbitrary linear subspace of $E'$ of the same dimension as $E_1$ with $A(U) \subset E'_1$ (see \citep[Section II.3]{Katok86}). Further for two linear subspaces $E_1, E_2 \subset E$ of the same dimension we define the aperture between $E_1$ and $E_2$ with respect to the norm $\abs{\cdot}_E$ to be
\begin{align*}
\Gamma_{\abs{\cdot}_E}(E_1,E_2) := \sup_{\substack{e_1 \in E_1 \\ \abs{e_1}_E = 1}} \inf_{e_2 \in E_2} \abs{e_1 - e_2}_E.
\end{align*}
Then we have the following lemma.

\begin{lemma} \label{lem:BasicLemmaOnDeterminante}
For every $p \in \N$ there exists a number $\const{C}{7} = \const{C}{7}(p) > 0$ such that for every two finite dimensional Hilbert spaces $H_1$ and $H_2$, for any $a \geq 1$, any two linear operators $A,B: H_1 \to H_2$ with $\abs{A}_{H_1} \leq a$, $\abs{B}_{H_1} \leq a$ and any two linear subspaces $E_1, E_2 \subset H_1$ of dimension $p$ we have
\begin{align*}
\abs{\abs{\determinante(A|_{E_1})} - \abs{\determinante(B|_{E_2})}} \leq \const{C}{7} a^p \left(\abs{A - B}_{H_1} + \Gamma_{\abs{\cdot}_{H_1}}(E_1,E_2)\right).
\end{align*}
\end{lemma}

\begin{proof}
This is \citep[Lemma II.3.2]{Katok86}.
\end{proof}

For a linear operator $A: H_1 \to H_2$ between two Hilbert spaces $H_1$ and $H_2$ let us denote the graph of $A$ by $\graph(A) :=\left\{(x,Ax) : x \in H_1 \right\} \subset H_1 \times H_2$. Then the aperture between two graphs can be bounded as follows.

\begin{lemma} \label{lem:BasicLemmaOnAperture}
Let $H_1$ and $H_2$ be two finite dimensional Hilbert spaces. For any two linear operators $A,B: H_1 \to H_2$ we have
\begin{align*}
\Gamma_{\abs{\cdot}_{H_1 \times H_2}}(\graph(A),\graph(B)) \leq 2 (\abs{A}_{H_1} + \abs{B}_{H_1}).
\end{align*}
\end{lemma}

\begin{proof}
This is \citep[Proposition II.3.4]{Katok86}.
\end{proof}

\section*{Acknowledgement}
The present research was supported by the International Research Training Group {\it Stochastic Models of Complex Processes} funded by the German Research Council (DFG). The author gratefully thanks Michael Scheutzow and Simon Wasserroth from TU Berlin for their support and several fruitful discussions.

\bibliography{bibliography}
\bibliographystyle{amsplain}

\end{document}